\documentclass[reqno,11pt]{amsart}
\usepackage{amsfonts}
\usepackage{a4wide}
\usepackage{color}
\usepackage{mathrsfs}
\usepackage{mathtools}
\usepackage{amsmath}
\usepackage{amssymb}
\usepackage{bbm}
\usepackage{bm}
\usepackage{esint}
\usepackage{nicefrac}
\usepackage{enumitem}
\usepackage{comment}
\newenvironment{conditions}[1]
  {\begin{enumerate}[
      label=(#1\arabic*),
      ref=(#1\arabic*),
      leftmargin=*,
      labelsep=0.6em,
      itemsep=0.4em
   ]}
  {\end{enumerate}}
\numberwithin{equation}{section}
\newcommand{\fm}{\mathfrak{m}}
\newcommand{\Ext}{\operatorname{Ext}}
\newcommand{\Tr}{\operatorname{Tr}}
\newcommand{\loc}{\operatorname{loc}}
\usepackage[colorlinks,citecolor=green,linkcolor=red]{hyperref}

\usepackage[utf8]{inputenc}
\input xy
\xyoption{all}
\newtheorem{theorem}{Theorem}[section]
\newtheorem{Ca}[theorem]{Corollary}
\newtheorem{Th}[theorem]{Theorem}
\newtheorem{Lm}[theorem]{Lemma}
\newtheorem{Prop}[theorem]{Proposition}
\newtheorem{Def}[theorem]{Definition}

\newtheorem{Remark}[theorem]{Remark}
\newtheorem{Problem}[theorem]{Problem}

\title{Traces of Besov and Lizorkin--Triebel spaces to thick subsets of $\mathbb{R}^n$}
\author{Aleksei Y. Chikalov}
\address{Steklov Mathematical Institute of Russian Academy of Sciences}
\email{chikalov.a@phystech.edu}
\begin{document}
\allowdisplaybreaks
\begin{abstract}
Let $d\in [0, n]$ and let $E\subset \mathbb{R}^n$ be a closed $d$-thick set. Given  $p\in [1, \infty]$, $q\in (0, \infty]$, and $s\in (\frac{n-d}{p}, 1)$, we provide an intrinsic description of the trace-space of the Besov space $B^s_{p, q}(\mathbb{R}^n)$ to $E$. If, in addition, $p<\infty$, we give an intrinsic description of the trace-space of the Lizorkin--Triebel space $F^s_{p, q}(\mathbb{R}^n)$ to $E$.
\end{abstract}
\maketitle
\tableofcontents
\section{Introduction}
\noindent
The trace problem -- that is, the problem of the sharp intrinsic descriptions of traces of function spaces to closed subsets of $\mathbb{R}^n$ -- is a classical and still largely open problem in analysis. Intrinsic characterizations are known for a number of geometrically regular classes of sets, most notably for Ahlfors--David regular sets. In the present paper, we pass to the substantially broader class of closed $d$-thick sets and give intrinsic descriptions of the traces of Besov and Lizorkin--Triebel spaces to such sets.
\par
In order to formulate the trace problem, we first recall several basic facts. There are several approaches to defining Besov and Lizorkin--Triebel spaces used in modern analysis. General background and further results can be found in \cite{soto_liz_tr, AKZ, kosk_Yang_Zhou, Sawano, soto_bes, trieb} and the references therein. For our purposes, it will be convenient to use the definition in terms of fractional Haj\l{}asz gradients (see, e.g., \cite{AKZ,kosk_Yang_Zhou}). More specifically, we work with the Haj\l{}asz--Besov spaces  $N^s_{p,q}(\mathbb R^n)$ and the Haj\l{}asz--Lizorkin--Triebel spaces $M^s_{p,q}(\mathbb R^n)$, where $s\in(0,1)$, $p, q\in (0,\infty]$; in the Lizorkin--Triebel case we additionally assume that $p<\infty$ (see Section~\ref{section.preliminaries} for details). We emphasize that, for the full range of parameters specified above, the corresponding oscillation-based definitions yield the same spaces (see \cite{AKZ}). Accordingly, throughout the paper we refer to these spaces simply as Besov and Lizorkin--Triebel spaces and use the usual notation $B^s_{p, q}(\mathbb{R}^n)$ and $F^s_{p, q}(\mathbb{R}^n)$. We write $A^s_{p,q}(\mathbb R^n)$, $A\in\{B,F\}$, for either of these spaces. Furthermore, if $p>\frac{n}{n+s}$ in the Besov case, and $p>\frac{n}{n+s}$, $q>\frac{n}{n+s}$ in the Lizorkin--Triebel case, then these spaces also coincide with the corresponding spaces defined by the Fourier-analytic approach (see \cite{kosk_Yang_Zhou}).
\par
In the range of parameters considered in the present paper, let \(p\in[1,\infty]\) and \(q\in(0,\infty]\), with \(p<\infty\) in the Lizorkin--Triebel case. It is well known that elements of $A^s_{p,q}(\mathbb R^n)$ admit natural pointwise representatives. If $s>\frac np$, then every $f\in A^s_{p,q}(\mathbb R^n)$ admits a continuous representative (see, e.g., \cite{trieb}); in particular, this representative is uniquely determined pointwise. If $s\le \frac np$, then $f$ admits a representative $\bar{f}$ which has the Lebesgue point property at every point outside a set $N_f\subset\mathbb R^n$ satisfying $\dim_H N_f\le n-ps$,
where $\dim_H$ denotes the Hausdorff dimension (see, e.g., \cite{saks}). We refer to such representatives as \emph{sharp representatives}. The interested reader can find a discussion of closely related questions concerning capacities and quasicontinuous representatives in \cite{Karak,Netrusov,Netrusov_capacity,Nuutinen} and the references therein.
In the latter case, although a sharp representative need not be unique pointwise, any two sharp representatives of the same element agree outside a set of Hausdorff dimension at most $n-ps$. Consequently, if $d\in [0, n]$ and $s>\frac{n-d}{p}$, then $n-ps<d$, and hence the exceptional set is $\mathcal{H}^d$-negligible. Thus, for every $E\subset\mathbb R^n$, the restriction of a sharp representative to $E$ is well defined modulo $\mathcal H^d$-a.e. equality. In the case $s>\frac np$, the same conclusion follows directly from the uniqueness of the continuous representative. We denote the resulting equivalence class by $f\big|^d_E$ and refer to it as the \emph{$d$-trace of $f$ to $E$}.
\par
In the present work, we use a slightly more flexible framework. Let $E\subset \mathbb{R}^n$ be a nonempty closed set equipped with a locally finite Borel regular measure $\fm$ satisfying $\operatorname{supp}\fm = E$. Assume moreover that $\fm$ is absolutely continuous with respect to $\mathcal{H}^d$ for some $d\in [0, n]$ such that $s>\frac{n-d}{p}$. Then the restriction to $E$ of any sharp representative of $f\in A^s_{p, q}(\mathbb{R}^n)$ is well defined up to $\fm$-a.e. equality. We call the resulting $\fm$-equivalence class the \emph{$\fm$-trace of $f$ to $E$} and denote it by $f\big|^{\fm}_E$.
\par
Having at our disposal the definition of $\fm$-trace we introduce the corresponding trace-space. Let $A\in \{B, F\}$, $s\in (0, 1)$, $p\in [1, \infty]$, $q\in (0, \infty]$, and $d\in [0, n]$ be such that $s>\frac{n-d}{p}$. In the Lizorkin--Triebel case, we additionally assume that $p<\infty$. Given $E\subset \mathbb{R}^n$ and $\fm$ as above, we define the trace-space by
\begin{equation}
    A^s_{p, q}(\mathbb{R}^n)\big|_E^{\fm} := \{f\big|^{\fm}_E: f\in A^s_{p, q}(\mathbb{R}^n)\},
\end{equation}
and equip it with the natural quotient-space quasi-norm
\begin{equation}
    \|\phi\|_{A^s_{p, q}(\mathbb{R}^n) \big|_E^{\fm}} := \inf\{\|f\|_{A^s_{p, q}(\mathbb{R}^n)}: \phi = f\big|^{\fm}_E\}.
\end{equation}
Furthermore, we define the trace operator by
\begin{equation}
    \operatorname{Tr}\big|^{\fm}_E: A^s_{p, q}(\mathbb{R}^n) \to A^s_{p, q}(\mathbb{R}^n)\big|_E^{\fm}
, \quad f \mapsto f\big|^{\fm}_E.
\end{equation}
We recall that $L_0(\fm)$ denotes the space of all $\fm$-equivalence classes of Borel functions $\phi:E\to\mathbb{R}$.
The trace problem considered in this paper can now be formulated as follows.
\begin{Problem}
\label{Pr.fm_trace_problem}
    Let $A\in \{B, F\}$, $s\in (0, 1)$, $p\in [1, \infty]$, $q\in (0, \infty]$, and $d\in [0, n]$ be such that $s>\frac{n-d}{p}$. In the Lizorkin--Triebel case, assume additionally that $p<\infty$. Let $E\subset \mathbb{R}^n$ be a closed set endowed with a Borel regular locally finite measure $\fm$ satisfying $\fm\ll \mathcal{H}^d$.
    \begin{enumerate}
        \item Given $\phi \in L_0(\fm)$, find necessary and sufficient conditions for $\phi \in A^s_{p, q}(\mathbb{R}^n)\big|^{\fm}_E$.
        \item Find an intrinsic quasi-norm on $A^s_{p, q}(\mathbb{R}^n)\big|^{\fm}_E$ equivalent to its quotient-space quasi-norm.
        \item Determine whether there exists a bounded operator, called an extension operator,
        \begin{equation}
            \operatorname{Ext}_{\fm}: A^s_{p, q}(\mathbb{R}^n)\big|_E^{\fm} \to A^s_{p, q}(\mathbb{R}^n),
        \end{equation}
        which is a right inverse to the trace operator, that is,
        \begin{equation}
            \operatorname{Tr}\big|^{\fm}_E \circ \operatorname{Ext}_{\fm} = \operatorname{Id} \qquad \text{on } A^s_{p, q}(\mathbb{R}^n)\big|_E^{\fm},
        \end{equation}
        and, in particular, whether such an operator can be chosen linear.
    \end{enumerate}
\end{Problem}
In the present paper, we address this problem for the class of $d$-thick sets. Our main emphasis is on the Banach range $p, q\ge 1$ (with $p<\infty$ in the Lizorkin--Triebel case). We also include the range $0<q<1$, since the methods developed below extend naturally to this quasi-Banach setting.
\par
\textbf{Previously known results.} 
We briefly discuss known results related to Problem~\ref{Pr.fm_trace_problem} and several adjacent questions.
\par
The trace problem for Besov spaces goes back to the pioneering work of Besov \cite{Bes_or}. In particular, he studied traces of \(B^s_{p,q}(\mathbb R^n)\) to \(\mathbb R^d\subset\mathbb R^n\), where \(p,q\in[1,\infty]\) and \(s>\frac{n-d}{p}\). The corresponding classical trace results for Lizorkin--Triebel spaces \(F^s_{p,q}(\mathbb R^n)\), in the range \(p\in(1,\infty)\), \(q\in[1,\infty)\), were obtained by Kalyabin \cite{Kalyabin}. These results are now classical; detailed accounts and further background can be found, for example, in \cite{Sawano,trieb}.
\par
An important subsequent direction concerns traces to Ahlfors--David regular sets. We recall that a closed set $E \subset \mathbb{R}^n$ is said to be Ahlfors--David $d$-regular, $d\in [0, n]$, if there are constants $C_1(E), C_2(E)>0$ such that (we use the symbol $Q_r(x)$ to denote the corresponding closed ball in the $\ell_{\infty}$-norm)
\begin{equation}
    C_1(E) r^d \le \mathcal{H}^d(Q_r(x)\cap E) \le C_2(E) r^d \qquad \text{for all } (x, r) \in E\times (0, 1].
\end{equation}
A substantial literature is devoted to traces to Ahlfors--David regular sets in various settings
(see, for example, \cite{ Ihnat, Jon, saks} and the
references therein). In particular, let \(X=(X,\rho,\mu)\) be an Ahlfors--David \(Q\)-regular metric measure space, and let \(E\subset X\) be Ahlfors--David \(d\)-regular. In \cite{saks}, the trace-space of \(B^s_{p,q}(X)\) to \(E\) was characterized for $s\in (0, 1)$, $\max\{\frac{Q}{d+s}, \frac{Q-d}{s}\}<p<\infty$, and $q\in (0, \infty]$. Moreover, under the additional assumption that \(E\) is porous and \(q>\frac{Q}{Q+s}\), corresponding trace results for \(F^s_{p,q}(X)\) were obtained in the same work.
\par
Ahlfors--David regularity, however, imposes a uniform dimensional structure on \(E\) and therefore excludes many natural irregular and mixed-dimensional sets. Classical methods developed for the Ahlfors--David regular setting do not directly apply, for instance, even to unions of Ahlfors--David regular pieces of different dimensions. This motivates passing to the substantially broader class of \(d\)-thick sets. We recall that $E\subset \mathbb{R}^n$ is called $d$-thick if there exists $C_d(E)>0$ such that
\begin{equation}
    \mathcal{H}^d_{\infty}(Q_r(x)\cap E) \ge C_d(E) r^d \qquad \text{for all } (x, r) \in E \times (0, 1].
\end{equation}
A number of examples of \(d\)-thick sets can be found in \cite{tyul,tyulenev_thick_sets_R_n}. In particular, every Ahlfors--David \(d\)-regular set is \(d\)-thick. More generally, if \(E_j\subset\mathbb R^n\), \(1\le j\le N\), are Ahlfors--David \(d_j\)-regular sets, then \(\bigcup_{j=1}^N E_j\) is \(d_*\)-thick, where \(d_*:=\min_j d_j\). Furthermore, every nondegenerate path-connected subset of \(\mathbb R^n\) is \(1\)-thick, while every nonempty subset of \(\mathbb R^n\) is \(0\)-thick.
The notion of $d$-thickness in extension problems for function spaces
goes back to the work of Rychkov \cite{Rychkov}, who studied linear
extension operators from $d$-thick open subsets of
$\mathbb{R}^n$. Related extension problems for Haj\l{}asz--Besov and
Haj\l{}asz--Lizorkin--Triebel spaces on subsets of metric measure
spaces were studied in \cite{Heikkinen}, where the existence of
extension operators is closely related to a measure density condition. Traces of Besov, Lizorkin--Triebel, and Sobolev spaces to $n$-thick sets were studied by Shvartsman in \cite{shvartsman_regular}.
\par
In the range $s>\frac{n}{p}$, where elements of $A^s_{p, q}(\mathbb{R}^n)$ admit continuous
representatives, traces of Besov spaces to arbitrary closed subsets of
$\mathbb{R}^n$ were characterized by Jonsson \cite{Jonsson}. A different intrinsic characterization in this range was obtained by Shvartsman \cite{shvartsman}. Subsequently, Jonsson \cite{Jonsson2009} developed an approach based on atomic decompositions that applies to traces of both Besov and Lizorkin--Triebel spaces on closed sets.
\par
Finally, we emphasize the works of Vodop'yanov and Tyulenev
\cite{tyulenev_thick_sets_R_n} and Tyulenev \cite{tyul}, where traces
of first-order Sobolev spaces to $d$-thick subsets of $\mathbb{R}^n$
and to their substitutes in metric measure spaces,
respectively, were studied. The methods developed in these works constitute an important starting point for the techniques used throughout the present paper.
\par
The \emph{aim of the present paper} is to address Problem~\ref{Pr.fm_trace_problem} for the class of closed \(d\)-thick sets. More precisely, we consider measures $\fm$ arising from strongly \(d\)-regular sequences associated with such sets (see Section~\ref{section.preliminaries} for a detailed discussion). In the Ahlfors--David regular setting, our results recover the corresponding classical trace theorems (see Section~\ref{section.Examples}). We emphasize that our results also cover the non-continuous range $ \frac{n-d}{p}<s\le\frac np$,
which was not treated by the previously known results in the present geometric generality.
\par
\textbf{Main results}
In order to formulate the main results of this paper, we first
introduce the necessary notation and technical tools.
\par
First, we recall the notions of a $d$-regular
sequence of measures. Given $d\in [0, n]$, assume that $E \subset \mathbb{R}^n$ is a closed $d$-thick set. For brevity, we write $\theta:= n-d$ for the corresponding codimension.
A sequence of Borel regular locally finite measures $\{\fm_k\}_{k=0}^{\infty}$ is called $d$-regular on $E$ if there
exist constants $C_1, C_2, C_3>0$ such that the following
conditions hold:
\begin{enumerate}
    \item for each $k\in \mathbb{N}_0$, $\operatorname{supp}\fm_k=E$;
    \item for each $k \in \mathbb{N}_0$, $\fm_k(Q_r(x))\le C_1r^d$ for all $(x, r) \in \mathbb{R}^n\times (0, 2^{-k}]$;
    \item for each $k \in \mathbb{N}_0$, $\fm_k(Q_r(x))\ge C_2r^d$ for all $(x, r) \in E\times [2^{-k}, 1]$;
    \item for each $k\in\mathbb{N}_0$, $\fm_k = \gamma_k\fm_0$, where $\gamma_k \in L_{\infty}(\fm_0)$, and furthermore, for every $j\in \mathbb{N}_0$
    \begin{equation}
        \frac{2^{(d-n)j}}{C_3}\le \frac{\gamma_{k}(x)}{\gamma_{k+j}(x)} \le C_3 \qquad \text{for $\fm_0$-a.e. } x\in E.
    \end{equation}
\end{enumerate}
Furthermore, a $d$-regular sequence $\{\fm_k\}_{k=0}^{\infty}$ is called strongly $d$-regular if, for each Borel set $G\subset E$,
\begin{equation}
    \limsup_{k\to \infty} \frac{\fm_k(Q_{2^{-k}}(x)\cap G)}{\fm_k(Q_{2^{-k}}(x))}>0 \qquad \text{for $\fm_0$-a.e. } x\in G.
\end{equation}
It is well known (see \cite[Theorem~1.3]{tyul}) that every closed $d$-thick set admits a strongly $d$-regular sequence of measures.
\par
Next we introduce Besov and Lizorkin--Triebel type functionals on a $d$-thick set, $d\in [0, n]$ as follows. Given $\sigma \in (0, \infty)$, a locally finite Borel regular measure $\fm$, and a Borel set $G\subset \mathbb{R}^n$ with $0<\fm(G) <\infty$, $\mathcal{E}_{\fm_k, \sigma}(f, G)$ denotes the best approximation by constants of $f\in L_{\sigma}(G, \fm)$, that is,
\begin{equation}
    \mathcal{E}_{\fm, \sigma}(f, G) := \inf_{c\in \mathbb{R}}\Bigl(\frac{1}{\fm(G)}\int\limits_{G}|f(x)-c|^{\sigma}d\fm(x) \Bigr)^{1/\sigma}.
\end{equation}
Furthermore, for each $(x, r) \in \mathbb{R}^n\times(0, \infty)$, we set
\begin{equation}
    \widetilde{\mathcal{E}}_{\fm, \sigma}(f, Q_r(x)) := \begin{cases}
        \mathcal{E}_{\fm, \sigma}(f, Q_{2r}(x)), \quad \text{if } Q_r(x) \cap \operatorname{supp}\fm \neq \emptyset, \\
        0, \quad \text{otherwise}.
    \end{cases}
\end{equation}
Given a $d$-regular sequence of measures $\{\fm_k\}_{k=0}^{\infty}$ on $E$, $p, q \in (0, \infty]$, $s\in (0, 1)$, and $\sigma \in (0, \infty)$, we set, for each $\phi \in L_{\sigma}^{\loc}(\{\fm_k\}):= \cap_{k=0}^{\infty}L_{\sigma}^{\loc}(\fm_k)$,
\begin{equation}
    \|\phi\|_{\mathfrak{b}^{s-\theta/p}_{p, q, \sigma}(E)} := \|\{2^{ks}\|\widetilde{\mathcal{E}}_{\fm_k, \sigma}(\phi, Q_{2^{-k}}(\cdot))\|_{L_p(\mathbb{R}^n)}\}_{k=0}^{\infty}\|_{\ell_q}.
\end{equation}
If additionally $p<\infty$, we put
\begin{equation}
    \|\phi\|_{\mathfrak{f}^{s-\theta/p}_{p, q, \sigma}(E)} := \|\|\{2^{ks}\widetilde{\mathcal{E}}_{\fm_k, \sigma}(\phi, Q_{2^{-k}}(\cdot))\}_{k=0}^{\infty}\|_{\ell_q}\|_{L_p(\mathbb{R}^n)}.
\end{equation}
Finally, for $\mathfrak{A} \in \{\mathfrak{B}, \mathfrak{F}\}$, we put
\begin{equation}
    \|\phi\|_{\mathfrak{A}^{s-\theta/p}_{p, q, \sigma}(E)} := \|\phi\|_{L_p(\fm_0)} + \|\phi\|_{\mathfrak{a}^{s-\theta/p}_{p, q, \sigma}(E)},
\end{equation}
where $\mathfrak a=\mathfrak b$ if $\mathfrak A=\mathfrak B$,
and $\mathfrak a=\mathfrak f$ if $\mathfrak A=\mathfrak F$.
\par
The \emph{first main result} of this paper reads as follows.
\begin{Th}
    \label{Th.main_stated_Besov}
    Let $E\subset \mathbb{R}^n$ be a $d$-thick closed set, $d\in [0, n]$, and let $\{\fm_k\}_{k=0}^{\infty}$ be a strongly $d$-regular sequence of measures on $E$. Assume that $p\in [1, \infty]$, $q\in (0, \infty]$, $s\in (\frac{n-d}{p}, 1)$, and $\sigma \in (0, \infty)$ such that $\sigma\le p$. Given $\phi \in L_{\sigma}^{\loc}(\{\fm_k\})$, the following conditions are equivalent:
    \begin{enumerate}
        \item $\phi \in B^s_{p, q}(\mathbb{R}^n)\big|_E^{\fm_0}$;
        \item $\|\phi\|_{\mathfrak{B}^{s-\theta/p}_{p, q, \sigma}(E)}<\infty$.
    \end{enumerate}
    Furthermore, the following equivalence holds
    \begin{equation}
        \label{eq.main_stated_Besov}
        \|\phi\|_{B^s_{p, q}(\mathbb{R}^n)\big|_E^{\fm_0}}\approx \|\phi\|_{\mathfrak{B}^{s-\theta/p}_{p, q, \sigma}(E)}
    \end{equation}
    with implicit constants independent of $\phi$. 
    Finally, there exists a bounded linear extension operator
    \begin{equation}
        \label{eq.main_Besov_stated_extension}
        \Ext_{\fm_0}: B^s_{p, q}(\mathbb{R}^n)\big|_E^{\fm_0} \to B^s_{p, q}(\mathbb{R}^n),
    \end{equation}
    which is a right inverse to the trace operator, i.e., $ \Tr\big|_E^{\fm_0} \circ \Ext_{\fm_0} = \operatorname{Id}$ on $B^s_{p, q}(\mathbb{R}^n)\big|_E^{\fm_0}$.
\end{Th}
The \emph{second main result} reads as follows.
\begin{Th}
    \label{Th.main_stated_LT}
    Let $E\subset \mathbb{R}^n$ be a $d$-thick closed set, $d\in [0, n]$, and let $\{\fm_k\}_{k=0}^{\infty}$ be a strongly $d$-regular sequence of measures on $E$. Assume that $p\in [1, \infty)$, $q\in (0, \infty]$, $s\in (\frac{n-d}{p}, 1)$, and $\sigma \in (0, \infty)$ such that $\sigma< p$ and $\sigma \le q$. Given $\phi \in L_{\sigma}^{\loc}(\{\fm_k\})$, the following conditions are equivalent:
    \begin{enumerate}
        \item $\phi \in F^s_{p, q}(\mathbb{R}^n)\big|_E^{\fm_0}$;
        \item $\|\phi\|_{\mathfrak{F}^{s-\theta/p}_{p, q, \sigma}(E)}<\infty$.
    \end{enumerate}
    Furthermore, the following equivalence holds
    \begin{equation}
        \label{eq.main_stated_LT}
        \|\phi\|_{F^s_{p, q}(\mathbb{R}^n)\big|_E^{\fm_0}}\approx \|\phi\|_{\mathfrak{F}^{s-\theta/p}_{p, q, \sigma}(E)}
    \end{equation}
    with implicit constants independent of $\phi$. 
    Finally, there exists a bounded extension operator
    \begin{equation}
        \label{eq.main_LT_stated_extension}
        \Ext_{\fm_0}: F^s_{p, q}(\mathbb{R}^n)\big|_E^{\fm_0} \to F^s_{p, q}(\mathbb{R}^n),
    \end{equation}
    which is a right inverse to the trace operator, i.e., $ \Tr\big|_E^{\fm_0} \circ \Ext_{\fm_0} = \operatorname{Id}$ on $F^s_{p, q}(\mathbb{R}^n)\big|_E^{\fm_0}$. If $p>1$ and $q\ge1$, the extension operator can be chosen to be bounded and linear.
\end{Th}
While our method does not produce a bounded linear extension operator in the Lizorkin--Triebel case when \(p=1\) or \(q<1\) (compare also with the range of parameters in \cite[Remark~1.7]{shvartsman_regular}, who used a similar construction), we do not claim that such an operator does not exist. The obstruction concerns the particular construction used in the present paper. Indeed, our extension operator is based on a family of almost-best approximations in \(L_\sigma(Q_{2^{-k}}(x),\fm_k)\), where \(k\in\mathbb N_0\) and \(x\in E\) (see Section~\ref{section.extension}). More precisely, for some \(\lambda>1\) we choose mappings $\mathcal T_k^x: L_\sigma(Q_{2^{-k}}(x),\fm_k)\to\mathbb R$
satisfying
\begin{equation}
    \Bigl( \frac1{\fm_k(Q_{2^{-k}}(x))} \int_{Q_{2^{-k}}(x)} |\phi(y)-\mathcal T_k^x\phi|^\sigma\,d\fm_k(y) \Bigr)^{1/\sigma} \le \lambda\, \mathcal E_{\fm_k,\sigma} (\phi,Q_{2^{-k}}(x)).
\end{equation}
If \(0<\sigma<1\) and \(\fm_k|_{Q_{2^{-k}}(x)}\) is nonatomic, no such mapping can be linear. Indeed, if \(\mathcal T_k^x\) were linear, the preceding inequality would imply that it reproduces constants and is a nonzero continuous linear functional on \(L_\sigma(Q_{2^{-k}}(x),\fm_k)\). This contradicts the classical theorem of Day asserting that the continuous dual of \(L_\sigma(\mu)\), \(0<\sigma<1\), is trivial whenever \(\mu\) is nonatomic. Thus, whenever our construction necessarily requires \(\sigma<1\), it cannot be made linear by choosing linear local almost-best approximations. For \(d>0\), the measures \(\fm_k\) are nonatomic by the upper \(d\)-regularity estimate.
This argument is only an obstruction to the present construction and does not rule out the existence of a bounded linear extension operator.
\par
\textbf{Organization of the paper.} The paper is organized as follows.
\begin{itemize}
    \item In Section~\ref{section.preliminaries} we provide the necessary background and prove several auxiliary results.
    \item In Section~\ref{section.extension} we give a general construction of the extension operator and prove its boundedness.
    \item Section~\ref{section.extension_property} is devoted to the important concept of $\{\fm_k\}$-Lebesgue points and the right inverse property of the extension operator.
    \item In Section~\ref{Section.direct_trace_theorem} we prove the so-called direct trace theorem. More specifically, we prove the boundedness of Besov and Lizorkin--Triebel type functionals on the corresponding trace-spaces.
    \item In Section~\ref{section.proof_of_the_main_results}, we give a proof of the main results, namely Theorem~\ref{Th.main_stated_Besov} and Theorem~\ref{Th.main_stated_LT}.
    \item In Section~\ref{section.Examples} we recover the trace theorems for Ahlfors--David regular sets from our main results.
\end{itemize}
\section{Preliminaries}
\label{section.preliminaries}
The aim of this section is to fix notation and collect a few auxiliary results useful in what follows.
\par
Throughout the paper, $C$ denotes a positive inessential constant whose value may change from line to line. If the dependence on parameters is important, we write $C=C(a,b,c,\ldots)$. We write $A\lesssim B$ if $A\le CB$, and $A\approx B$ if both $A\lesssim B$ and $B\lesssim A$.
\subsection{Geometric analysis background}
Throughout the paper, we fix an integer $n \in \mathbb{N}$. We use Latin letters $x, y, z$ to denote the points in the Euclidean space $\mathbb{R}^n$. Given $x\in\mathbb{R}^n$ and $r>0$, we write $Q_r(x)$ for the closed cube centered at $x$, with sides parallel to the coordinate axes and side length $2r$. When no confusion is possible, we use the notation $Q_k(x):=Q_{2^{-k}}(x)$, for each $k\in\mathbb{Z}$ and $x\in\mathbb{R}^n$. Given $c\in (0,\infty)$ and a cube $Q:=Q_r(x)$, we write $cQ:= Q_{cr}(x)$. A cube $Q$ is called a dyadic cube if there are $k\in\mathbb{Z}$ and $m\in\mathbb{Z}^n$ such that $Q=Q_{k, m}:=\prod_{i=1}^n\left[\frac{m_i}{2^k}, \frac{m_{i}+1}{2^k}\right]$. For convenience, we work with cubes rather than balls. Accordingly, throughout the paper $|x|$ denotes the $\ell_{\infty}$-norm of $x\in\mathbb{R}^n$, i.e., $|x|:=\max\{|x_i|: 1\le i\le n\}$. 
\par
Given a nonempty set $G\subset\mathbb{R}^n$, we set $\operatorname{diam}(G):=\sup\{|x-y|:x, y\in G\}$. Furthermore, given two nonempty sets $G_1, G_2 \subset\mathbb{R}^n$, we define $\operatorname{dist}(G_1, G_2):=\inf\{|x-y|:x\in G_1, y\in G_2\}$. For a set $G\subset\mathbb{R}^n$, we denote its closure and interior by $\overline{G}$ and $\operatorname{int}G$, respectively. Given $\delta>0$ and a nonempty set $G\subset\mathbb{R}^n$, set $U_{\delta}(G):=\{x\in\mathbb{R}^n: \operatorname{dist}(x, G)<\delta\}$. When no confusion is possible, we write $U_k(G):=U_{2^{-k}}(G)$, for each $k\in\mathbb{Z}$.
\par
Throughout the paper, the symbol $\fm$ denotes a \emph{locally finite Borel regular measure} on $\mathbb{R}^n$. Given a Borel set $G\subset \mathbb{R}^n$ and a measure $\fm$ on $\mathbb{R}^n$, we define the restriction of $\fm$ to $G$ by
\begin{equation}
    \label{eq.resriction_of_measure_definition}
    \fm\lfloor_G(F) = \fm(G\cap F) \qquad \text{for all Borel sets } F\subset\mathbb{R}^n.
\end{equation}
Sometimes we use weighted measures. More precisely, given a measure $\fm$, a locally integrable Borel function $\gamma:\mathbb{R}^n\to\mathbb{R}$ is called a \emph{weight} (or \emph{weight function}) if $\gamma(x)>0$ for $\fm$-a.e. $x\in\mathbb{R}^n$. By $\gamma\fm$ we denote the corresponding weighted measure, that is
\begin{equation}
    \gamma\fm(G) = \int\limits_{G}\gamma(x)d\fm(x) \qquad \text{for all Borel sets } G\subset \mathbb{R}^n.
\end{equation}
\par
Given a measure $\fm$ and $p\in (0, \infty]$, by $L_p(\fm)$ we denote the space of $\fm$-equivalence classes of all Borel functions $f:\mathbb{R}^n\to\mathbb{R}$ such that the corresponding quasi-norm is finite:
\begin{equation}
\label{eq.L_p_norm_definition}
\|f\|_{L_p(\fm)}:=\Bigl(\int\limits_{\mathbb{R}^n}|f(x)|^pd\fm(x)\Bigr)^{1/p},
\end{equation}
with the usual modification when $p=\infty$. The corresponding space of locally $p$-integrable functions is denoted by $L_p^{\loc}(\fm)$. For a Borel set $G\subset\mathbb{R}^n$, by $L_p(G, \fm)$ ($L_p^{\loc}(G, \fm)$) we denote $L_p(\fm\lfloor_G)$ ($L_p^{\loc}(\fm\lfloor_G)$). We adopt the standard notation $L_0(\fm)$ for the space of all $\fm$-equivalence classes of Borel functions $f:\mathbb{R}^n \to \mathbb{R}$.
\par
For each Borel set $G\subset\mathbb{R}^n$ with $\fm(G)<\infty$ and each $f\in L_1(G, \fm)$, we put
\begin{equation}
\label{eq.average_definition}
    f_{G, \fm}:=\fint\limits_{G}f(x)d\fm(x):=\begin{cases}
        \frac{1}{\fm(G)}\int\limits_{G}f(x)d\fm(x), \quad \text{if } \fm(G)>0,\\
        0, \quad \text{if } \fm(G)=0.
    \end{cases}
\end{equation}
Furthermore, if $G = Q_r(x)$ we also define the \emph{averaging operator} by
\begin{equation}
\label{eq.average_cube_definition}
    A_{r, \fm}f(x):= \fint\limits_{Q_{r}(x)}f(y)d\fm(y).
\end{equation}
 We shall also use \emph{median values}. More precisely, given a measure $\fm$, a Borel function $f:\mathbb{R}^n\to\mathbb{R}$, and a Borel set $G\subset \mathbb{R}^n$ with $\fm(G)<\infty$, we set
\begin{equation}
    \label{eq.median_value_definition}
    \operatorname{med}_{\fm}(f, G) := \begin{cases}\max\Bigl\{c\in \mathbb{R}: \fm(\{x\in G: f(x)<c\})\le \frac{\fm(G)}{2}\Bigr\}, \quad \text{if } \fm(G)>0,\\
    0, \quad \text{if } \fm(G)=0.
    \end{cases}
\end{equation}
When no ambiguity is possible, we let $A_{k, \fm}:=A_{2^{-k}, \fm}$ for brevity. If $\fm=\mathcal{L}^n$ is the Lebesgue measure, then we suppress it from the notation. More precisely, $f_{G}:=f_{G, \mathcal{L}^n}$, $A_r:=A_{r, \mathcal{L}^n}$, and $\operatorname{med}(f, G):=\operatorname{med}_{\mathcal{L}^n}(f, G)$. The following property of a median is well-known. We provide details for completeness.
\begin{Lm}
    \label{Lm.meadian_average_relation} Given $\sigma \in (0, \infty)$ and a measure $\fm$, for each Borel set $G\subset\mathbb{R}^n$ with $\fm(G)<\infty$ and each $f\in L_{\sigma}(G, \fm)$, the following inequality holds:
    \begin{equation}
        \label{eq.median_average_relation}
        |\operatorname{med}_{\fm}(f, G)-c|^{\sigma} \le 2 \fint\limits_{G}|f(y)-c|^{\sigma}d\fm(y) \qquad \text{for all } c \in \mathbb{R}.
    \end{equation}
    Moreover,
    \begin{equation}
    \label{eq.median_at_lebesgue_points}
        f(x) = \lim_{r\to 0}\operatorname{med}_{\fm}(f, Q_r(x))
    \end{equation}
    at every Lebesgue point $x\in \mathbb{R}^n$ of $f$.
\end{Lm}
\begin{proof}
    If $\fm(G) = 0$, then the assertion is trivial. Assume that $0<\fm(G)<\infty$ and $f\in L_{\sigma}(G, \fm)$. If $c\le\operatorname{med}_{\fm}(f, G)$, then estimate
    \begin{equation}
        \fm\bigl(\{x\in G: f(x)\ge \operatorname{med}_{\fm}(f, G)\}\bigr) \ge \frac{\fm(G)}{2}
    \end{equation}
    gives
    \begin{equation}
        \fint\limits_{G}|f(y)-c|^{\sigma}d\fm(y)\ge \frac{1}{2}|\operatorname{med}_{\fm}(f, G) - c|^{\sigma}.
    \end{equation}
    If $c\ge\operatorname{med}_{\fm}(f, G)$, then the same bound follows from 
    \begin{equation}
        \fm\bigl(\{x\in G: f(x)\le \operatorname{med}_{\fm}(f, G)\}\bigr) \ge \frac{\fm(G)}{2}.
    \end{equation}
    The latter estimate follows directly from the maximality property in the definition of the median.
    \par
    Property \eqref{eq.median_at_lebesgue_points} follows from \eqref{eq.median_average_relation}. Indeed, choosing $\sigma =1$, we obtain
    \begin{equation}
        |f(x)-\operatorname{med}_{\fm}(f, Q_r(x))| \le 2\fint\limits_{Q_r(x)}|f(y)-f(x)|d\fm(y).
    \end{equation}
    The last quantity tends to $0$ as $r\to 0$ at every Lebesgue point. The proof is complete.
\end{proof}
\par
Next we define the \emph{local best approximation by constants} in $L_{\sigma}(\fm)$. More precisely, given a measure $\fm$, a Borel set $G\subset\mathbb{R}^n$ with $\fm(G)<\infty$, $\sigma \in (0, \infty)$, and $f\in L_{\sigma}(G, \fm)$, we set
\begin{equation}
\label{eq.local_approximation_definition}
    \mathcal{E}_{\fm, \sigma}(f, G) := \inf_{c \in \mathbb{R}}\Bigl(\fint\limits_{G}|f(x)-c|^{\sigma}d\fm(x)\Bigr)^{1/\sigma}.
\end{equation}
Furthermore, we put
\begin{equation}
\label{eq.modified_local_approximation_definition}
    \widetilde{\mathcal{E}}_{\fm, \sigma}(f, Q_r(x)) := \begin{cases}
        \mathcal{E}_{\fm, \sigma}(f, Q_{2r}(x)), \quad \text{if } Q_r(x)\cap\operatorname{supp}\fm \neq \emptyset,\\
        0, \quad\text{otherwise}.
    \end{cases}
\end{equation}
Occasionally we write $\mathcal{E}_{ \sigma}:=\mathcal{E}_{\mathcal{L}^n, \sigma}$, $\mathcal{E}_{ \fm}:=\mathcal{E}_{\fm, 1}$, and $\mathcal{E}:=\mathcal{E}_{\mathcal{L}^n, 1}$. The same notation applies to $\widetilde{\mathcal{E}}_{\fm, \sigma}$. 
\begin{Def}
\label{Def.almost_best_approximating_constant_definition}
    Let $\lambda> 1$. Given $\sigma\in (0, \infty)$, a measure $\fm$, a Borel set $G\subset \mathbb{R}^n$ with $\fm(G)<\infty$, and $f\in L_{\sigma}(G, \fm)$, a number $c\in \mathbb{R}$ is called a $\lambda$-almost best approximating constant of $f$ in $L_{\sigma}(G, \fm)$ if
    \begin{equation}
        \label{eq.almost_best_approximating_constant_definition}
        \Bigl(\fint\limits_{G}|f(x)-c|^{\sigma}d\fm(x)\Bigr)^{1/\sigma} \le \lambda \mathcal{E}_{\fm, \sigma}(f, G).
    \end{equation}
    We denote the set of all $\lambda$-almost best approximating constants of $f$ in $L_{\sigma}(G, \fm)$ by $\mathfrak{C}^{\lambda}_{\fm, \sigma}(f, G)$.
\end{Def}
In the following lemma we show that integral averages and medians provide almost best approximating constants.
\begin{Lm}
    \label{Lm.average_median_approximation_property}
    Let $\mathfrak{m}$ be a measure and let $G\subset \mathbb{R}^n$ be a Borel set with $\fm(G)<\infty$. 
    \begin{enumerate}
        \item For each $\sigma \in [1, \infty)$, there exists $\lambda=\lambda(\sigma)$ such that for every $f\in L_{\sigma}(G, \fm)$, $f_{G, \fm} \in \mathfrak{C}^{\lambda}_{\fm, \sigma}(f, G)$.
        \item For each $\sigma \in (0, \infty)$, there exists $\lambda=\lambda(\sigma)$ such that for every $f\in L_{\sigma}(G, \fm)$, $\operatorname{med}_{\fm}(f, G) \in \mathfrak{C}^{\lambda}_{\fm, \sigma}(f, G)$.
    \end{enumerate}
\end{Lm}
\begin{proof}
If $\fm(G) = 0$, the desired assertions are trivial. Thus, we assume $0<\fm(G)<\infty$. We prove the first statement. For each $c\in \mathbb{R}$, the elementary power inequality and Jensen's inequality give
    \begin{equation}
        \fint\limits_{G}\bigl|f(x)-f_{G, \fm}\bigr|^{\sigma}d\fm(x) \lesssim \fint\limits_{G}|f(x)-c|^{\sigma}d\fm(x) + |c-f_{G, \fm}|^{\sigma}\lesssim\fint\limits_{G}|f(x)-c|^{\sigma}d\fm(x).
    \end{equation}
    Taking the infimum over $c\in\mathbb{R}$ yields the desired estimate.
    \par
    For the second assertion, we use Lemma~\ref{Lm.meadian_average_relation}. Then for every $c\in \mathbb{R}$,
     \begin{equation}
     \begin{split}
        \fint\limits_{G}\bigl|f(x)-\operatorname{med}_{\fm}(f, G)\bigr|^{\sigma}d\fm(x) &\lesssim \fint\limits_{G}\bigl|f(x)-c\bigr|^{\sigma}d\fm(x) + |\operatorname{med}_{\fm}(f, G)-c|^{\sigma} \\ &\lesssim \fint\limits_{G}|f(x)-c|^{\sigma}d\fm(x).
    \end{split}
    \end{equation}
    Taking the infimum over $c\in\mathbb{R}$ completes the proof.
\end{proof}
\begin{Def}
\label{Def.approximating_operator}
    Let $\lambda> 1$. Given $\sigma\in (0, \infty)$, a measure $\fm$, and a Borel set $G\subset \mathbb{R}^n$ with $\fm(G)<\infty$, we say that $\mathcal{T}:L_{\sigma}(G, \fm)\to\mathbb{R}$ is a $\lambda$-almost best approximating operator if $\mathcal{T}f \in \mathfrak{C}^{\lambda}_{\fm, \sigma}(f, G)$ for every $f\in L_{\sigma}(G, \fm)$.
\end{Def}
\begin{Remark}
\label{Rm.approximating_operator}
    Lemma~\ref{Lm.average_median_approximation_property} provides two examples of almost best approximating operators: integral averages and medians. Note also that an almost best approximating operator need not be linear (e.g., medians).
\end{Remark}
\par
Finally, we recall the concept of Hausdorff measures and contents. Given $d\ge 0$ and $G\subset\mathbb{R}^n$, we put, for each $\delta \in (0, \infty]$,
\begin{equation}
\label{eq.hausdorff_content_definition}
    \mathcal{H}^d_{\delta}(G) := \inf\left\{\sum_{i}r_i^d: G\subset \bigcup_{i}Q_{r_i}(x_i), \quad 0<r_i<\delta\right\},
\end{equation}
where the infimum is taken over all at most countable coverings of $G$ by cubes $\{Q_{r_i}(x_i)\}$. The $d$-Hausdorff measure of $G$ is defined by
\begin{equation}
\label{eq.hausdorff_measure_definition}
    \mathcal{H}^d(G) := \lim_{\delta\to 0}\mathcal{H}^d_{\delta}(G).
\end{equation}
For $G\subset \mathbb{R}^n$, we define the Hausdorff dimension of $G$ by
\begin{equation}
\label{eq.hausdorff_dimension_definition}
    \operatorname{dim}_{H}G:=\inf\{t\ge 0: \mathcal{H}^t(G)=0\}.
\end{equation}
\subsection{Maximal function}
Here we recall the maximal function, one of the important tools used throughout the paper.
\begin{Def}
    Given $f\in L_1^{\loc}(\mathbb{R}^n)$, the symbol $M[f]$ denotes the Hardy--Littlewood maximal function, that is,
    \begin{equation}
        M[f](x) := \sup_{r>0} \fint\limits_{Q_r(x)}|f(y)|dy \qquad \text{for all } x\in\mathbb{R}^n.
    \end{equation}
\end{Def}
The following property of the maximal function is well-known (see, e.g., \cite{Stein}).
\begin{Th}
\label{Th.maximal_function_boundedness}
    Given $p \in (1, \infty]$, there exists $C = C(p)>0$ such that for each $f\in L_p(\mathbb{R}^n)$,
    \begin{equation}
        \|M[f]\|_{L_p(\mathbb{R}^n)} \le C\|f\|_{L_p(\mathbb{R}^n)}.
    \end{equation}
\end{Th}
We also use the generalization of the preceding estimate to the setting of vector-valued functions, that is, the famous Fefferman--Stein inequality (see, e.g., \cite{Fefferman_Stein}).
\begin{Th}
    \label{Th.Fefferman_Stein_inequality}
    Given $p\in (1, \infty)$ and $q \in (1, \infty]$, there exists $C = C(p, q)$ such that for each sequence of Borel functions $\{f_k\}_{k=1}^{\infty}$,
    \begin{equation}
        \Bigl\| \Bigl(\sum_{k=1}^{\infty} (M[f_k](\cdot))^q \Bigr)^{1/q} \Bigr\|_{L_p(\mathbb{R}^n)} \le C  \Bigl\| \Bigl(\sum_{k=1}^{\infty} (f_k(\cdot))^q \Bigr)^{1/q} \Bigr\|_{L_p(\mathbb{R}^n)}
    \end{equation}
    with the usual modification when $q = \infty$.
\end{Th}
\subsection{Besov and Lizorkin--Triebel spaces} There are several equivalent approaches to the definition of Besov and Lizorkin--Triebel spaces. For our purposes, it is convenient to use a definition based on fractional gradients.
\begin{Def}
    Let $s\in (0, 1)$ and let $f:\mathbb{R}^n\to\mathbb{R}$ be a Borel function. A sequence of nonnegative Borel functions $\vec{g} :=\{g_k\}_{k\in\mathbb{Z}}$ is called a fractional $s$-Haj\l{}asz gradient of $f$ if there exists an $\mathcal{L}^n$-negligible set $N\subset\mathbb{R}^n$ such that for each $k\in\mathbb{Z}$ and all $x, y\in\mathbb{R}^n\setminus N$ satisfying $2^{-k-1}\le |x-y|<2^{-k}$
    \begin{equation}
    \label{eq.fractional_gradient_definition}
        |f(x)-f(y)|\le |x-y|^s(g_k(x)+g_k(y)).
    \end{equation}
    We denote the set of all fractional $s$-Haj\l{}asz gradients of $f$ by $\mathbb{D}^s(f)$.
\end{Def}
In what follows, for $p, q\in (0, \infty]$ and a sequence $\vec{g}$ of Borel functions, we set
\begin{equation}
    \|\vec{g}\|_{\ell_q(L_p(\mathbb{R}^n))}:=\|\{\|g_k\|_{L_p(\mathbb{R}^n)}\}_{k}\|_{\ell_q}, \qquad \|\vec{g}\|_{L_p(\mathbb{R}^n, \ell_q)}:=\|\|\{g_k\}_{k}\|_{\ell_q}\|_{L_p(\mathbb{R}^n)}.
\end{equation}
\begin{Def}
    Given $s\in (0, 1)$, $p, q \in (0, \infty]$, the Besov space $B^s_{p, q}(\mathbb{R}^n)$ consists of all $f\in L_p(\mathbb{R}^n)$ such that
    \begin{equation}
    \label{eq.besov_seminorm_definition}
        \|f\|_{b^s_{p, q}(\mathbb{R}^n)}:=\inf_{\vec{g}\in \mathbb{D}^s(f)} \|\vec{g}\|_{\ell_q(L_p(\mathbb{R}^n))} <\infty.
    \end{equation}
    We equip this space with the quasi-norm
    \begin{equation}
    \label{eq.besov_norm_definition}
        \|f\|_{B^s_{p, q}(\mathbb{R}^n)}:=\|f\|_{L_p(\mathbb{R}^n)} + \|f\|_{b^s_{p, q}(\mathbb{R}^n)}, \qquad f\in B^s_{p, q}(\mathbb{R}^n)
    \end{equation}
\end{Def}
\begin{Def}
    Given $s\in (0, 1)$, $p\in (0, \infty)$, and $q \in (0, \infty]$, the Lizorkin--Triebel space $F^s_{p, q}(\mathbb{R}^n)$ consists of all $f\in L_p(\mathbb{R}^n)$ such that
    \begin{equation}
    \label{eq.lizorkin_triebel_seminorm_definition}
        \|f\|_{f^s_{p, q}(\mathbb{R}^n)}:=\inf_{\vec{g}\in \mathbb{D}^s(f)}\|\vec{g}\|_{L_p(\mathbb{R}^n, \ell_q)}.
    \end{equation}
    We equip this space with the quasi-norm
    \begin{equation}
    \label{eq.lizorkin_triebel_norm_definition}
        \|f\|_{F^s_{p, q}(\mathbb{R}^n)}:=\|f\|_{L_p(\mathbb{R}^n)} + \|f\|_{f^s_{p, q}(\mathbb{R}^n)}, \qquad f\in F^s_{p, q}(\mathbb{R}^n)
    \end{equation}
\end{Def}
We refer the interested reader to \cite{AKZ, kosk_Yang_Zhou} and the references therein for a detailed discussion of fractional gradients. In particular, by \cite[Theorem~1.2 and Theorem~3.1]{AKZ}, the definitions of Besov and Lizorkin--Triebel spaces introduced above are equivalent to the corresponding oscillation-based definitions, with equivalent quasi-norms, for the whole range of $p, q$, and $s$. If additionally $p\in (\frac{n}{n+s}, \infty)$ in the Besov case and $p\in (\frac{n}{n+s}, \infty)$, $q\in (\frac{n}{n+s}, \infty]$ in the Lizorkin--Triebel case, then the Fourier-analytical approach gives the same spaces. Note also that the Besov and Lizorkin--Triebel spaces introduced above are usually called Haj\l{}asz--Besov and Haj\l{}asz--Lizorkin--Triebel spaces and denoted by $N^s_{p, q}(\mathbb{R}^n)$ and $M^s_{p, q}(\mathbb{R}^n)$, respectively. Since these spaces are linearly isomorphic to the corresponding oscillation-based spaces, we will keep the notation $B^s_{p, q}(\mathbb{R}^n)$ and $F^s_{p, q}(\mathbb{R}^n)$.
\begin{Remark}
\label{Rm.fractional_gradient=_truncation}
Let $s\in (0, 1)$, $p, q\in (0, \infty]$. Since we work with inhomogeneous Besov and Lizorkin--Triebel spaces, we may truncate the fractional gradients under consideration. More precisely, assume that $k_0\in \mathbb{Z}$ is fixed. Given $f\in L_p(\mathbb{R}^n)$, we say that a sequence of nonnegative Borel functions $\vec{h}:=\{h_k\}_{k>k_0}$ is a $k_0$-truncated fractional $s$-Haj\l{}asz gradient of $f$ if there exists an $\mathcal{L}^n$-negligible set $N\subset \mathbb{R}^n$ such that \eqref{eq.fractional_gradient_definition} holds for every $k> k_0$ and all $x, y\in\mathbb{R}^n\setminus N$ satisfying $2^{-k-1}\le |x-y|<2^{-k}$. We denote the set of all $k_0$-truncated fractional $s$-Haj\l{}asz gradients of $f$ by $\mathbb{D}_{k_0}^s(f)$. Given $\vec{h} \in \mathbb{D}_{k_0}^s(f)$, we define the fractional gradient $\vec{g}\in \mathbb{D}^s(f)$ by setting $g_k:=2^{(k+1)s}|f|$ for each $k\le k_0$ and $g_k:=h_k$ for each $k>k_0$. Then
    \begin{equation}
         \|\vec{g}\|_{\ell_q(L_p(\mathbb{R}^n))} \approx \|\vec{h}\|_{\ell_q(L_p(\mathbb{R}^n))}+\|f\|_{L_p(\mathbb{R}^n)}
    \end{equation}
    with implicit constants depending only on $s$, $q$, and $k_0$.
    Conversely, every $\vec{g} \in \mathbb{D}^s(f)$ gives a truncated gradient by restriction. Consequently, for each $f\in B^s_{p, q}(\mathbb{R}^n)$
    \begin{equation}
        \|f\|_{B^s_{p, q}(\mathbb{R}^n)}\approx \|f\|_{L_p(\mathbb{R}^n)} + \inf_{\vec{h}\in \mathbb{D}^s_{k_0}(f)}\|\vec{h}\|_{\ell_q(L_p(\mathbb{R}^n))}.
    \end{equation}
    \par
    If $p<\infty$, an analogous equivalence holds for Lizorkin--Triebel spaces, with $\ell_q(L_p(\mathbb{R}^n))$ replaced by $L_p(\mathbb{R}^n, \ell_q)$.
\end{Remark}
\par
We shall repeatedly use the following discrete Hardy inequality. We refer the interested reader to \cite{kuf} for the general form of Hardy's inequality and related topics.
\begin{Th}
    \label{Th.Hardy_inequality}
    Let $q\in(0,\infty]$, $\sigma\in(0,q]$, and $\alpha>0$.
    Given a sequence of real numbers $\{a_k\}_{k\in\mathbb Z}$, put
    \begin{equation}
        b_k^+:=\|\{a_j\}_{j\ge k}\|_{\ell_\sigma},
        \qquad
        b_k^-:=\|\{a_j\}_{j\le k}\|_{\ell_\sigma}.
    \end{equation}
    Then
    \begin{equation}
        \label{eq.Hardy_inequality}
        \|\{2^{k\alpha}b_k^+\}_{k\in\mathbb{Z}}\|_{\ell_q}
        \le C\|\{2^{k\alpha}a_k\}_{k\in\mathbb{Z}}\|_{\ell_q},
    \end{equation}
    and
    \begin{equation}
        \label{eq.dual_Hardy_inequality}
        \|\{2^{-k\alpha}b_k^-\}_{k\in\mathbb{Z}}\|_{\ell_q}
        \le C\|\{2^{-k\alpha}a_k\}_{k\in\mathbb{Z}}\|_{\ell_q}.
    \end{equation}
\end{Th}
\begin{Ca}
    \label{Ca.geometric_discrete_convolution}
    Let $q\in(0,\infty]$ and $\alpha>0$. Then for every
    sequence of nonnegative numbers $\{a_k\}_{k\in\mathbb Z}$,
    \begin{equation}
    \label{eq.geometric_discrete_convolution1}
        \Bigl\|
        \Bigl\{
        \sum_{j=k}^{\infty}2^{-(j-k)\alpha}a_j
        \Bigr\}_{k\in\mathbb Z}
        \Bigr\|_{\ell_q}
        \lesssim
        \|\{a_k\}_{k\in \mathbb{Z}}\|_{\ell_q},
    \end{equation}
    and
    \begin{equation}
    \label{eq.geometric_discrete_convolution2}
        \Bigl\|
        \Bigl\{
        \sum_{j=-\infty}^{k}2^{-(k-j)\alpha}a_j
        \Bigr\}_{k\in\mathbb Z}
        \Bigr\|_{\ell_q}
        \lesssim
        \|\{a_k\}_{k\in\mathbb{Z}}\|_{\ell_q}.
    \end{equation}
\end{Ca}
\begin{proof}
    We prove the first assertion. The second one can be established by the same argument. If $q\ge 1$, then \eqref{eq.geometric_discrete_convolution1} is precisely the Hardy inequality. For $q<1$, we use the subadditivity of the function $t\mapsto t^q$, $t\ge 0$. Then
    \begin{equation}
        \sum_{k\in\mathbb{Z}}\Bigl(\sum_{j=k}^{\infty}2^{-(j-k)\alpha}a_j\Bigr)^q \le \sum_{j\in\mathbb{Z}}a_j^q\sum_{k=-\infty}^{j}2^{-(j-k)\alpha q} \approx\sum_{j\in\mathbb{Z}}a_j^q.
    \end{equation}
\end{proof}
\begin{Remark}
    \label{Rm.geometric_discrete_convolution_half_line}
    The estimates of Corollary~\ref{Ca.geometric_discrete_convolution}
    remain valid, with the same constant, on every discrete half-line
    $\{k\in\mathbb{Z}:k\ge \underline{k}\}$. More precisely,
    \begin{equation}
        \Bigl\|
        \Bigl\{
        \sum_{j=k}^{\infty}2^{-(j-k)\alpha}a_j
        \Bigr\}_{k\ge \underline{k}}
        \Bigr\|_{\ell_q}
        \lesssim
        \|\{a_k\}_{k\ge \underline{k}}\|_{\ell_q},
    \end{equation}
    and
    \begin{equation}
        \Bigl\|
        \Bigl\{
        \sum_{j=\underline{k}}^{k}2^{-(k-j)\alpha}a_j
        \Bigr\}_{k\ge \underline{k}}
        \Bigr\|_{\ell_q}
        \lesssim
        \|\{a_k\}_{k\ge\underline{k}}\|_{\ell_q}.
    \end{equation}
    Indeed, it suffices to extend the sequence $\{a_k\}_{k\ge \underline{k}}$
    by zero to all $k\in\mathbb{Z}$ and apply
    Corollary~\ref{Ca.geometric_discrete_convolution}.
\end{Remark}
\par
In Section~\ref{Section.direct_trace_theorem} we shall use fractional gradients satisfying an additional almost-monotonicity property. More precisely, we say that a sequence of Borel functions $\{g_k\}_{k\in \mathbb{Z}}$ is \emph{$\varepsilon$-nonincreasing} if $g_j(x) \le 2^{(j-k)\varepsilon}g_k(x)$ for every $x\in \mathbb{R}^n$ and all $j\ge k$.
\begin{Lm}
    \label{Lm.monotonic_fractional_gradient} 
    Let $s\in (0, 1)$, $p\in [1, \infty]$, $q \in (0, \infty]$, and let $\varepsilon>0$. Given a Borel function $f:\mathbb{R}^n\to \mathbb{R}$, for each $\vec{g}\in \mathbb{D}^s(f)$ there exists an $\varepsilon$-nonincreasing fractional gradient $\vec{h}\in\mathbb{D}^s(f)$ such that
    \begin{equation}
    \label{eq.monotonic_fractional_gradient_triebel}
        \|\{h_k(x)\}_{k\in\mathbb{Z}}\|_{\ell_q} \le C \|\{g_k(x)\}_{k\in\mathbb{Z}}\|_{\ell_q}
    \end{equation}
    for all $x\in\mathbb{R}^n$.
    Furthermore,
    \begin{equation}
    \label{eq.monotonic_fractional_gradient_besov}
        \|\vec{h}\|_{\ell_q(L_p(\mathbb{R}^n))} \le C \|\vec{g}\|_{\ell_q(L_p(\mathbb{R}^n))},
    \end{equation}
   where $C$ may depend on $\varepsilon, p, q$, but is independent of $f, \vec{g}$, and $\vec{h}$.
\end{Lm}
\begin{proof}
    Given $\vec{g} \in \mathbb{D}^s$, put $h_k:=2^{k\varepsilon}\sum_{j=k}^{\infty}2^{-j\varepsilon}g_j$, $k\in \mathbb{Z}$.
    Then for every $j\ge k$,
    \begin{equation}
        h_j = 2^{(j-k)\varepsilon}2^{k\varepsilon}\sum_{i=j}^{\infty}2^{-i\varepsilon}g_i \le  2^{(j-k)\varepsilon}h_k.
    \end{equation}
    Since $g_k(x)\le h_k(x)$ for all $(x, k)\in\mathbb{R}^n\times \mathbb{Z}$, we have $\vec{h}\in\mathbb{D}^s(f)$. Moreover, Corollary~\ref{Ca.geometric_discrete_convolution} gives
    \begin{equation}
        \|\{h_k(x)\}_{k\in\mathbb{Z}}\|_{\ell_q} = \Bigl\|\Bigl\{\sum_{j=k}^{\infty}2^{-(j-k)\varepsilon}g_j(x)\Bigr\}_{k\in\mathbb{Z}}\Bigr\|_{\ell_q} \lesssim \|\{g_k(x)\}_{k\in\mathbb{Z}}\|_{\ell_q}
    \end{equation}
    for all $x\in \mathbb{R}^n$. The second estimate can be obtained by the same argument combined with Minkowski's inequality. Indeed, for each $k\in \mathbb{Z}$,
    \begin{equation}
        \|h_k\|_{L_p(\mathbb{R}^n)}\le 2^{k\varepsilon}\sum_{j=k}^{\infty}2^{-j\varepsilon}\|g_j\|_{L_p(\mathbb{R}^n)}.
    \end{equation}
    Consequently, Corollary~\ref{Ca.geometric_discrete_convolution} gives \eqref{eq.monotonic_fractional_gradient_besov}.
\end{proof}
\subsection{Thick sets and regular sequences of measures}
Following \cite{tyul, tyulenev_thick_sets_R_n}, we recall the notion of a $d$-thick set.
\begin{Def}
    Given $d\in [0, n]$, a set $E \subset \mathbb{R}^n$ is said to be $d$-thick if there exists a constant $C_d(E)>0$ such that
    \begin{equation}
    \label{eq.d_thick_set_definition}
        \mathcal{H}^d_{\infty}(E\cap Q_r(x)) \ge C_d(E)r^d \qquad \text{for all } (x, r) \in E\times (0, 1].
    \end{equation}
    We write $\theta:=n-d$ for the corresponding codimension.
\end{Def}
Recall also the following standard notation.
\begin{Def}
    Given $d\in [0, n]$, a closed set $E\subset\mathbb{R}^n$ is called Ahlfors--David $d$-regular if there exist $C_1, C_2>0$ such that
    \begin{equation}
        \label{eq.regular_sets_definition}
        C_1r^d \le \mathcal{H}^d(Q_r(x)\cap E)\le C_2r^d \qquad \text{for all } (x, r) \in E \times (0, 1].
    \end{equation}
\end{Def}
The class of $d$-thick sets is fairly broad (see \cite{tyul, tyulenev_thick_sets_R_n}). First, every Ahlfors--David $d$-regular set is $d$-thick. Moreover, if $E = \bigcup_{j=1}^N E_j$ with $E_j$ being Ahlfors--David $d_j$-regular, then $E$ is $d := \min_j d_j$-thick. There are also more general examples: every path-connected $E\subset \mathbb{R}^n$ is $1$-thick and every $(\varepsilon, \delta)$-domain $\Omega \subset \mathbb{R}^n$ is $n$-thick. For a more detailed discussion of different examples, see Section~\ref{section.Examples}.
\begin{Def}
    Given $d\in [0, n]$ and a closed $d$-thick set $E \subset \mathbb{R}^n$, a sequence of Borel measures $\{\mathfrak{m}_k\}_{k=0}^{\infty}$ is called $d$-regular if $\operatorname{supp}\mathfrak{m}_k = E$ for every $k\in \mathbb{N}_0$, and there exist constants $C_1, C_2, C_3>0$ such that the following conditions hold for every $k\in \mathbb{N}_0$:
    \begin{conditions}{\textbf{M}}
    \item\label{cond:M1} $\fm_k(Q_r(x))\le C_1r^d$ for all $(x, r) \in \mathbb{R}^n\times (0, 2^{-k}]$;
    \item\label{cond:M2} $\fm_k(Q_r(x))\ge C_2r^d$ for all $(x, r) \in E\times[2^{-k}, 1]$;
    \item\label{cond:M3} $\fm_k = \gamma_k\fm_0$, where $\gamma_k\in L_{\infty}(\fm_0)$, and, for each $j\in \mathbb{N}_0$,
    \begin{equation}
    \label{eq.weights_of_regular_measures_definition}
        \frac{2^{(d-n)j}}{C_3}\gamma_{k+j}(x)\le \gamma_k(x)\le C_3\gamma_{k+j}(x) \qquad \text{for $\fm_0$-a.e. } x\in E.
    \end{equation}
\end{conditions}
\end{Def}
An important result on thick sets asserts that every closed $d$-thick set $E\subset\mathbb{R}^n$ admits a $d$-regular sequence of measures (see \cite{tyul, tyulenev_thick_sets_R_n}). 
\par
Throughout the rest of this subsection, we assume that $d\in [0, n]$, $E\subset\mathbb{R}^n$ is a closed $d$-thick set, and $\{\fm_k\}_{k=0}^{\infty}$ is a $d$-regular sequence of measures on $E$. We use the following notation. For each $p\in (0, \infty]$, we set $L_p(\{\fm_k\}):=\cap_{k=0}^{\infty}L_p(\fm_k)$,  $L^{\loc}_p(\{\fm_k\}):=\cap_{k=0}^{\infty}L^{\loc}_p(\fm_k)$.
\par
The following lemmas record various properties of a $d$-regular sequence of measures $\{\fm_k\}$. The first important result is a relaxed version of the doubling property proved in \cite[Theorem~5.2]{tyul}.
\begin{Lm}
    \label{Lm.relaxed_doubling_property} 
    For each $c\ge 1$ there exist constants $C_1, C_2>0$ such that for each $k\in \mathbb{N}_0$,
    \begin{equation}
    \label{eq.relaxed_doubling_property}
        C_1\fm_k(Q_{k}(x)) \le \fm_k\left(\frac{1}{c}Q_{k}(x)\right) \le \fm_k(cQ_k(x)) \le C_2\fm_k(Q_k(x)) \quad \text{for all } x\in E.
    \end{equation}
\end{Lm}
\begin{Ca}
    \label{Ca.relaxed_doubling_property}
    Given $c\ge 1$, there exists $C>0$ such that,
    for each $k\in\mathbb{N}_0$,
    \begin{equation}
        \label{eq.corollary_relaxed_doubling_property}
        \fm_k(Q_k(x))
        \le \fm_k(cQ_k(y))
        \le C\fm_k(Q_k(x))
    \end{equation}
    whenever $x\in E$ and $|x-y|\le (c-1)2^{-k}$.
\end{Ca}
\begin{proof}
    Since $|x-y|\le (c-1)2^{-k}$, we have
    \begin{equation}
        Q_k(x)\subset cQ_k(y)\subset(2c-1)Q_k(x).
    \end{equation}
    Indeed, these inclusions follow from
    \begin{equation}
        2^{-k}+|x-y|\le c2^{-k},
        \qquad
        c2^{-k}+|x-y|\le(2c-1)2^{-k}.
    \end{equation}
    Therefore, by the monotonicity of $\fm_k$ and
    Lemma~\ref{Lm.relaxed_doubling_property},
    \begin{equation}
        \fm_k(Q_k(x))
        \le \fm_k(cQ_k(y))
        \le \fm_k((2c-1)Q_k(x))
        \lesssim \fm_k(Q_k(x)).
    \end{equation}
\end{proof}
A similar estimate, used repeatedly throughout the paper, reads as follows.
\begin{Ca}
    \label{Ca.M1_inflated}
    Given $c\ge 1$, there exists $C>0$ such that,
    for each $k\in\mathbb{N}_0$,
    \begin{equation}
        \label{eq.M1_inflated}
        \fm_k(cQ_k(x)) \le C2^{-kd} \qquad \text{for all } x\in\mathbb{R}^n.
    \end{equation}
\end{Ca}
\begin{proof}
    We assume that $cQ_{k}(x)\cap E\neq\emptyset$ since otherwise the assertion is obvious. Choose $y \in cQ_{k}(x)\cap E$; then $cQ_{k}(x) \subset 2cQ_{k}(y)$. Hence, using Lemma~\ref{Lm.relaxed_doubling_property} and property~\ref{cond:M1}, we obtain
    \begin{equation}
        \fm_k(cQ_k(x)) \le \fm_k(2cQ_{k}(y)) \lesssim \fm_k(Q_k(y)) \lesssim 2^{-kd}.
    \end{equation}
\end{proof}
Next, we record several properties related to the local best approximation by constants in $L_{\sigma}(\fm_k)$.
\begin{Lm}
    \label{Lm.local_approximation_outside_E}
    Let $\sigma\in(0,\infty)$, $c\ge1$, and
    $\underline{k}\in\mathbb{N}_0$. Then there exists a constant
    $C>0$ such that for each
    $\phi\in L_{\sigma}^{\loc}(\{\fm_k\})$,
    $k\ge\underline{k}$, $x\in E$, and
$y\in\mathbb R^n$ satisfying $|x-y|\le(c-1)2^{-k}$,
the following estimates hold:
    \begin{enumerate}
        \item 
        \begin{equation}
            \label{eq.local_approximation_outside_E_ordinary}
            \mathcal{E}_{\fm_k,\sigma}(\phi,Q_k(x))
            \le
            C\mathcal{E}_{\fm_{k-\underline{k}},\sigma}
            (\phi,cQ_k(y));
        \end{equation}
        \item 
        \begin{equation}
            \label{eq.local_approximation_outside_E}
            \mathcal{E}_{\fm_k,\sigma}(\phi,Q_k(x))
            \le
            C\widetilde{\mathcal{E}}_{\fm_{k-\underline{k}},\sigma}
            (\phi,cQ_k(y));
        \end{equation}
        \item 
        \begin{equation}
            \label{eq.local_approximation_outside_E_modified}
            \widetilde{\mathcal{E}}_{\fm_k,\sigma}(\phi,Q_k(x))
            \le
            C\widetilde{\mathcal{E}}_{\fm_{k-\underline{k}},\sigma}
            (\phi,cQ_k(y));
        \end{equation}
    \end{enumerate}
\end{Lm}
\begin{proof}
We prove the second assertion. The other estimates can be obtained by the same argument. Assume that $|x-y|\le(c-1)2^{-k}$. Then $x\in cQ_k(y)$, and hence $cQ_k(y)\cap E\neq\emptyset$. Therefore,
    \begin{equation}
        \widetilde{\mathcal{E}}_{\fm_{k-\underline{k}},\sigma}
        (\phi,cQ_k(y))
        =
        \mathcal{E}_{\fm_{k-\underline{k}},\sigma}
        (\phi,2cQ_k(y)).
    \end{equation}
    Moreover, Corollary~\ref{Ca.relaxed_doubling_property},
    applied with $2c$ in place of $c$, gives $\fm_k(Q_k(x))\approx\fm_k(2cQ_k(y))$. Set $m:=\operatorname{med}_{\fm_{k-\underline{k}}}(\phi,2cQ_k(y))$.
    Using the inclusion $Q_k(x)\subset 2cQ_k(y)$, property~\ref{cond:M3}, the preceding comparison of measures, and Lemma~\ref{Lm.average_median_approximation_property}, we obtain
    \begin{equation}
    \begin{split}
        \mathcal{E}_{\fm_k,\sigma}(\phi,Q_k(x))^\sigma&\le
        \fint\limits_{Q_k(x)}|\phi(z)-m|^\sigma d\fm_k(z)\lesssim
        \fint\limits_{2cQ_k(y)}
        |\phi(z)-m|^\sigma d\fm_{k-\underline{k}}(z)\\&\lesssim
        \mathcal{E}_{\fm_{k-\underline{k}},\sigma}
        (\phi,2cQ_k(y))^\sigma=
        \widetilde{\mathcal{E}}_{\fm_{k-\underline{k}},\sigma}
        (\phi,cQ_k(y))^\sigma.
    \end{split}
    \end{equation}
\end{proof}
\begin{Remark}
    \label{Rm.two_local_approximations_relation}
    In particular, by \eqref{eq.local_approximation_outside_E} in Lemma~\ref{Lm.local_approximation_outside_E}, applied with $c=1$, $\underline{k}=0$, and $y=x\in E$, we obtain,
    for $\sigma\in(0,\infty)$,
    \begin{equation}
        \label{eq.two_local_approximations_relation}
        \mathcal{E}_{\fm_k,\sigma}(\phi,Q_k(x))
        \lesssim
        \widetilde{\mathcal{E}}_{\fm_k,\sigma}(\phi,Q_k(x))
        \qquad
        \text{for all }(x,k)\in E\times\mathbb{N}_0.
    \end{equation}
\end{Remark}
\begin{Lm}
    \label{Lm.local_approximation_sequence_of_measures}
      Let $\lambda>1$, $\sigma \in (0, \infty)$. Assume that $\{\mathcal{T}_k^x\}_{k\in \mathbb{N}_0, x\in E}$ is a family of $\lambda$-almost best approximating operators on $L_{\sigma}(Q_k(x), \fm_k)$, i.e., $\mathcal{T}_k^x:L_{\sigma}(Q_k(x), \fm_k) \to \mathbb{R}$. Then there exist constants $C_1, C_2$ such that for each $\phi \in L_{\sigma}^{\loc}(\{\fm_k\})$, the following properties hold:
    \begin{enumerate}
        \item for each $k \in \mathbb{N}_0$ and $x\in E$
         \begin{equation}
         \label{eq.local_approximation_nested_cubes}
         \bigl|\mathcal{T}^x_k\phi-\mathcal{T}^x_{k+1}\phi\bigr| \le C_1 \mathcal{E}_{\fm_k, \sigma}(\phi, Q_k(x));
         \end{equation}
         \item for each $k \in \mathbb{N}$ and $x, y\in E$ satisfying $|x-y|\le 2^{-k}$
     \begin{equation}
         \label{eq.local_approximation_adjacent_cubes}
         \bigl|\mathcal{T}^x_k\phi-\mathcal{T}^y_k\phi\bigr| \le C_2 \mathcal{E}_{\fm_{k-1}, \sigma}(\phi, Q_{k-1}(x));
     \end{equation}
    \end{enumerate}
\end{Lm}
\begin{proof}
    Fix $k\in \mathbb{N}_0$ and $x\in E$. For every $z\in E$, we have
    \begin{equation}
    \label{eq.local_approximation_sequence_of_measures1}
         \bigl|\mathcal{T}^x_k\phi-\mathcal{T}^x_{k+1}\phi\bigr|^{\sigma} \lesssim  \bigl|\phi(z)-\mathcal{T}^x_k\phi\bigr|^{\sigma} +  \bigl|\phi(z)-\mathcal{T}^x_{k+1}\phi\bigr|^{\sigma}.
    \end{equation}
    Properties~\ref{cond:M1}--\ref{cond:M2} of $\{\fm_k\}$ give $\fm_{k+1}(Q_{k+1}(x))\approx \fm_{k}(Q_{k}(x))$. Therefore, averaging \eqref{eq.local_approximation_sequence_of_measures1} over $z \in Q_{k+1}(x)$ with respect to $\fm_{k+1}$ and using the inclusion $Q_{k+1}(x)\subset Q_{k}(x)$ and property~\ref{cond:M3} of $\{\fm_k\}$, we obtain
    \begin{equation}
    \begin{split}
        \bigl|\mathcal{T}^x_k\phi-\mathcal{T}^x_{k+1}\phi\bigr| \lesssim \Bigl(\fint\limits_{Q_{k+1}(x)}\bigl|\phi(z)-\mathcal{T}^x_{k+1}\phi\bigr|^{\sigma}d\fm_{k+1}(z) \Bigr)^{1/\sigma} +\\ \Bigl(\fint\limits_{Q_{k}(x)}\bigl|\phi(z)-\mathcal{T}^x_{k}\phi\bigr|^{\sigma}d\fm_{k}(z) \Bigr)^{1/\sigma} \lesssim \mathcal{E}_{\fm_{k+1}, \sigma}(\phi, Q_{k+1}(x))+\mathcal{E}_{\fm_{k}, \sigma}(\phi, Q_{k}(x))
    \end{split}
    \end{equation}
    Consequently, estimate \eqref{eq.local_approximation_outside_E_ordinary} in Lemma~\ref{Lm.local_approximation_outside_E}, applied with $c=2$,  $\underline{k}=1$, and $x=y$, gives \eqref{eq.local_approximation_nested_cubes}.
    \par
    For the second assertion, we fix $k\in\mathbb{N}$ and $x, y \in E$ satisfying $|x-y|\le2^{-k}$. We clearly have
    \begin{equation}
        |\mathcal{T}^x_k\phi - \mathcal{T}^y_{k}\phi|\le |\mathcal{T}^x_{k}\phi - \mathcal{T}^x_{k-1}\phi| + |\mathcal{T}^x_{k-1}\phi - \mathcal{T}^y_k\phi|.
    \end{equation}
    By \eqref{eq.local_approximation_nested_cubes}, the first term is bounded by $\mathcal{E}_{\fm_{k-1}, \sigma}(\phi, Q_{k-1}(x))$. For the second term, we obviously have
    \begin{equation}
        |\mathcal{T}^x_{k-1}\phi - \mathcal{T}^y_k\phi|\le |\phi(z)-\mathcal{T}^x_{k-1}\phi |+ |\phi(z) - \mathcal{T}^y_k\phi|
    \end{equation}
    for all $z\in Q_{k}(y)$. Averaging the latter estimate over $z\in Q_k(y)$ and using the inclusion $Q_{k}(y)\subset Q_{k-1}(x)$ and properties of $\{\fm_k\}$ as above yield \eqref{eq.local_approximation_adjacent_cubes}.
\end{proof}
\par
Following \cite{tyul}, we introduce the notion of a Lebesgue point associated with a sequence of measures. In what follows, for $r \in (0, 1]$, we let $k(r)\in \mathbb{N}_0$ be the unique integer such that $2^{-k(r)-1}<r\le 2^{-k(r)}$.
\begin{Def}
    Given $\phi\in L_{\sigma}^{\loc}(\{\fm_k\})$, $\sigma\in (0, \infty)$, we say that $x\in E$ is a $(\{\fm_k\}, \sigma)$-Lebesgue point of $\phi$ if
    \begin{equation}
        \label{eq.lebesgue_point_sequence_of_mesuares_definition}
        \lim_{r\to 0} \fint\limits_{Q_{r}(x)}|\phi(x)-\phi(y)|^{\sigma}d\fm_{k(r)}(y) = 0.
    \end{equation}
    The set of all $(\{\fm_k\}, \sigma)$-Lebesgue points of $\phi$ is denoted by $\mathfrak{L}_{\{\fm_k\}, \sigma}(\phi)$.
\end{Def}
\begin{Remark}
    \label{Rm.lebesgue_points_sequence_of_measures} Let $\sigma \in (0, \infty)$ and let $\phi \in L_{\sigma}^{\loc}(\{\fm_k\})$. It follows from Lemma~\ref{Lm.relaxed_doubling_property} that a point $x\in E$ is a $(\{\fm_k\}, \sigma)$-Lebesgue point of $\phi$ if and only if
    \begin{equation}
        \lim_{k\to\infty} \fint\limits_{Q_k(x)} |\phi(x) - \phi(y)|^{\sigma}d\fm_k(y) = 0.
    \end{equation}
\end{Remark}
The next lemma describes the behavior of almost best approximating constants at $(\{\fm_k\}, \sigma)$-Lebesgue points.
\begin{Lm}
    \label{Lm.approximation_at_Lebesgue_points}
   Let $\lambda>1$, $\sigma \in (0, \infty)$. Assume that $\{\mathcal{T}_k^x\}_{k\in \mathbb{N}_0, x\in E}$ is a family of $\lambda$-almost best approximating operators on $L_{\sigma}(Q_k(x), \fm_k)$. Then for each $\phi \in L_{\sigma}^{\loc}(\{\fm_k\})$ and $x\in \mathfrak{L}_{\{\fm_k\}, \sigma}(\phi)$,
    \begin{equation}
        \label{eq.approximation_at_Lebesgue_points}
        \lim_{k\to \infty}\mathcal{T}_k^x\phi = \phi(x).
    \end{equation}
\end{Lm}
\begin{proof}
Fix $\phi\in L_{\sigma}^{\loc}(\{\fm_k\})$ and $x\in E$.
    For every $k\in \mathbb{N}_0$ and $y\in Q_k(x)$, we clearly have
    \begin{equation}
        |\phi(x) - \mathcal{T}_k^x\phi|^{\sigma} \lesssim |\phi(y)-\mathcal{T}_k^x\phi|^{\sigma} + |\phi(y)-\phi(x)|^{\sigma}.
    \end{equation}
    Averaging the latter estimate over $y\in Q_k(x)$ with respect to $\fm_k$, we obtain
    \begin{equation}
    \begin{split}
        |\phi(x)- \mathcal{T}_k^x\phi| &\lesssim \mathcal{E}_{\fm_k, \sigma}(\phi, Q_k(x)) + \Bigl(\fint\limits_{Q_k(x)}|\phi(y) - \phi(x)|^{\sigma}d\fm_k(y)\Bigr)^{1/\sigma}\\&\lesssim \Bigl(\fint\limits_{Q_k(x)}|\phi(y) - \phi(x)|^{\sigma}d\fm_k(y)\Bigr)^{1/\sigma}
    \end{split}
    \end{equation}
    The last quantity tends to $0$ as $k\to \infty$ whenever $x\in \mathfrak{L}_{\{\fm_k\}, \sigma}(\phi)$. The proof is complete.
\end{proof}
\subsection{Traces}
Here we recall the notions of traces and trace-spaces. To this end, we define \emph{sharp representatives} of functions in Besov and Lizorkin--Triebel spaces.
\begin{Th}
\label{Th.sharp_representatives}
     Let $A\in \{B, F\}$. Let $s\in (0, 1)$, $p\in [1, \infty]$, and $ q\in (0, \infty]$. If $A=F$, assume also that $p<\infty$. For each $f\in A^s_{p, q}(\mathbb{R}^n)$, we set
    \begin{equation}
        \bar{f}(x):=\limsup_{r\to 0} \fint\limits_{Q_r(x)}f(y)dy.
    \end{equation}
    The following assertions hold:
    \begin{enumerate}
        \item if $s>\frac{n}{p}$, then $\bar{f}\in C(\mathbb{R}^n)$;
        \item if $p \le \frac{n}{s}$, then there exists $N_f\subset\mathbb{R}^n$ such that $\operatorname{dim}_H N_f \le n-ps$ and
    \begin{equation}
        \lim_{r\to0}\fint\limits_{Q_r(x)}|\bar{f}(x)-\bar{f}(y)|dy=0 \qquad \text{for all } x\in \mathbb{R}^n\setminus N_f.
    \end{equation}
    \end{enumerate}
\end{Th}
\begin{proof}
    The first assertion follows from the inclusion $A^s_{p, q}(\mathbb{R}^n)\subset C(\mathbb{R}^n)$ (see, for example, \cite[Section 2.8.3]{trieb}). The second assertion can be found in \cite[Lemma 3.1]{saks}. Notice that in the Lizorkin--Triebel case the restriction $q>\frac{n}{n+s}$ appearing there can be removed by the monotonicity of $F^s_{p, q}(\mathbb{R}^n)$ with respect to $q$.
\end{proof}
\par
By the Lebesgue differentiation theorem, $\bar{f}=f$ $\mathcal{L}^n$-almost everywhere, and hence $f$ and $\bar{f}$ define the same element of the corresponding Besov or Lizorkin--Triebel space. We refer to $\bar{f}$ as a \emph{sharp representative} of $f$. The interested reader can find a discussion of closely related questions concerning capacity and quasicontinuity in \cite{Karak, Netrusov, Netrusov_capacity, Nuutinen}.
\par
Having at our disposal Theorem~\ref{Th.sharp_representatives}, we introduce the concept of \emph{traces} and \emph{trace-spaces}. First, we remark that each function in $A^s_{p, q}(\mathbb{R}^n)$, $A\in \{B, F\}$, has a well-defined restriction to a $d$-thick set $E$, where $s>\frac{n-d}{p}$. More precisely, if $\{\fm_k\}_{k=0}^{\infty}$ is a $d$-regular sequence of measures, then condition~\ref{cond:M1} guarantees that $\fm_0 \ll \mathcal{H}^d$. Hence, by Theorem~\ref{Th.sharp_representatives}, for each $f\in A^s_{p, q}(\mathbb{R}^n)$, the pointwise restriction $\bar{f}\big|_E$ is well-defined $\fm_0$-almost everywhere. In what follows, $f\big|_{E}^{\fm_0}$ denotes the $\fm_0$-equivalence class of $\bar{f}\big|_E$. Equivalently, a Borel function $\phi:E\to\mathbb{R}$ is a trace of $f\in A^s_{p, q}(\mathbb{R}^n)$ if
\begin{equation}
\label{eq.trace_definition}
    \lim_{r\to 0}\fint\limits_{Q_r(x)}|f(y)-\phi(x)|dy = 0, \qquad \text{for $\fm_0$-a.e. } x\in E.
\end{equation}
In this case, the $\fm_0$-equivalence class of $\phi$ is equal to $f\big|_{E}^{\fm_0}$. The latter definition will also be used for locally integrable functions. That is, a Borel function $\phi:E \to \mathbb{R}$ is called a trace of $f\in L_1^{\loc}(\mathbb{R}^n)$ if \eqref{eq.trace_definition} holds.
\begin{Def}
\label{Def.trace_spaces}
    Let $E\subset \mathbb{R}^n$ be a closed $d$-thick set, $d\in [0, n]$, equipped with a $d$-regular sequence of measures $\{\fm_k\}_{k=0}^{\infty}$. Assume that $s\in(0,1)$, $p\in[1,\infty]$, $q\in(0,\infty]$, and $s>\frac{n-d}{p}$. Let $A\in \{B, F\}$, with the additional assumption $p<\infty$ when $A=F$. We define the trace-space $A^s_{p, q}(\mathbb{R}^n)\big|_E^{\fm_0}$ as the linear space of all $f\big|_E^{\fm_0}$, where $f\in A^s_{p, q}(\mathbb{R}^n)$. We equip this space with the quotient-space quasi-norm, i.e.,
    \begin{equation}
        \|\phi\|_{A^s_{p, q}(\mathbb{R}^n)\big|^{\fm_0}_E} := \inf \Bigl\{\|f\|_{A^s_{p, q}(\mathbb{R}^n)}:f\in A^s_{p, q}(\mathbb{R}^n), \phi=f\big|_E^{\fm_0} \Bigr\}, \qquad \phi \in A^s_{p, q}(\mathbb{R}^n)\big|_E^{\fm_0}.
    \end{equation}
    Furthermore, we define the trace operator by
    \begin{equation}
        \Tr\big|^{\fm_0}_E : A^s_{p, q}(\mathbb{R}^n)\to A^s_{p, q}(\mathbb{R}^n)\big|_E^{\fm_0}, \qquad f\mapsto f\big|_E^{\fm_0}.
    \end{equation}
\end{Def}
Finally, we introduce functionals of Besov and Lizorkin--Triebel type on a closed $d$-thick set $E$.
\begin{Def}
    Let $E\subset \mathbb{R}^n$ be a closed $d$-thick set, $d\in [0, n]$, equipped with a $d$-regular sequence of measures $\{\fm_k\}_{k=0}^{\infty}$.
    Given $p, q\in (0, \infty]$, $s\in (\frac{\theta}{p}, 1)$, and $\sigma \in (0, \infty)$, we set, for each $\phi\in L_{\sigma}^{\loc}(\{\fm_k\})$,
\begin{equation}
\label{eq.besov_seminorm_on_thick_set}
    \|\phi\|_{\mathfrak{b}^{s-\theta/p}_{p, q, \sigma}(E)}:= \|\{2^{ks}\widetilde{\mathcal{E}}_{\fm_k, \sigma}(\phi, Q_k(\cdot))\}_{k=0}^{\infty}\|_{\ell_q(L_p(\mathbb{R}^n))}.
\end{equation}
Furthermore, we set
\begin{equation}
\label{eq.besov_norm_on_thick_set}
    \|\phi\|_{\mathfrak{B}^{s-\theta/p}_{p, q, \sigma}(E)}:=\|\phi\|_{L_p(\fm_0)}+\|\phi\|_{\mathfrak{b}^{s-\theta/p}_{p, q, \sigma}(E)}.
\end{equation}
\end{Def}
\begin{Def}
    Let $E\subset \mathbb{R}^n$ be a closed $d$-thick set, $d\in [0, n]$, equipped with a $d$-regular sequence of measures $\{\fm_k\}_{k=0}^{\infty}$.
    Given $p\in (0, \infty)$, $q\in (0, \infty]$, $s\in (\frac{\theta}{p}, 1)$, and $\sigma \in (0, \infty)$, we set, for each $\phi\in L_{\sigma}^{\loc}(\{\fm_k\})$,
\begin{equation}
\label{eq.lizorkin_triebel_seminorm_on_thick_set}
    \|\phi\|_{\mathfrak{f}^{s-\theta/p}_{p, q, \sigma}(E)}:= \|\{2^{ks}\widetilde{\mathcal{E}}_{\fm_k, \sigma}(\phi, Q_k(\cdot))\}_{k=0}^{\infty}\|_{L_p(\mathbb{R}^n, \ell_q)}.
\end{equation}
Furthermore, we set
\begin{equation}
\label{eq.lizorkin_triebel_norm_on_thick_set}
    \|\phi\|_{\mathfrak{F}^{s-\theta/p}_{p, q, \sigma}(E)}:=\|\phi\|_{L_p(\fm_0)}+\|\phi\|_{\mathfrak{f}^{s-\theta/p}_{p, q, \sigma}(E)}.
\end{equation}
\end{Def}
\begin{Remark}
    It follows from Theorem~\ref{Th.main_stated_Besov} and Theorem~\ref{Th.main_stated_LT} that different admissible values of $\sigma$ give equivalent Besov and Lizorkin--Triebel functionals on a $d$-thick set. More precisely, different values of $0<\sigma\le p$ give equivalent Besov-type functionals and different values of $0<\sigma< \min\{p, q\}$ give equivalent Lizorkin--Triebel functionals.
\end{Remark}
In Section~\ref{section.Examples}, we show that these functionals recover the classical trace spaces on Ahlfors--David regular sets. The following elementary estimate will be useful throughout the paper.
\begin{Lm}
    \label{Lm.besov_and_lizorkin_triebel_rapid_convergence}
     Let $E\subset \mathbb{R}^n$ be a closed $d$-thick set, $d\in [0, n]$, equipped with a $d$-regular sequence of measures $\{\fm_k\}_{k=0}^{\infty}$. Given $s\in (0, 1)$, $p, q\in (0, \infty]$, $\alpha \in (-\infty, s)$, and $\sigma \in (0, \infty)$, there exists $C$ such that for each $\phi \in L_{\sigma}^{\loc}(\{\fm_k\})$ and each $k\in \mathbb{N}_0$
     \begin{equation}
         \label{eq.besov_rapid_convergence}
         \|\{2^{j\alpha}\widetilde{\mathcal{E}}_{\fm_j, \sigma}(\phi, Q_j(\cdot))\}_{j\ge k}\|_{\ell_1(L_p(\mathbb{R}^n))} \le C  2^{-k(s-\alpha)}\|\phi\|_{\mathfrak{b}^{s-\theta/p}_{p, q, \sigma}(E)}.
     \end{equation}
     Furthermore, if $p<\infty$, then
     \begin{equation}
         \label{eq.lizorkin_triebel_rapid_convergence}
          \|\{2^{j\alpha}\widetilde{\mathcal{E}}_{\fm_j, \sigma}(\phi, Q_j(\cdot))\}_{j\ge k}\|_{\ell_1(L_p(\mathbb{R}^n))} \le C  2^{-k(s-\alpha)}\|\phi\|_{\mathfrak{f}^{s-\theta/p}_{p, q, \sigma}(E)}.
     \end{equation}
\end{Lm}
\begin{proof}
    First assume that $q>1$. Since $\beta:=s-\alpha>0$, H\"older's inequality gives
    \begin{equation}
        \|\{2^{j\alpha}\widetilde{\mathcal{E}}_{\fm_j, \sigma}(\phi, Q_j(\cdot))\}_{j\ge k}\|_{\ell_1(L_p(\mathbb{R}^n))} \le \|\{2^{-\beta j}\}_{j\ge k}\|_{l_{q'}} \|\phi\|_{\mathfrak{b}^{s-\theta/p}_{p, q, \sigma}(E)} \approx 2^{-\beta k}\|\phi\|_{\mathfrak{b}^{s-\theta/p}_{p, q, \sigma}(E)}.
    \end{equation}
    If $q\le 1$, the latter estimate follows immediately from the continuous inclusion $\ell_q\subset \ell_1$.
    \par
    For the Lizorkin--Triebel case, we set $r:=\max\{p, q\}$. We claim that 
    \begin{equation}
    \label{eq.Lebesgue_points_on_E4}
        \|\{2^{js}\widetilde{\mathcal{E}}_{\fm_j, \sigma}(\phi, Q_j(\cdot))\}_{j\ge k}\|_{l_r(L_p(\mathbb{R}^n))} \le \|\{2^{js}\widetilde{\mathcal{E}}_{\fm_j, \sigma}(\phi, Q_j(\cdot))\}_{j\ge k}\|_{L_p(\mathbb{R}^n, \ell_q)}.
    \end{equation}
    Indeed, for $q\le p$, this estimate follows from the equality of norms $\ell_p(L_p(\mathbb{R}^n)) = L_p(\mathbb{R}^n, \ell_p)$ and the continuous inclusion $\ell_q\hookrightarrow \ell_p$, whereas for $q>p$, it follows from Minkowski's inequality for integrals. Arguing as in the Besov case, we obtain, for $r>1$,
    \begin{equation}
    \begin{split}
    &\|\{2^{j\alpha}\widetilde{\mathcal{E}}_{\fm_j, \sigma}(\phi, Q_j(\cdot))\}_{j\ge k}\|_{\ell_1(L_p(\mathbb{R}^n))} \lesssim \\\|\{2^{-\beta j}\}_{j\ge k}\|_{l_{r'}}& \|\{2^{js}\widetilde{\mathcal{E}}_{\fm_j, \sigma}(\phi, Q_j(\cdot))\}_{j\ge k}\|_{l_r(L_p(\mathbb{R}^n))}  \lesssim 2^{-\beta k}\|\phi\|_{\mathfrak{f}^{s-\theta/p}_{p, q, \sigma}(E)}.
    \end{split}
    \end{equation}
    If $r\le 1$, the latter estimate follows from the inclusion $\ell_r \hookrightarrow \ell_1$ and \eqref{eq.Lebesgue_points_on_E4}.
\end{proof}
\subsection{Whitney decomposition}
Another technical tool used extensively throughout the paper is the Whitney decomposition. For brevity, we set $Q^*:=\frac{9}{8}Q$ for each cube $Q\subset\mathbb{R}^n$.
\begin{Th}
\label{Th.Whitney_decomposition}
    For every nonempty closed set $E\subset\mathbb{R}^n$, there exists a family of dyadic cubes $W_{E}=\{Q_{\kappa}\}_{\kappa\in I} = \{Q_{r_{\kappa}}(x_{\kappa})\}_{\kappa\in I}$ such that
    \begin{conditions}{\textbf{W}}
        \item\label{cond:W1} $\mathbb{R}^n\setminus E = \bigcup_{\kappa \in I}Q_{\kappa}$;
        \item\label{cond:W2} for each $\kappa \in I$
        \begin{equation}
            \label{eq.whitney_decomposition_distance_property}
            \operatorname{diam}(Q_{\kappa}) \le \operatorname{dist}(Q_{\kappa}, E) \le 4 \operatorname{diam}(Q_{\kappa});
        \end{equation}
        \item\label{cond:W3} for each $\kappa_1, \kappa_2 \in I$, whenever $Q_{\kappa_1}^*\cap Q_{\kappa_2}^*\neq \emptyset$,
        \begin{equation}
            \label{eq.whitney_decomposition_diameters_of_adjacent_cubes}
            \frac{1}{4}\operatorname{diam}(Q_{\kappa_1})\le \operatorname{diam}(Q_{\kappa_2})\le 4\operatorname{diam}(Q_{\kappa_1});
        \end{equation}
        \item\label{cond:W4} there exists $C=C(n)>0$ such that for each $\kappa\in I$
        \begin{equation}
            \#\left\{\kappa'\in I: Q_{\kappa'}^*\cap Q_{\kappa}^*\neq \emptyset\right\}\le C(n);
        \end{equation}
        \item\label{cond:W5} for every distinct $\kappa_1, \kappa_2 \in I$, $\operatorname{int}Q_{\kappa_1}\cap\operatorname{int}Q_{\kappa_2}=\emptyset$.
    \end{conditions}
\end{Th}
We next recall a smooth partition of unity subordinate to a Whitney decomposition.
\begin{Prop}
    \label{Prop_whitney_decomposition_of_unity}
    Let $E\subset\mathbb{R}^n$ be a nonempty closed set, and let $W_E=\{Q_{\kappa}\}_{\kappa \in I}$ be a Whitney decomposition from Theorem~\ref{Th.Whitney_decomposition}. There exists a family of smooth functions $\{\psi_{\kappa}\}_{\kappa\in I}\subset C^{\infty}(\mathbb{R}^n)$ such that
    \begin{conditions}{\textbf{F}}
        \item\label{cond:F1} for each $\kappa \in I$ and $x\in \mathbb{R}^n$, $0\le \psi_{\kappa}(x)\le \chi_{Q^*_{\kappa}}(x)$;
        \item\label{cond:F2} $\sum_{\kappa\in I}\psi_{\kappa}(x)=1$ for all $x\in\mathbb{R}^n\setminus E$;
        \item\label{cond:F3} for each $\kappa\in I$, $\|\nabla \psi_{\kappa}\|_{L_{\infty}(\mathbb{R}^n)} \le C(n) \left(\operatorname{diam}(Q_{\kappa})\right)^{-1}$.
    \end{conditions}
\end{Prop}
Let $E$ be a nonempty closed set. Given $x\in\mathbb{R}^n\setminus E$, we say that $\widetilde{x}$ is a \emph{metric projection} of $x$ to $E$ if $\operatorname{dist}(x, E) = |x-\widetilde{x}|$. Such a projection need not be unique, but the particular choice will be immaterial. Furthermore, given a Whitney decomposition $W_{E} = \{Q_{\kappa}\}_{\kappa\in I}$, we associate with each $Q_{\kappa}$ the projected cube $\widetilde{Q}_{\kappa} := Q_{r_{\kappa}}(\widetilde{x}_{\kappa})$. Finally, for each $\kappa\in I$, we set $k(\kappa):=-\log_{2}r_{\kappa}\in\mathbb{Z}$.
For $x\in\mathbb{R}^n\setminus E$, we define the set of adjacent indexes by
\begin{equation}
    \label{eq.a_notation}
    a(x) := \left\{\kappa \in I: x \in Q_{\kappa}^*\right\}.
\end{equation}
We record a few simple estimates that will be useful in Section~\ref{section.extension}. 
\begin{Lm}
    \label{Lm.distance_estimates}
    Let $E\subset\mathbb{R}^n$ be a nonempty closed set, and let $W_E=\{Q_{\kappa}\}_{\kappa \in I}$ be a Whitney decomposition from Theorem~\ref{Th.Whitney_decomposition}.
    \begin{enumerate}
        \item For each $\kappa\in I$,
        \begin{equation}
        \label{eq.distance_to_projection}
            |x_{\kappa}-\widetilde{x}_{\kappa}| \le 9\cdot2^{-k(\kappa)};
        \end{equation}
        \item For each $x_1, x_2 \in \mathbb{R}^n\setminus E$ and $\kappa_i\in a(x_i)$, $i \in \{1, 2\}$,
        \begin{equation}
            \label{eq.distance_between_adjacent_projections}
            |\widetilde{x}_{\kappa_1}- \widetilde{x}_{\kappa_2}| \le 21\max\{2^{-k(\kappa_1)}, 2^{-k(\kappa_2)}\} + |x_1-x_2|;
        \end{equation}
        \item For each $k\in\mathbb{Z}$, $x\in\mathbb{R}^n\setminus E$ satisfying $\operatorname{dist}(x, E)<2^{-k}$, and $\kappa \in a(x)$, 
        \begin{equation}
            \label{eq.upper_diameter_estimate}
           2^{-k(\kappa)}<\frac{8}{15}2^{-k};
        \end{equation}
        \item For each $k\in\mathbb{Z}$, $x\in\mathbb{R}^n\setminus E$ satisfying $\operatorname{dist}(x, E)\ge 2^{-k}$, and $\kappa \in a(x)$,
        \begin{equation}
            \label{eq.lower_diameter_estimate}
            2^{-k(\kappa)}\ge \frac{4}{41}2^{-k}.
        \end{equation}
    \end{enumerate}
\end{Lm}
\begin{proof}
    To prove \eqref{eq.distance_to_projection}, we use property~\ref{cond:W2} of the Whitney decomposition. Then
    \begin{equation}
        |x_{\kappa} - \widetilde{x}_{\kappa}|\le \operatorname{dist}(Q_{\kappa}, E) + \frac{1}{2}\operatorname{diam}(Q_{\kappa}) \le \frac{9}{2}\operatorname{diam}(Q_{\kappa}) = 9\cdot2^{-k(\kappa)}.
    \end{equation}
    \par
    For \eqref{eq.distance_between_adjacent_projections}, we apply \eqref{eq.distance_to_projection} and take into account that $x_i\in Q_{\kappa_i}^* = \frac{9}{8}Q_{\kappa_i}$, $i\in \{1, 2\}$. Then
    \begin{equation}
    \begin{split}
        |\widetilde{x}_{\kappa_1}- \widetilde{x}_{\kappa_2}| &\le |\widetilde{x}_{\kappa_1}-x_{\kappa_1}|+|x_{\kappa_1}-x_1|+|x_1-x_2|+|x_2-x_{\kappa_2}|+|x_{\kappa_2}-\widetilde{x}_{\kappa_2}|\\&\le \left(9+\frac{9}{8}\right)2^{-k(\kappa_1)} + |x_1-x_2|+\left(9+\frac{9}{8}\right)2^{-k(\kappa_2)}\\&\le 21\max\{2^{-k(\kappa_1)}, 2^{-k(\kappa_2)}\} + |x_1-x_2|.
    \end{split}
    \end{equation}
    \par
    Finally, the upper estimate  \eqref{eq.upper_diameter_estimate} follows from property~\ref{cond:W2}. Indeed,
    \begin{equation}
        2^{-k}>\operatorname{dist}(x, E) \ge \operatorname{dist}(Q_{\kappa}, E) - \frac{1}{16}\operatorname{diam}(Q_{\kappa}) \ge \frac{15}{8}2^{-k(\kappa)}.
    \end{equation}
    On the other hand, by the property~\ref{cond:W2} of the Whitney decomposition, we have
    \begin{equation}
        4\operatorname{diam}(Q_{\kappa})\ge \operatorname{dist}(Q_{\kappa}, E) \ge \operatorname{dist}(x, E)-\frac{9}{8}\operatorname{diam}(Q_{\kappa}).
    \end{equation}
    If $\operatorname{dist}(x, E)\ge 2^{-k}$, then the preceding inequality yields
    \begin{equation}
        \frac{41}{8}\operatorname{diam}(Q_{\kappa}) = \frac{41}{4}2^{-k(\kappa)}\ge \operatorname{dist}(x, E)\ge2^{-k}
    \end{equation}
    Thus, the lower estimate \eqref{eq.lower_diameter_estimate} follows immediately.
\end{proof}
Finally, we recall the existence of a useful covering of $E$ related to the Whitney decomposition. Given $k\in \mathbb{Z}$, it is clear that the family $\{\widetilde{Q}_{\kappa}\}_{k(\kappa) = k}$ has uniformly bounded multiplicity. We want a stronger version of this family that has bounded multiplicity over all $\kappa$ (see \cite[Theorem~C]{tyulenev_thick_sets_R_n} or \cite[Theorem~2.4]{shvartsman_regular}).
\begin{Th}
    \label{Th.substitute_of_projections}
    Assume that $E \subset \mathbb{R}^n$ is Ahlfors--David $n$-regular and $W_{E} = \{Q_{\kappa}\}_{\kappa \in I}$ is the corresponding Whitney decomposition. Then there exist a family $\{U_{\kappa}\}_{\kappa \in I: r_{\kappa}\le 1}$ and constants $c_1, c_2>0$ with the following properties:
    \begin{enumerate}
        \item $U_{\kappa} \subset (10 Q_{\kappa}\cap E)$;
        \item $\mathcal{L}^n(Q_{\kappa}) \le c_1\mathcal{L}^n(U_{\kappa})$;
        \item $\sum\limits_{\kappa \in I: r_{\kappa}\le 1} \chi_{U_{\kappa}}(x)\le c_2$ for all $x\in E$.
    \end{enumerate}
\end{Th}
\section{Extension operator}
\label{section.extension}
Throughout this section we fix:
\begin{conditions}{\textbf{D}.3.}
    \item\label{cond:D31} a parameter $d\in [0, n]$ and a closed $d$-thick set $E\subset \mathbb{R}^n$;
    \item\label{cond:D32} a $d$-regular sequence of measures $\{\fm_k\}_{k=0}^{\infty}$ on $E$;
    \item\label{cond:D33} parameters $p\in [1, \infty]$, $q\in (0, \infty]$, and $s\in (\frac{n-d}{p}, 1)$.
\end{conditions}
\par
The aim of this section is to construct the extension operator and to prove its boundedness. To this end, we fix a Whitney decomposition $W_{E} = \{Q_{\kappa}\}_{\kappa\in I}$ obtained by applying Theorem~\ref{Th.Whitney_decomposition} to $E$. We are interested in the truncated version of the Whitney decomposition. More precisely, we put $\mathcal{J}:=\{\kappa\in I: r_{\kappa}\le 1\}$.
To treat the average and median constructions simultaneously, we introduce an \emph{extension operator associated with a family of almost best approximating operators}. 
\begin{Def}
    Let $\lambda>1$, $\sigma \in (0, \infty)$, and let $\mathcal{T} = \{\mathcal{T}_{k}^x\}_{k\in \mathbb{N}_0, x\in E}$ be a family of $\lambda$-almost best approximating operators such that $\mathcal{T}_{k}^x:L_{\sigma}(Q_k(x), \fm_k) \to \mathbb{R}$.
    Given $\phi\in L_{\sigma}^{\loc}(\{\fm_k\})$, we define $\Ext_{\mathcal{T}}\phi:\mathbb{R}^n\to\mathbb{R}$ as
    \begin{equation}
    \label{eq.extension_operator_definition}
        \Ext_{\mathcal{T}}\phi(x) := \chi_{E}(x)\phi(x)+ \sum_{\kappa\in \mathcal{J}}\psi_{\kappa}(x)\phi_{\kappa}^{\mathcal{T}}, \qquad x\in\mathbb{R}^n,
    \end{equation}
    where
    \begin{equation}
    \label{eq.extension_coefficients_definition}
    \phi_{\kappa}^{\mathcal{T}} = \mathcal{T}_{k(\kappa)}^{\widetilde{x}_{\kappa}}\phi \qquad \text{for all } \kappa \in \mathcal{J}.
    \end{equation}
\end{Def}
\begin{Remark}
    An extension operator associated with $\mathcal{T}$ need not be linear. However, if $\mathcal{T}_k^x$ is linear for each $k\in\mathbb{N}_0$ and $x\in E$, then $\Ext_{\mathcal{T}}$ is linear. We are mainly interested in two particular choices of the family $\mathcal{T}$. By Lemma~\ref{Lm.average_median_approximation_property}, for $\sigma\in [1, \infty)$ and for a suitable $\lambda=\lambda(\sigma)$, we can take $\mathcal{T}^x_k\phi := A_{k, \fm_k}\phi(x)$. Furthermore, for $\sigma \in (0, \infty)$ and for a suitable $\lambda = \lambda(\sigma)$, we can also take $\mathcal{T}^x_k\phi:= \operatorname{med}_{\fm_k}(\phi, Q_k(x))$.
\end{Remark}
\begin{Remark}
    \label{Rm.extension_support}
    Assume that $\kappa \in \mathcal{J}$ and $x\in Q_{\kappa}^*$. Then, by \eqref{eq.distance_to_projection},
    \begin{equation}
        \operatorname{dist}(x, E) \le |x-x_{\kappa}|+|x_{\kappa}-\widetilde{x}_{\kappa}| < 11r_{\kappa}.
    \end{equation}
    Since $r_\kappa\le1$ for $\kappa\in\mathcal{J}$, it follows that $\operatorname{supp}\Ext_{\mathcal{T}}\phi \subset U_{-4}(E)$.
\end{Remark}
\par
Throughout the rest of this section, we fix $\sigma \in (0, \infty)$, $\lambda>1$, and a family $\mathcal{T}=\{\mathcal{T}_k^x\}_{k\in\mathbb{N}_0, x\in E}$ of $\lambda$-almost best approximating operators. Furthermore, we fix $k_0\ge 10$, say $k_0:=10$. For $k\in\mathbb{Z}$, we set
\begin{equation}
    V_k(E):=U_k(E)\setminus U_{k+1}(E), \quad \widehat{V}_k(E) := V_{k-1}(E)\cup V_k(E)\cup V_{k+1}(E).
\end{equation}
Given $\phi \in L_{\sigma}^{\loc}(\{\fm_k\})$, we define, for each $k>k_0$
\begin{equation}
\label{eq.extension_fractional_gradient_1_part}
    g_k^{\sigma, I}(x):= 2^{ks}\chi_{U_{k-1}(E)}(x) \sum_{j=k-10}^{\infty}\widetilde{\mathcal{E}}_{\fm_j, \sigma}(\phi, Q_j(x)),
\end{equation}
\begin{equation}
\label{eq.extension_fractional_gradient_2_part}
    g_k^{\sigma, II}(x):= 2^{ks}\sum_{l=k_0}^{k-1}2^{l-k}\chi_{\widehat{V}_l(E)}(x) \sum_{j=l-10}^{\infty}\widetilde{\mathcal{E}}_{\fm_j, \sigma}(\phi, Q_j(x)),
\end{equation}
\begin{equation}
\label{eq.extension_fractional_gradient_3_part}
    g_k^{\sigma, III}(x):=2^{k(s-1)}\chi_{U_{-5}(E)}(x)\Bigl(A_{-10, \fm_0}|\phi|^{\sigma}(x)\Bigr)^{1/\sigma}.
\end{equation}
Our goal is to establish that, for a sufficiently large constant $C>0$, 
\begin{equation}
\label{eq.extension_fractional_gradient_statement}
    \{C(g_k^{\sigma, I} + g_k^{\sigma, II}+g_k^{\sigma, III})\}_{k>k_0} \in \mathbb{D}_{k_0}^s(\Ext_{\mathcal{T}} \phi).
\end{equation}
To this end, for each $k>k_0$, we need to provide a pointwise estimate for $|\Ext_{\mathcal{T}}\phi(x_1)-\Ext_{\mathcal{T}}\phi(x_2)|$ holding for $\mathcal{L}^n$-a.e. pair $(x_1, x_2)$ whenever $2^{-k-1}\le |x_1-x_2|<2^{-k}$. The argument depends on the relative positions of $x_1$, $x_2$, and $E$. After interchanging $x_1$ and $x_2$, if necessary, we are left with the following five cases.
\begin{enumerate}
    \item $x_1, x_2\in E$;
    \item $x_1\in E$, $x_2\notin E$;
    \item $x_1, x_2\notin E$, and at least one of them belongs to $ U_k(E)$;
    \item $x_1, x_2 \notin U_k(E)$, and at least one of them belongs to $U_{k_0}(E)$;
    \item $x_1, x_2 \in \mathbb{R}^n\setminus U_{k_0}(E)$.
\end{enumerate}
We present the proofs of all cases as lemmas below.
\subsection{Pointwise estimate for extension}
\begin{Lm}
    \label{Lm.extension_fractional_gradient_points_in_E}
    Let $\phi \in L_\sigma^{\loc}(\{\fm_k\})$ be such that $\mathcal{L}^n(E\setminus \mathfrak{L}_{\{\fm_k\}, \sigma}(\phi)) = 0$. Then there exists an $\mathcal{L}^n$-negligible set $N_1 \subset E$ such that, for each $k > k_0$ and all $x_1, x_2 \in E\setminus N_1$ satisfying $2^{-k-1}\le |x_1-x_2|<2^{-k}$, the following inequality holds:
    \begin{equation}
    \label{eq.extension_fractional_gradient_points_in_E}
        |\Ext_{\mathcal{T}} \phi(x_1) - \Ext_{\mathcal{T}}\phi(x_2)|\le C_1|x_1-x_2|^s\left(g_k^{\sigma, I}(x_1)+g_k^{\sigma, I}(x_2)\right),
    \end{equation}
    where $C_1$ is independent of $x_1, x_2$, $\phi$, and $k$.
\end{Lm}
\begin{proof}
By the definition of the extension operators $\Ext_{\mathcal{T}}\phi(x) = \phi(x)$ for all $x\in E$. Set $N_1:=E\setminus \mathfrak{L}_{\{\fm_k\}, \sigma}(\phi)$. Then for each $x \in  E\setminus N_1$, Lemma~\ref{Lm.approximation_at_Lebesgue_points} gives
    \begin{equation}
        \label{eq.extension_fractional_gradient_points_in_E1}
         \phi(x) = \lim_{j\to\infty}\mathcal{T}^x_j\phi.
    \end{equation}
    Fix $k>k_0$ and $x_1, x_2 \in E\setminus N_1$ such that $2^{-k-1}\le |x_1-x_2|<2^{-k}$. Using \eqref{eq.extension_fractional_gradient_points_in_E1} and the standard telescoping series, we obtain
    \begin{equation}
        \label{eq.telescoping_series_estimation_in_E}
        \begin{split}
        |\phi(x_1)-\phi(x_2)|\le \bigl|\mathcal{T}^{x_1}_k\phi-\mathcal{T}^{x_2}_k\phi\bigr|+\sum_{j=k}^{\infty} \bigl|\mathcal{T}^{x_1}_{j+1}\phi-\mathcal{T}^{x_1}_j\phi\bigr|+\sum_{j=k}^{\infty} \bigl|\mathcal{T}^{x_2}_{j+1}\phi-\mathcal{T}^{x_2}_j\phi\bigr|
        \end{split}
    \end{equation}
    Since $|x_1-x_2|<2^{-k}$, Lemma~\ref{Lm.local_approximation_sequence_of_measures} gives
    \begin{equation}
        |\phi(x_1)-\phi(x_2)| \lesssim \sum_{j=k-1}^{\infty}\mathcal{E}_{\fm_j, \sigma}(\phi, Q_j(x_1)) + \sum_{j=k-1}^{\infty}\mathcal{E}_{\fm_j, \sigma}(\phi, Q_j(x_2)).
    \end{equation}
    It remains to notice that $|x_1-x_2|\ge 2^{-k-1}$ yields $2^{-ks}\lesssim|x_1-x_2|^s$. Combining this observation with
   Remark~\ref{Rm.two_local_approximations_relation}, we obtain \eqref{eq.extension_fractional_gradient_points_in_E}. 
\end{proof}
\begin{Lm}
     \label{Lm.extension_fractional_gradient_points_mixed} 
     Let $\phi \in L_{\sigma}^{\loc}(\{\fm_k\})$ be such that $\mathcal{L}^n(E\setminus \mathfrak{L}_{\{\fm_k\}, \sigma}(\phi)) = 0$. Then there exists an $\mathcal{L}^n$-negligible set $N_2 \subset E$ such that, for each $k > k_0$ and all $x_1\in E\setminus N_2$, $x_2 \in \mathbb{R}^n\setminus E$ satisfying $2^{-k-1}\le |x_1-x_2|<2^{-k}$, the following inequality holds:
     \begin{equation}
    \label{eq.extension_fractional_gradient_points_mixed}
        |\Ext_{\mathcal{T}} \phi(x_1) - \Ext_{\mathcal{T}}\phi(x_2)|\le C_2|x_1-x_2|^s\left(g_k^{\sigma, I}(x_1)+g_k^{\sigma, I}(x_2)\right),
    \end{equation}
    where $C_2$ is independent of $x_1, x_2$, $\phi$, and $k$.
\end{Lm}
\begin{proof}
 Set $N_2:=E\setminus \mathfrak{L}_{\{\fm_k\}, \sigma}(\phi)$ as in the previous lemma. Fix $k> k_0$, $x_1\in E\setminus N_2$, and $x_2\in \mathbb{R}^n\setminus E$ such that $2^{-k-1}\le|x_1-x_2|<2^{-k}$. Since $x_1 \in E$ and $|x_1-x_2|<2^{-k}$, it follows that $x_2\in U_k(E)$. Therefore, for each $\kappa \in a(x_2)$, estimate \eqref{eq.upper_diameter_estimate} gives $k<k(\kappa)$. Since $k>k_0$, we then have $\sum_{\kappa \in \mathcal{J}}\psi_{\kappa}(x_2)=1$. Consequently,
    \begin{equation}
    \label{eq.extension_fractional_gradient_mixed0}
        |\Ext_{\mathcal{T}} \phi(x_1) - \Ext_{\mathcal{T}}\phi(x_2)| \le \sum_{\kappa \in a(x_2)}\psi_{\kappa}(x_2)\Bigl|\phi(x_1)-\phi_{\kappa}^{\mathcal{T}}\Bigr|.
    \end{equation}
    \par
    We estimate $\Bigl|\phi(x_1)-\phi_{\kappa}^{\mathcal{T}}\Bigr|$ for all $\kappa \in a(x_2)$. To this end, we observe that estimate \eqref{eq.distance_to_projection} implies
    \begin{equation}
        |x_1-\widetilde{x}_{\kappa}| \le |x_1-x_2|+|x_2-x_{\kappa}| +|x_{\kappa}-\widetilde{x}_{\kappa}|<2^{-k}+\frac{9}{8}2^{-k(\kappa)} + 9\cdot 2^{-k(\kappa)} < 2^{-k+3}.
    \end{equation}
    Since $x_1 \in \mathfrak{L}_{\{\fm_k\}, \sigma}(\phi)$, Lemma~\ref{Lm.approximation_at_Lebesgue_points} gives
    \begin{equation}
        \phi(x_1)=\mathcal{T}^{x_1}_{k-3}\phi+\sum_{j=k-3}^{\infty}\Bigl(\mathcal{T}^{x_1}_{j+1}\phi -\mathcal{T}^{x_1}_j\phi\Bigr).
    \end{equation}
    Consequently, by Lemma~\ref{Lm.local_approximation_sequence_of_measures}, we obtain
    \begin{equation}
    \label{eq.extension_fractional_gradient_mixed1}
    \begin{split}
        |\phi(x_1)-\phi_{\kappa}^{\mathcal{T}}| \lesssim |\mathcal{T}^{x_1}_{k-3}\phi - \phi_{\kappa}^{\mathcal{T}}| + \sum_{j=k-3}^{\infty} \mathcal{E}_{\fm_j, \sigma}(\phi, Q_j(x_1)).
    \end{split}
    \end{equation}
    \par
    Next, we expand $\phi_{\kappa}^{\mathcal{T}}$ by telescoping. More precisely, we clearly have
    \begin{equation}
        \phi_{\kappa}^{\mathcal{T}} = \mathcal{T}^{\widetilde{x}_{\kappa}}_{k-3}\phi+ \sum_{j=k-3}^{k(\kappa)-1}\Bigl(\mathcal{T}^{\widetilde{x}_{\kappa}}_{j+1}\phi-\mathcal{T}^{\widetilde{x}_{\kappa}}_j\phi\Bigr)
    \end{equation}
    Since $|x_1 - \widetilde{x}_{\kappa}|<2^{-k+3}$, Lemma~\ref{Lm.local_approximation_sequence_of_measures} then yields
    \begin{equation}
    \begin{split}
        |\mathcal{T}^{x_1}_{k-3}\phi - \phi_{\kappa}^{\mathcal{T}}|  \lesssim |\mathcal{T}^{x_1}_{k-3}\phi-\mathcal{T}^{\widetilde{x}_{\kappa}}_{k-3}\phi| + \sum_{j=k-3}^{k(\kappa)-1}\mathcal{E}_{\fm_j, \sigma}(\phi, Q_j(\widetilde{x}_{\kappa}))\lesssim \sum_{j=k-4}^{k(\kappa)-1}\mathcal{E}_{\fm_j, \sigma}(\phi, Q_j(\widetilde{x}_{\kappa})).
    \end{split}
    \end{equation}
    For each $j \in \{k-4, \ldots k(\kappa)-1\}$, we have $|x_2-\widetilde{x}_{\kappa}|\le 7\cdot 2^{-j}$. Indeed, since $x_2\in Q_{\kappa}^*$, we obtain, by \eqref{eq.distance_to_projection},
    \begin{equation}
        |\widetilde{x}_{\kappa}-x_2| \le |\widetilde{x}_{\kappa} - x_{\kappa}|+|x_{\kappa}-x_2|\le 11\cdot2^{-k(\kappa)} \le 7\cdot2^{-j}.
    \end{equation}
    Consequently, applying \eqref{eq.local_approximation_outside_E} in Lemma~\ref{Lm.local_approximation_outside_E} with $c=8$ and $\underline{k}=3$, we get
    \begin{equation}
    \label{eq.extension_fractional_gradient_mixed2}
        \sum_{j=k-4}^{k(\kappa)-1}\mathcal{E}_{\fm_j, \sigma}(\phi, Q_j(\widetilde{x}_{\kappa}))  \lesssim\sum_{j=k-7}^{\infty}\widetilde{\mathcal{E}}_{\fm_j, \sigma}(\phi, Q_j(x_2)).
    \end{equation}
    \par
    Combining \eqref{eq.extension_fractional_gradient_mixed1}-\eqref{eq.extension_fractional_gradient_mixed2} and Remark~\ref{Rm.two_local_approximations_relation} yields
    \begin{equation}
    \begin{split}
    |\phi(x_1)-\phi_{\kappa}^{\mathcal{T}}|&\le\sum_{j=k-3}^{\infty}\widetilde{\mathcal{E}}_{\fm_j, \sigma}(\phi, Q_j(x_1))+\sum_{j=k-7}^{\infty}\widetilde{\mathcal{E}}_{\fm_j, \sigma}(\phi, Q_j(x_2))\\& \lesssim 2^{-ks} (g_k^{\sigma, I}(x_1)+g_k^{\sigma, I}(x_2)).
    \end{split}
    \end{equation}
    By \eqref{eq.extension_fractional_gradient_mixed0} and property~\ref{cond:F2} of $\{\psi_{\kappa}\}$, we then have
    \begin{equation}
        |\Ext_{\mathcal{T}}\phi(x_1)-\Ext_{\mathcal{T}}\phi(x_2)|\lesssim 2^{-ks} (g_k^{\sigma, I}(x_1)+g_k^{\sigma, I}(x_2))
    \end{equation}
    Since $|x_1-x_2|\ge 2^{-k-1}$ yields $2^{-ks}\lesssim|x_1-x_2|^s$, we obtain \eqref{eq.extension_fractional_gradient_points_mixed}. The proof is complete.
\end{proof}
\begin{Lm}
     \label{Lm.extension_fractional_gradient_close_to_E}
     Let $\phi \in L_{\sigma}^{\loc}(\{\fm_k\})$. Then for each $k > k_0$ and all $x_1, x_2 \in \mathbb{R}^n\setminus E$ satisfying $2^{-k-1}\le |x_1-x_2|<2^{-k}$ and such that either $x_1 \in U_{k}(E)$ or $x_2\in U_{k}(E)$, the following inequality holds:
     \begin{equation}
     \label{eq.extension_fractional_gradient_points_close_to_E}
         |\Ext_{\mathcal{T}} \phi(x_1) - \Ext_{\mathcal{T}}\phi(x_2)|\le C_3|x_1-x_2|^s\left(g_k^{\sigma, I}(x_1)+g_k^{\sigma, I}(x_2)\right),
    \end{equation}
    where $C_3$ is independent of $x_1, x_2$, $\phi$, and $k$.
\end{Lm}
\begin{proof}
Fix $k> k_0$ and $x_1, x_2\in \mathbb{R}^n\setminus E$ such that $2^{-k-1}\le|x_1-x_2|<2^{-k}$. By symmetry, we may assume that $x_1\in U_{k}(E)$, and hence $x_2\in U_{k-1}(E)$. Therefore, for each $\kappa_i\in a(x_i)$, $i\in\{1, 2\}$, by \eqref{eq.upper_diameter_estimate}, we have $k<k(\kappa_1)$ and $k\le k(\kappa_2)$. Since $k>k_0$, we then have $\sum_{\kappa \in \mathcal{J}}\psi_{\kappa}(x_1) = \sum_{\kappa\in\mathcal{J}}\psi_{\kappa}(x_2) = 1$. Consequently, 
    \begin{equation}
    \label{eq.extension_fractional_gradient_points_close_to_E0}
        |\Ext_{\mathcal{T}} \phi(x_1) - \Ext_{\mathcal{T}}\phi(x_2)|\le \sum_{\kappa_1 \in a(x_1)}\sum_{\kappa_2 \in a(x_2)}\psi_{\kappa_1}(x_1)\psi_{\kappa_2}(x_2)\Bigl|\phi_{\kappa_1}^{\mathcal{T}}-\phi_{\kappa_2}^{\mathcal{T}}\Bigr|.
    \end{equation}
    \par
    Fix $\kappa_i\in a(x_i)$, $i\in \{1, 2\}$. We estimate $\Bigl|\phi_{\kappa_1}^{\mathcal{T}}-\phi_{\kappa_2}^{\mathcal{T}}\Bigr|$. Using \eqref{eq.distance_between_adjacent_projections}, we obtain
    \begin{equation}
    \begin{split}
        |\widetilde{x}_{\kappa_1} - \widetilde{x}_{\kappa_2}|< 21\max\{2^{-k(\kappa_1)}, 2^{-k(\kappa_2)}\}+|x_1-x_2|< 22\cdot 2^{-k}<2^{-k+5}. 
        \end{split}
    \end{equation}
    Furthermore, for $i\in \{1, 2\}$, we clearly have
    \begin{equation}
    \phi_{\kappa_i}^{\mathcal{T}}=\mathcal{T}^{\widetilde{x}_{\kappa_i}}_{k-5}\phi + \sum_{j=k-5}^{k(\kappa_i)-1}\Bigl(\mathcal{T}^{\widetilde{x}_{\kappa_i}}_{j+1}\phi-\mathcal{T}^{\widetilde{x}_{\kappa_i}}_j\phi\Bigr).
    \end{equation}
    Since $|\widetilde{x}_{\kappa_1} - \widetilde{x}_{\kappa_2}|<2^{-k+5}$, Lemma~\ref{Lm.local_approximation_sequence_of_measures} gives
    \begin{equation}
    \begin{split}
        |\phi_{\kappa_1}^{\mathcal{T}}-\phi_{\kappa_2}^{\mathcal{T}}|\lesssim \sum_{j=k-6}^{k(\kappa_1)-1}\mathcal{E}_{\fm_j, \sigma}(\phi, Q_j(\widetilde{x}_{\kappa_1})) + \sum_{j=k-6}^{k(\kappa_2)-1}\mathcal{E}_{\fm_j, \sigma}(\phi, Q_j(\widetilde{x}_{\kappa_2}))
    \end{split}
    \end{equation}
    \par
    Take an arbitrary $i\in\{1, 2\}$ and $j \in \{k-6, \ldots, k(\kappa_i)-1\}$. Since $x_i \in Q_{\kappa_i}^*$, we obtain, by \eqref{eq.distance_to_projection},
    \begin{equation}
        |x_i-\widetilde{x}_{\kappa_i}| \le|x_i-x_{\kappa_i}| + |x_{\kappa_i}-\widetilde{x}_{\kappa_i}|<  11\cdot 2^{-k(\kappa_i)}\le 7\cdot2^{-j}
    \end{equation}
    Therefore, estimate \eqref{eq.local_approximation_outside_E} in Lemma~\ref{Lm.local_approximation_outside_E}, applied with $c=8$ and $\underline{k}=3$, gives
    \begin{equation}
         |\phi_{\kappa_1}^{\mathcal{T}}-\phi_{\kappa_2}^{\mathcal{T}}|\lesssim \sum_{i=1}^2\sum_{j=k-9}^{\infty}\widetilde{\mathcal{E}}_{\fm_j, \sigma}(\phi, Q_j(x_i)).
    \end{equation}
    \par
    By \eqref{eq.extension_fractional_gradient_points_close_to_E0} and property~\ref{cond:F2} of $\{\psi_{\kappa}\}$, we then have
    \begin{equation}
    \begin{split}
        |\Ext_{\mathcal{T}} \phi(x_1) - \Ext_{\mathcal{T}}\phi(x_2)| \lesssim \sum_{i=1}^2\sum_{j=k-9}^{\infty}\widetilde{\mathcal{E}}_{\fm_j, \sigma}(\phi, Q_j(x_i)) \lesssim2^{-ks}(g_k^{\sigma, I}(x_1)+g_k^{\sigma, I}(x_2 )).
    \end{split}
    \end{equation}
     Since $|x_1-x_2|\ge 2^{-k-1}$ yields $2^{-ks}\lesssim|x_1-x_2|^s$, we obtain \eqref{eq.extension_fractional_gradient_points_close_to_E}. The proof is complete.
\end{proof}
\begin{Lm}
     \label{Lm.extension_fractional_gradient_far_from_E}
     Let $\phi \in L_{\sigma}^{\loc}(\{\fm_k\})$. Then for each $k > k_0$ and all $x_1, x_2 \in \mathbb{R}^n\setminus U_k(E)$ satisfying $2^{-k-1}\le |x_1-x_2|<2^{-k}$ and such that either $x_1\in U_{k_0}(E)\setminus U_k(E)$ or $x_2\in U_{k_0}(E)\setminus U_k(E)$, the following inequality holds:
     \begin{equation}
    \label{eq.extension_fractional_gradient_points_far_from_E}
        |\Ext_{\mathcal{T}} \phi(x_1) - \Ext_{\mathcal{T}}\phi(x_2)|\le C_4|x_1-x_2|^s\left(g_k^{\sigma, II}(x_1)+g_k^{\sigma, II}(x_2)\right),
    \end{equation}
    where $C_4$ is independent of $x_1, x_2, \phi$, and $k$.
\end{Lm}
\begin{proof}
Fix $k> k_0$ and $x_1,x_2\in \mathbb{R}^n\setminus U_k(E)$ such that $2^{-k-1}\le|x_1-x_2|<2^{-k}$. By symmetry, we may assume that $x_1\in U_{k_0}(E)$. Since $U_{k_0}(E)\setminus U_{k}(E) = \bigcup_{l=k_0}^{k-1}V_l(E)$, we choose $l\in [k_0, k-1]$ such that $x_1\in V_l(E)$. Since $|x_1-x_2|<2^{-k}$ and $l < k$, we have $x_2 \in \widehat{V}_l(E)$. Thus, $\operatorname{dist}(x_i, E)< 2^{-l+1}$, $i\in \{1, 2\}$, and hence, for every $\kappa \in a(x_1)\cup a(x_2)$, estimate \eqref{eq.upper_diameter_estimate} gives $k(\kappa) \ge l\ge k_0$.  Consequently, $\sum_{\kappa\in\mathcal{J}}\psi_{\kappa}(x_1) = \sum_{\kappa\in\mathcal{J}}\psi_{\kappa}(x_2)=1$. Furthermore, $\operatorname{dist}(x_i, E)\ge2^{-l-2}$, $i \in \{1, 2\}$, and therefore, by \eqref{eq.lower_diameter_estimate}, $k(\kappa)\le l+5$ for every $\kappa \in a(x_1)\cup a(x_2)$.
    \par
    Fix $\kappa_0 \in a(x_1)$. Since $k(\kappa)\in \{l, \ldots, l+5\}$ for all $\kappa \in a(x_1)\cup a(x_2)$, by the mean value theorem and property~\ref{cond:F3} of $\{\psi_{\kappa}\}$, we get
    \begin{equation}
    \label{eq.extension_fractional_gradient_points_far_from_E1}
    \begin{split}
        |\Ext_{\mathcal{T}} \phi(x_1) - \Ext_{\mathcal{T}}\phi(x_2)| &= \Bigl|\sum_{\kappa\in\mathcal{J}}(\psi_{\kappa}(x_1)-\psi_{\kappa}(x_2))(\phi_{\kappa}^{\mathcal{T}}-\phi_{\kappa_0}^\mathcal{T})\Bigr| \\ &\lesssim |x_1-x_2|2^{l}\sum_{\kappa \in a(x_1)\cup a(x_2)}|\phi_{\kappa}^\mathcal{T}-\phi_{\kappa_0}^\mathcal{T}|.
    \end{split}
    \end{equation}
    Take an arbitrary $\kappa \in a(x_1)\cup a(x_2)$. Our goal is to estimate $|\phi_{\kappa}^\mathcal{T}-\phi_{\kappa_0}^\mathcal{T}|$. By \eqref{eq.distance_between_adjacent_projections}, we have
    \begin{equation}
        |\widetilde{x}_{\kappa} - \widetilde{x}_{\kappa_0}| \le 21\max\{2^{-k(\kappa)}, 2^{-k(\kappa_0)}\} + |x_1 - x_2| < 43\cdot 2^{-l}<2^{-l+6}.
    \end{equation}
    From the definition of $\phi^{\mathcal{T}}_{\kappa}$, we obtain
    \begin{equation}
        \phi^{\mathcal{T}}_{\kappa} = \mathcal{T}^{\widetilde{x}_{\kappa}}_{l-6}\phi + \sum_{j=l-6}^{k(\kappa)-1}\Bigl(\mathcal{T}^{\widetilde{x}_{\kappa}}_{j+1}\phi-\mathcal{T}^{\widetilde{x}_{\kappa}}_j\phi\Bigr).
    \end{equation}
    The same representation holds for $\phi^{\mathcal{T}}_{\kappa_0}$. Since $|\widetilde{x}_{\kappa}-\widetilde{x}_{\kappa_0}|<2^{-l+6}$, Lemma~\ref{Lm.local_approximation_sequence_of_measures} then gives
    \begin{equation}
    \label{eq.extension_fractional_gradient_points_far_from_E2}
    \begin{split}
        |\phi_{\kappa}^{\mathcal{T}}-\phi_{\kappa_0}^{\mathcal{T}}| \lesssim \sum_{j=l-7}^{k(\kappa)-1}\mathcal{E}_{\fm_j, \sigma}(\phi, Q_j(\widetilde{x}_{\kappa})) + \sum_{j=l-7}^{k(\kappa_0)-1}\mathcal{E}_{\fm_j, \sigma}(\phi, Q_{j}(\widetilde{x}_{\kappa_0})).
    \end{split}
    \end{equation}
     Now, for each $\kappa_i\in a(x_i)$, $i\in \{1, 2\}$, (including $\kappa_i=\kappa_0$) and each $j\in \{l-7, \ldots, k(\kappa_i)-1\}$, by \eqref{eq.distance_to_projection} and $x_i\in Q_{\kappa_i}^*$, we obtain
    \begin{equation}
         |\widetilde{x}_{\kappa_i}-x_i| \le |\widetilde{x}_{\kappa_i} - x_{\kappa_i}|+|x_{\kappa_i}-x_i|<7\cdot 2^{-j}.
    \end{equation}
    Therefore, by \eqref{eq.local_approximation_outside_E} in Lemma~\ref{Lm.local_approximation_outside_E}, applied with $c=8$ and $\underline{k}=3$,
    \begin{equation}
    \label{eq.extension_fractional_gradient_points_far_from_E3}
        \sum_{j=l-7}^{k(\kappa_i)-1}\mathcal{E}_{\fm_j, \sigma}(\phi, Q_j(\widetilde{x}_{\kappa_i})) \lesssim \sum_{j=l-10}^{\infty} \widetilde{\mathcal{E}}_{\fm_j, \sigma}(\phi, Q_{j}(x_i)).
    \end{equation}
    Using \eqref{eq.extension_fractional_gradient_points_far_from_E1}, \eqref{eq.extension_fractional_gradient_points_far_from_E2}, \eqref{eq.extension_fractional_gradient_points_far_from_E3}, and property~\ref{cond:W4} of the Whitney decomposition, we obtain
    \begin{equation}
        |\Ext_{\mathcal{T}} \phi(x_1) - \Ext_{\mathcal{T}}\phi(x_2)|\le 2^{l-k}\sum_{i=1}^2\sum_{j=l-10}^{\infty}\widetilde{\mathcal{E}}_{\fm_j, \sigma}(\phi, Q_j(x_i)).
    \end{equation}
    Since $x_1\in V_l(E)\subset\widehat{V}_l(E)$ and $x_2\in \widehat{V}_l(E)$, the $l$-th term in \eqref{eq.extension_fractional_gradient_2_part} is nonzero. Thus, from $2^{-ks}\lesssim |x_1-x_2|^s$, we obtain \eqref{eq.extension_fractional_gradient_points_far_from_E} completing the proof.
\end{proof}
\begin{Lm}
     \label{Lm.extension_fractional_gradient_very_far_from_E}
     Let $\phi \in L_{\sigma}^{\loc}(\{\fm_k\})$. Then for each $k > k_0$ and all $x_1, x_2 \in \mathbb{R}^n\setminus U_{k_0}(E)$ satisfying $2^{-k-1}\le |x_1-x_2|<2^{-k}$, the following inequality holds:
     \begin{equation}
    \label{eq.extension_fractional_gradient_points_very_far_from_E}
        |\Ext_{\mathcal{T}} \phi(x_1) - \Ext_{\mathcal{T}}\phi(x_2)|\le C_5|x_1-x_2|^s\left(g_k^{\sigma, III}(x_1)+g_k^{\sigma, III}(x_2)\right),
    \end{equation}
    where $C_5$ is independent of $x_1, x_2, \phi$, and $k$.
\end{Lm}
\begin{proof}
    Fix $k> k_0$ and $x_1, x_2\in \mathbb{R}^n\setminus U_{k_0}(E)$ such that $2^{-k-1}\le|x_1-x_2|<2^{-k}$. By the mean value theorem and property~\ref{cond:F3} of $\{\psi_{\kappa}\}$,
    \begin{equation}
    \label{eq.extension_fractional_gradient_points_very_far_from_E0}
        |\Ext_{\mathcal{T}} \phi(x_1) - \Ext_{\mathcal{T}}\phi(x_2)|\le |x_1-x_2| \sum_{\kappa \in \mathcal{J}\cap(a(x_1)\cup a(x_2))}2^{k(\kappa)}|\phi_{\kappa}^{\mathcal{T}}|.
    \end{equation}
    For each $\kappa \in \mathcal{J}$ and $y\in \widetilde{Q}_{\kappa}$, we have
    \begin{equation}
        |\phi_{\kappa}^{\mathcal{T}}|^{\sigma} \lesssim |\mathcal{T}_{k(\kappa)}^{\widetilde{x}_{\kappa}}\phi - \phi(y)|^{\sigma}+|\phi(y)|^{\sigma}.
    \end{equation}
     Therefore, averaging the latter estimate over $y\in \widetilde{Q}_{\kappa}$ with respect to $\fm_{k(\kappa)}$, we obtain
    \begin{equation}
    \label{eq.extension_fractional_gradient_points_very_far_from_E1}
        |\phi_{\kappa}^{\mathcal{T}}| \lesssim \mathcal{E}_{{\fm_{k(\kappa)}}, \sigma}(\phi, \widetilde{Q}_{\kappa}) + \Bigl(A_{k(\kappa), \fm_{k(\kappa)}}|\phi|^{\sigma}(\widetilde{x}_{\kappa})\Bigr)^{1/\sigma} \lesssim \Bigl(A_{k(\kappa), \fm_{k(\kappa)}}|\phi|^{\sigma}(\widetilde{x}_{\kappa})\Bigr)^{1/\sigma}.
    \end{equation}
    Choose $l<k_0$ such that $x_1\in V_l(E)$, and hence $x_2\in \widehat{V}_l(E)$. Thus, $\operatorname{dist}(x_i, E)< 2^{-l+1}$, $i\in \{1, 2\}$, and, by \eqref{eq.upper_diameter_estimate}, $k(\kappa) \ge l$ for every $\kappa \in a(x_1)\cup a(x_2)$.  Furthermore, $\operatorname{dist}(x_i, E)\ge 2^{-l-2}$, and therefore, by \eqref{eq.lower_diameter_estimate}, $k(\kappa)\le l+5$ for every $\kappa \in a(x_1)\cup a(x_2)$. By Remark~\ref{Rm.extension_support}, $\Ext_{\mathcal{T}}\phi\equiv0$ outside $U_{-4}(E)$. Therefore, we may assume that $l\in \{-5, \ldots, k_0-1\}$. Since $k_0$ is fixed and $\max\{l, 0\}\le k(\kappa)\le l+5$, it follows that, for each $\kappa \in a(x_1)\cup a(x_2)$, $k(\kappa)$ ranges over a fixed finite set $\{0, \ldots, k_0+4\}$.
    \par
    Take an arbitrary $\kappa_i\in a(x_i)\cap \mathcal{J}$, $i\in\{1, 2\}$. By \eqref{eq.distance_to_projection} and Corollary~\ref{Ca.relaxed_doubling_property}, applied with $c=2^{10+k(\kappa_i)}\le 2^{k_0+15}$, $k=k(\kappa_i)$, $x=\widetilde{x}_{\kappa_i}$ and $y = x_i$, we then have 
    \begin{equation}
    \label{eq.solo_rev1}
        \fm_{k(\kappa_i)}(\widetilde{Q}_{\kappa_i})\approx \fm_{k(\kappa_i)}(Q_{-10}(x_i)).
    \end{equation}
    Furthermore, property~\ref{cond:M3} of $\{\fm_k\}$ allows us to replace $\fm_{k(\kappa_i)}$ in \eqref{eq.extension_fractional_gradient_points_very_far_from_E1} and \eqref{eq.solo_rev1} with $\fm_0$. Consequently,
    \begin{equation}
      |\phi_{\kappa_i}^{\mathcal{T}}| \lesssim \Bigl(A_{-10,\fm_0}|\phi|^{\sigma}(x_i)\Bigr)^{1/\sigma}.
    \end{equation}
    Combining the latter estimate, \eqref{eq.extension_fractional_gradient_points_very_far_from_E0}, property~\ref{cond:W4} of the Whitney decomposition with the fact that $k(\kappa)$: $0\le k(\kappa)<k_0+5$, we obtain
    \begin{equation}
    \begin{split}
        |\Ext_{\mathcal{T}}\phi(x_1)-\Ext_{\mathcal{T}}\phi(x_2)| &\lesssim 2^{-k}\sum_{i=1}^2\Bigl(A_{-10,\fm_0}|\phi|^{\sigma}(x_i)\Bigr)^{1/\sigma}\\&\lesssim|x_1-x_2|^s(g_k^{\sigma,III}(x_1)+g_k^{\sigma,III}(x_2)).
    \end{split}
    \end{equation}
    The proof is complete.
\end{proof}
\subsection{The boundedness of the extension operator}
Next, we estimate the $L_p(\mathbb{R}^n)$-norm of the extension. We recall the notation $\theta = n-d$, \eqref{eq.besov_norm_on_thick_set}, and \eqref{eq.lizorkin_triebel_norm_on_thick_set}.
\begin{Th}
    \label{Th.extension_L_p_estimate}
    Assume that $\sigma\le p$. Then there exists a constant $C>0$ such that for each $\phi \in L_{\sigma}^{\loc}(\{\fm_k\})$
    \begin{equation}
        \label{eq.extension_L_p_estimate}
        \|\Ext_{\mathcal{T}}\phi\|_{L_p(\mathbb{R}^n)}\le C\|\phi\|_{\mathfrak{B}^{s-\theta/p}_{p, q, \sigma}(E)}.
    \end{equation}
    Furthermore, if $p<\infty$ and $\sigma\le \min\{p, q\}$, then
    \begin{equation}
        \label{eq.extension_L_p_estimate_triebel}
        \|\Ext_{\mathcal{T}}\phi\|_{L_p(\mathbb{R}^n)}\le C\|\phi\|_{\mathfrak{F}^{s-\theta/p}_{p, q, \sigma}(E)}.
    \end{equation}
\end{Th}
\begin{proof}
Fix $\phi \in L_{\sigma}^{\loc}(\{\fm_k\})$.
We assume that the right-hand sides of \eqref{eq.extension_L_p_estimate} and \eqref{eq.extension_L_p_estimate_triebel} are finite, since otherwise the assertions are obvious. For brevity, we let
\begin{equation}
    \Phi_j(y):=\widetilde{\mathcal{E}}_{\fm_j, \sigma}(\phi, Q_j(y)), \qquad E_{j, \kappa}:={\mathcal{E}}_{\fm_j, \sigma}(\phi, Q_j(\widetilde{x}_{\kappa})) 
\end{equation}
for all $j\in \mathbb{N}_0$ and $\kappa\in \mathcal{J}$.
We clearly have
\begin{equation}
\label{eq.extension_L_p_estimate0}
     \|\Ext_{\mathcal{T}}\phi\|_{L_p(\mathbb{R}^n)} \approx  \|\Ext_{\mathcal{T}}\phi\|_{L_p(E)} +  \|\Ext_{\mathcal{T}}\phi\|_{L_p(\mathbb{R}^n\setminus E)}.
\end{equation}
We now prove \eqref{eq.extension_L_p_estimate}.
\par
\emph{Step 1.} We start with the first term. To this end, we remark that properties \ref{cond:M2} and \ref{cond:M3} yield (we recall the notation $k(r)$ for $r\in (0, 1]$)
\begin{equation}
    \fm_0(Q_r(x)) \gtrsim 2^{-k(r)\theta} \fm_{k(r)+1}(Q_r(x)) \gtrsim r^n \gtrsim \mathcal{L}^n\lfloor_E(Q_r(x))
\end{equation}
for all $(x, r) \in E \times (0, 1]$. By the standard density argument (see, e.g., \cite[Section~1.6]{Law}) and the Borel regularity of $\fm_0$,
the latter estimate extends to all Borel sets, and hence
\begin{equation}
\label{eq.extension_L_p_estimate1}
    \|\Ext_{\mathcal{T}}\phi\|_{L_p(E)} = \|\phi\|_{L_p(E)}\lesssim \|\phi\|_{L_p(\fm_0)}.
\end{equation}
\par
\emph{Step 2.} Now we estimate the second term in \eqref{eq.extension_L_p_estimate0}. First, assume that $p<\infty$. Then, by properties~\ref{cond:W4} of the Whitney decomposition and~\ref{cond:F1} of $\{\psi_{\kappa}\}$, we have
\begin{equation}
\label{eq.extension_L_p_estimate2}
    \|\Ext_{\mathcal{T}}\phi\|_{L_p({\mathbb{R}^n\setminus E})}^p \lesssim \sum_{\kappa \in \mathcal{J}}\mathcal{L}^n(Q_{\kappa}) |\phi_{\kappa}^{\mathcal{T}}|^p.
\end{equation}
For each $\kappa \in \mathcal{J}$, a telescoping argument, Lemma~\ref{Lm.local_approximation_sequence_of_measures}, and estimate \eqref{eq.extension_fractional_gradient_points_very_far_from_E1}, used in Lemma~\ref{Lm.extension_fractional_gradient_very_far_from_E}, give
\begin{equation}
\label{eq.extension_L_p_estimate3}
\begin{split}
    |\phi_{\kappa}^{\mathcal{T}}| &\le |\mathcal{T}^{\widetilde{x}_{\kappa}}_0\phi| + \sum_{j=0}^{k(\kappa)-1}\Bigl|\mathcal{T}^{\widetilde{x}_{\kappa}}_{j+1}\phi- \mathcal{T}^{\widetilde{x}_{\kappa}}_{j}\phi\Bigr|\\&\lesssim \Bigl(A_{0, \fm_0}|\phi|^{\sigma}(\widetilde{x}_{\kappa})\Bigr)^{1/\sigma} + \sum_{j=0}^{k(\kappa)-1}E_{j, \kappa}.
\end{split}
\end{equation}
\par
First, we consider the second term. For $\kappa \in \mathcal{J}$ and $j\le k(\kappa)-1$, the application of \eqref{eq.local_approximation_outside_E_ordinary} from Lemma~\ref{Lm.local_approximation_outside_E} with $c=2$ and $\underline{k}=0$ yields
\begin{equation}
    \label{eq.extension_L_p_estimate4}
    E_{j, \kappa} \lesssim \mathcal{E}_{\fm_j, \sigma}(\phi, 2Q_j(y)) = \Phi_j(y)
\end{equation}
for all $y\in Q_j(\widetilde{x}_{\kappa})$. 
Averaging over $y\in Q_{j}(\widetilde{x}_{\kappa})$ with respect to $\mathcal{L}^n$, we obtain
\begin{equation}
\label{eq.extension_L_p_estimate5}
    E_{j, \kappa}^p \lesssim \fint\limits_{Q_j(\widetilde{x}_{\kappa})}(\Phi_j(y))^pdy.
\end{equation}
Furthermore, Minkowski's inequality in the weighted $\ell_p$ space over $\kappa$ gives
\begin{equation}
 \label{eq.extension_L_p_estimate6}
\begin{split}
    \Bigl\{\sum_{\kappa\in\mathcal{J}}\mathcal{L}^n(Q_{\kappa}) \Bigl(\sum_{j=0}^{k(\kappa)-1}E_{j, \kappa}\Bigr)^p \Bigr\}^{1/p} &= \Bigl\{\sum_{\kappa\in\mathcal{J}}\mathcal{L}^n(Q_{\kappa}) \Bigl(\sum_{j=0}^{\infty}\chi_{\{j<k(\kappa)\}}E_{j, \kappa}\Bigr)^p \Bigr\}^{1/p}\\ &\le \sum_{j=0}^{\infty}\Bigl(\sum_{\substack{\kappa\in\mathcal{J}:\\ k(\kappa)>j}}\mathcal{L}^n(Q_{\kappa})E_{j, \kappa}^p\Bigr)^{1/p}.
\end{split}
\end{equation}
\par
We claim that, for each $j\in \mathbb{N}_0$ and $y\in \mathbb{R}^n$,
\begin{equation}
    \label{eq.whitney_packing_condition}
    \sum_{\substack{\kappa\in\mathcal{J}:\\ k(\kappa)\ge j}} \mathcal{L}^n(Q_{\kappa})\chi_{Q_j(\widetilde{x}_{\kappa})}(y)\lesssim 2^{-jn}.
\end{equation}
Indeed, if $y\in Q_j(\widetilde{x}_\kappa)$ and $k(\kappa)\ge j$, then the corresponding Whitney cube $Q_{\kappa}$ lies in one fixed enlargement of $Q_j(y)$. By property~\ref{cond:W5} the interiors of $Q_{\kappa}$, $\kappa\in \mathcal{J}$, are disjoint. Hence, we obtain \eqref{eq.whitney_packing_condition}. Combining this estimate with \eqref{eq.extension_L_p_estimate5} gives
\begin{equation}
    \sum_{\substack{\kappa\in\mathcal{J}:\\ k(\kappa)>j}}\mathcal{L}^n(Q_{\kappa})E_{j, \kappa}^p \lesssim 2^{jn}\int\limits_{\mathbb{R}^n}\Phi_j(y)^p\sum_{\substack{\kappa\in\mathcal{J}:\\ k(\kappa)>j}}\mathcal{L}^n(Q_{\kappa})\chi_{Q_j(\widetilde{x}_{\kappa})}(y)dy \lesssim \|\Phi_j\|_{L_p(\mathbb{R}^n)}^p.
\end{equation}
As a result, estimate \eqref{eq.extension_L_p_estimate6} and Lemma~\ref{Lm.besov_and_lizorkin_triebel_rapid_convergence}, applied with $\alpha = 0$, yield
\begin{equation}
 \label{eq.extension_L_p_estimate7}
    \Bigl\{\sum_{\kappa\in\mathcal{J}}\mathcal{L}^n(Q_{\kappa}) \Bigl(\sum_{j=0}^{k(\kappa)-1}E_{j, \kappa}\Bigr)^p \Bigr\}^{1/p} \lesssim \sum_{j=0}^{\infty}\|\Phi_j\|_{L_p(\mathbb{R}^n)} \lesssim \|\phi\|_{\mathfrak{b}^{s-\theta/p}_{p, q, \sigma}(E)}.
\end{equation}
\par
Now we consider the first term in \eqref{eq.extension_L_p_estimate3}.
Using \eqref{eq.whitney_packing_condition} with $j=0$ and Jensen's inequality (we recall that $\sigma\le p$), we get
\begin{equation}
 \label{eq.extension_L_p_estimate8}
    \sum_{\kappa \in \mathcal{J}}\mathcal{L}^n(Q_{\kappa})\Bigl(A_{0, \fm_0}|\phi|^{\sigma}(\widetilde{x}_{\kappa})\Bigr)^{p/\sigma} \lesssim  \sum_{\kappa \in \mathcal{J}}\mathcal{L}^n(Q_{\kappa})\int\limits_{Q_{0}(\widetilde{x}_{\kappa})}|\phi(y)|^pd\fm_0(y)\lesssim \|\phi\|_{L_p(\fm_0)}^p.
\end{equation}
Combining \eqref{eq.extension_L_p_estimate2}--\eqref{eq.extension_L_p_estimate3} and \eqref{eq.extension_L_p_estimate7}--\eqref{eq.extension_L_p_estimate8} yields
\begin{equation}
    \|\Ext_{\mathcal{T}}\phi\|_{L_p(\mathbb{R}^n\setminus E)}\lesssim \|\phi\|_{\mathfrak{B}^{s-\theta/p}_{p, q, \sigma}(E)}.
\end{equation}
\par
For $p=\infty$, the same estimate follows from
\begin{equation}
    |\phi^{\mathcal{T}}_{\kappa}| \lesssim \|\phi\|_{L_{\infty}(\fm_0)} + \sum_{j=0}^{\infty}\|\Phi_j\|_{L_{\infty}(\mathbb{R}^n)}\lesssim\|\phi\|_{\mathfrak{B}^{s}_{\infty, q, \sigma}(E)}.
\end{equation}
Together with \eqref{eq.extension_L_p_estimate1} this gives \eqref{eq.extension_L_p_estimate}.
\par
\emph{Step 3.} Finally, we prove \eqref{eq.extension_L_p_estimate_triebel}. The proof follows the same argument with one modification. We use the Lizorkin--Triebel case of Lemma~\ref{Lm.besov_and_lizorkin_triebel_rapid_convergence}, applied with $\alpha = 0$, at \eqref{eq.extension_L_p_estimate7}. Thus,
\begin{equation}
 \label{eq.extension_L_p_estimate_LT}
    \Bigl\{\sum_{\kappa\in\mathcal{J}}\mathcal{L}^n(Q_{\kappa}) \Bigl(\sum_{j=0}^{k(\kappa)-1}E_{j, \kappa}\Bigr)^p \Bigr\}^{1/p} \lesssim \sum_{j=0}^{\infty}\|\Phi_j\|_{L_p(\mathbb{R}^n)} \lesssim \|\phi\|_{\mathfrak{f}^{s-\theta/p}_{p, q, \sigma}(E)}.
\end{equation}
The remaining estimates follow in exactly the same way. The proof is complete.
\end{proof}

We are ready to prove the main results of this section. The \emph{first main result} reads as follows.
\begin{Th}
    \label{Th.main_inverse_Besov} Assume that $\sigma\le p$. Then there exists a constant $C>0$ such that for each $\phi \in L_{\sigma}^{\loc}(\{\fm_k\})$ with $\mathcal{L}^n(E\setminus\mathfrak{L}_{\{\fm_k\}, \sigma}(\phi)) = 0$,
    \begin{equation}
        \label{eq.main_inverse_Besov}
        \|\Ext_{\mathcal{T}}\phi\|_{B^s_{p, q}(\mathbb{R}^n)} \le C\|\phi\|_{\mathfrak{B}^{s-\theta/p}_{p,q, \sigma}(E)}.
    \end{equation}
\end{Th}
\begin{proof}
Fix $\phi \in L_{\sigma}^{\loc}(\{\fm_k\})$ such that $\mathcal{L}^n(E\setminus\mathfrak{L}_{\{\fm_k\}, \sigma}(\phi))=0$.
We assume that the right-hand side of \eqref{eq.main_inverse_Besov} is finite, since otherwise the assertion is obvious. First, we estimate the $\ell_q(L_p(\mathbb{R}^n))$-norm of the sequence $\{g_k^{\sigma} := g_k^{\sigma, I}+g_k^{\sigma,II}+g_k^{\sigma, III}\}_{k>k_0}$ constructed in \eqref{eq.extension_fractional_gradient_1_part}--\eqref{eq.extension_fractional_gradient_3_part}. Throughout the proof, put $e_k:=2^{ks}\|\widetilde{\mathcal{E}}_{\fm_k, \sigma}(\phi, Q_k(\cdot))\|_{L_p(\mathbb{R}^n)}$ for brevity.
\par
\emph{Step 1.} We start with the first term. By Minkowski's inequality and Corollary~\ref{Ca.geometric_discrete_convolution},
\begin{equation}
\label{eq.main_inverse_Besov1}
\begin{split}
        \|\{\|g_k^{\sigma, I}\|_{L_p(\mathbb{R}^n)}\}_{k>k_0}\|_{\ell_q}\lesssim\Bigl\|\Bigl\{\sum_{j=k-10}^{\infty}2^{-(j-k)s}e_j\Bigr\}_{k>k_0}\Bigr\|_{\ell_q} \lesssim \|\{e_j\}_{j\ge 0}\|_{\ell_q} = \|\phi\|_{\mathfrak{b}^{s-\theta/p}_{p, q, \sigma}(E)}.
\end{split}
    \end{equation}
    \emph{Step 2.} Now we consider $\{g_k^{\sigma, II}\}_{k>k_0}$. For each $k>k_0$, Minkowski's inequality gives
    \begin{equation}
    \label{eq.main_inverse_Besov2}
        \begin{split}
            2^{-ks}\|g_k^{\sigma, II}\|_{L_p(\mathbb{R}^n)} &\lesssim \sum_{l=k_0}^{k-1}2^{l-k}\sum_{j=l-10}^{\infty}2^{-js}e_j = \sum_{j=k_0-10}^{\infty}2^{-js}e_j\sum_{l=k_0}^{\min\{j+10, k-1\}}2^{l-k}\\&\approx \sum_{j=k_0-10}^{k-11}2^{j-k}2^{-js}e_j+\sum_{j=k-10}^{\infty}2^{-js}e_j.
        \end{split}
    \end{equation}
    Using Corollary~\ref{Ca.geometric_discrete_convolution}, we obtain
    \begin{equation}
    \label{eq.main_inverse_Besov3}
        \Bigl\|\Bigl\{2^{ks}\sum_{j=k-10}^{\infty}2^{-js}e_j\Bigr\}_{k>k_0}\Bigr\|_{\ell_q}\lesssim\|\{e_j\}_{j\ge 0}\|_{\ell_q}= \|\phi\|_{\mathfrak{b}^{s-\theta/p}_{p, q, \sigma}(E)}.
    \end{equation}
    Since for each $k>k_0$,
    \begin{equation}
        2^{ks}\sum_{j=k_0-10}^{k-11}2^{j-k}2^{-js}e_j = \sum_{j=k_0-10}^{k-11}2^{-(k-j)(1-s)}e_j,
    \end{equation}
     Corollary~\ref{Ca.geometric_discrete_convolution} gives
    \begin{equation}
        \Bigl\|\Bigl\{2^{ks}\sum_{j=k_0-10}^{k-11}2^{j-k}2^{-js}e_j\Bigr\}_{k>k_0}\Bigr\|_{\ell_q}\lesssim\|\{e_j\}_{j\ge 0}\|_{\ell_q}= \|\phi\|_{\mathfrak{b}^{s-\theta/p}_{p, q, \sigma}(E)}
    \end{equation}
     Combining the preceding estimates gives
    \begin{equation}
    \label{eq.main_inverse_Besov7}
        \|\{g_k^{\sigma, II}\}_{k>k_0}\|_{\ell_q(L_p(\mathbb{R}^n))}\lesssim\|\phi\|_{\mathfrak{b}^{s-\theta/p}_{p, q, \sigma}(E)}.
    \end{equation}
    \par
    \emph{Step 3.} We next estimate $\{g_k^{\sigma, III}\}_{k>k_0}$. Assume first that $p<\infty$. Since $\sigma\le p$, for each $k>k_0$, Fubini's theorem, Jensen's inequality, properties~\ref{cond:M1}--\ref{cond:M2}, and Corollary~\ref{Ca.relaxed_doubling_property} yield
    \begin{equation}
        \|g_k^{\sigma, III}\|_{L_p(\mathbb{R}^n)} \lesssim 2^{k(s-1)}\Bigl(\int\limits_{U_{-5}(E)} \fint\limits_{Q_{-10}(x)}|\phi(y)|^pd\fm_0(y)dx \Bigr)^{1/p} \lesssim 2^{k(s-1)}\|\phi\|_{L_p(\fm_0)}.
    \end{equation}
    The same estimate for $p=\infty$ is trivial. Therefore,
    \begin{equation}
        \label{eq.main_inverse_Besov8}
        \|\{g_k^{\sigma, III}\}_{k>k_0}\|_{\ell_q(L_p(\mathbb{R}^n))}\lesssim \|\phi\|_{L_p(\fm_0)}\|\{2^{k(s-1)}\}_{k>k_0}\|_{\ell_q}\lesssim\|\phi\|_{L_p(\fm_0)}.
    \end{equation}
    \par
    \emph{Step 4.} Let $C$ be the maximum of the constants from Lemmas~\ref{Lm.extension_fractional_gradient_points_in_E}--\ref{Lm.extension_fractional_gradient_very_far_from_E}. By these lemmas, $\{g_k := C(g_k^{\sigma, I}+g_k^{\sigma, II}+g_k^{\sigma, III})\}_{k>k_0}\in \mathbb{D}_{k_0}^s(\Ext_{\mathcal{T}}\phi)$. Hence, by \eqref{eq.main_inverse_Besov1}, \eqref{eq.main_inverse_Besov7}, and \eqref{eq.main_inverse_Besov8},
    \begin{equation}
    \label{eq.main_inverse_Besov9}
        \inf_{\vec{h}\in \mathbb{D}_{k_0}^s(\Ext_{\mathcal{T}}\phi)}\|\vec{h}\|_{\ell_q(L_p(\mathbb{R}^n))} \le \|\{g_k\}_{k>k_0}\|_{\ell_q(L_p(\mathbb{R}^n))} \lesssim \|\phi\|_{\mathfrak{B}^{s-\theta/p}_{p, q, \sigma}(E)}.
    \end{equation}
    \par
\emph{Step 5.} By Theorem~\ref{Th.extension_L_p_estimate},
\begin{equation}
    \label{eq.main_inverse_Besov15}
    \|\Ext_{\mathcal{T}}\phi\|_{L_p(\mathbb{R}^n)}\lesssim \|\phi\|_{\mathfrak{B}^{s-\theta/p}_{p, q, \sigma}(E)}.
\end{equation}
Using \eqref{eq.main_inverse_Besov9}, \eqref{eq.main_inverse_Besov15}, and Remark~\ref{Rm.fractional_gradient=_truncation}, we obtain
\begin{equation}
    \|\Ext_{\mathcal{T}}\phi\|_{B^s_{p, q}(\mathbb{R}^n)}\lesssim\|\phi\|_{\mathfrak{B}^{s-\theta/p}_{p, q, \sigma}(E)}.
\end{equation}
The proof is complete.
\end{proof}
The \emph{second main result} reads as follows.
\begin{Th}
    \label{Th.main_inverse_Lizorkin_Triebel} Assume that $p<\infty$ and $\sigma\le \min\{p, q\}$. There exists a constant $C>0$ such that for each $\phi \in L_{\sigma}^{\loc}(\{\fm_k\})$ with $\mathcal{L}^n(E\setminus\mathfrak{L}_{\{\fm_k\}, \sigma}(\phi)) = 0$,
    \begin{equation}
        \label{eq.main_inverse_Lizorkin_Triebel}
        \|\Ext_{\mathcal{T}}\phi\|_{F^s_{p, q}(\mathbb{R}^n)} \le C\|\phi\|_{\mathfrak{F}^{s-\theta/p}_{p,q, \sigma}(E)}.
    \end{equation}
\end{Th}
\begin{proof} The proof follows the arguments used in Theorem~\ref{Th.main_inverse_Besov}. We provide details for completeness.
Fix $\phi \in L_{\sigma}^{\loc}(\{\fm_k\})$ such that $\mathcal{L}^n(E\setminus\mathfrak{L}_{\{\fm_k\}, \sigma}(\phi))=0$.
We assume that the right-hand side of \eqref{eq.main_inverse_Lizorkin_Triebel} is finite, since otherwise the assertion is obvious. First, we estimate the $L_p(\mathbb{R}^n, \ell_q)$-norm of the sequence $\{g_k^{\sigma} := g_k^{\sigma, I}+g_k^{\sigma,II}+g_k^{\sigma, III}\}_{k>k_0}$ constructed in \eqref{eq.extension_fractional_gradient_1_part}--\eqref{eq.extension_fractional_gradient_3_part}. Throughout the proof, set $\Phi_k:=\widetilde{\mathcal{E}}_{\fm_k, \sigma}(\phi, Q_k(\cdot))$ for brevity.
\par
\emph{Step 1.} We start with the first term. By Corollary~\ref{Ca.geometric_discrete_convolution},
    \begin{equation}
    \label{eq.main_inverse_Lizorkin_Triebel1}
        \|
            \{g_k^{\sigma, I}(x)\}_{k>k_0}
        \|_{\ell_q}
        \lesssim
        \|
            \{2^{js}\Phi_j(x)\}_{j\ge0}
        \|_{\ell_q}
    \end{equation}
    for all $x\in \mathbb{R}^n$. Consequently,
    \begin{equation}
    \label{eq.main_inverse_Lizorkin_Triebel2}
        \|
            \{g_k^{\sigma, I}\}_{k>k_0}
        \|_{L_p(\mathbb{R}^n,\ell_q)}
        \lesssim
        \|\phi\|_{\mathfrak{f}^{s-\theta/p}_{p,q, \sigma}(E)}.
    \end{equation}
    \par
    \emph{Step 2.} Now we consider $\{g_k^{\sigma, II}\}_{k>k_0}$. For each $l \ge k_0$, define $b_l(x) := 2^{ls} \sum_{j=l-10}^{\infty}\Phi_j(x)$. Hence another pointwise application of Corollary~\ref{Ca.geometric_discrete_convolution} gives
    \begin{equation}
    \label{eq.main_inverse_Lizorkin_Triebel3}
        \left\|
            \{b_l(x)\}_{l\ge k_0}
        \right\|_{\ell_q}
        \lesssim
        \left\|
            \{2^{js}\Phi_j(x)\}_{j\ge0}
        \right\|_{\ell_q}
    \end{equation}
    for all $x\in\mathbb{R}^n$.
    On the other hand, the definition of $g_k^{\sigma, II}$ gives
    \begin{equation}
        g_k^{\sigma, II}(x)= \sum_{l=k_0}^{k-1} 2^{-(k-l)(1-s)}\chi_{\widehat{V}_l(E)}(x)b_l(x).
    \end{equation}
    Therefore, applying Corollary~\ref{Ca.geometric_discrete_convolution} once more, we obtain
    \begin{equation}
    \begin{split}
        \left\|
            \{g_k^{\sigma, II}(x)\}_{k>k_0}
        \right\|_{\ell_q}
        \lesssim
        \left\|
            \{
                \chi_{\widehat{V}_l(E)}(x)b_l(x)
            \}_{l\ge k_0}
        \right\|_{\ell_q}
        \\\le
        \left\|
            \{b_l(x)\}_{l\ge k_0}
        \right\|_{\ell_q}
        \lesssim
        \left\|
            \{2^{js}\Phi_j(x)\}_{j\ge0}
        \right\|_{\ell_q}.
    \end{split}
    \end{equation}
    Thus,
    \begin{equation}
    \label{eq.main_inverse_Lizorkin_Triebel4}
        \left\|
            \{g_k^{\sigma, II}\}_{k>k_0}
        \right\|_{L_p(\mathbb{R}^n,\ell_q)}
        \lesssim
        \|\phi\|_{\mathfrak{f}^{s-\theta/p}_{p,q, \sigma}(E)}.
    \end{equation}
    \par
    \emph{Step 3.} We next estimate $\{g_k^{\sigma, III}\}_{k>k_0}$. For each $x\in\mathbb{R}^n$,
    \begin{equation}
    \begin{split}
        \|\{g_k^{\sigma, III}(x)\}_{k>k_0}\|_{\ell_q}& \lesssim \|\{2^{k(s-1)}\}_{k>k_0}\|_{\ell_q} \chi_{U_{-5}(E)}(x) \Bigl(A_{-10, \fm_0}|\phi|^{\sigma}(x)\Bigr)^{1/\sigma}\\& \lesssim \chi_{U_{-5}(E)}(x) \Bigl(A_{-10, \fm_0}|\phi|^{\sigma}(x)\Bigr)^{1/\sigma}.
    \end{split}
    \end{equation}
     Since $\sigma\le p$, for each $k>k_0$, Fubini's theorem, Jensen's inequality, properties~\ref{cond:M1}--\ref{cond:M2}, and Corollary~\ref{Ca.relaxed_doubling_property} yield
    \begin{equation}
    \label{eq.main_inverse_Lizorkin_Triebel5}
        \|\{g_k^{\sigma, III}(x)\}_{k>k_0}\|_{L_p(\mathbb{R}^n,\ell_q)}^p \lesssim \int\limits_{U_{-5}(E)}\int\limits_{Q_{-10}(x)}|\phi(y)|^pd\fm_0(y)dx \lesssim \|\phi\|_{L_p(\fm_0)}^p.
    \end{equation}
    \par
    \emph{Step 4.} Let $C$ be the maximum of the constants from Lemmas~\ref{Lm.extension_fractional_gradient_points_in_E}--\ref{Lm.extension_fractional_gradient_very_far_from_E}. By these lemmas, $\{g_k := C(g_k^{\sigma, I}+g_k^{\sigma, II}+g_k^{\sigma, III})\}_{k>k_0}\in \mathbb{D}_{k_0}^s(\Ext_{\mathcal{T}}\phi)$. Hence, by \eqref{eq.main_inverse_Lizorkin_Triebel2}, \eqref{eq.main_inverse_Lizorkin_Triebel4}, and \eqref{eq.main_inverse_Lizorkin_Triebel5}
    \begin{equation}
    \label{eq.main_inverse_Lizorkin_Triebel6}
        \inf_{\vec{h}\in \mathbb{D}_{k_0}^s(\Ext_{\mathcal{T}}\phi)}\|\vec{h}\|_{L_p(\mathbb{R}^n, \ell_q)} \le \|\{g_k\}_{k>k_0}\|_{L_p(\mathbb{R}^n, \ell_q)} \lesssim \|\phi\|_{\mathfrak{F}^{s-\theta/p}_{p, q, \sigma}(E)}.
    \end{equation}
    \par
\emph{Step 5.} By Theorem~\ref{Th.extension_L_p_estimate},
\begin{equation}
    \label{eq.main_inverse_Lizorkin_Triebel7}
    \|\Ext_{\mathcal{T}}\phi\|_{L_p(\mathbb{R}^n)}\lesssim \|\phi\|_{\mathfrak{F}^{s-\theta/p}_{p, q, \sigma}(E)}.
\end{equation}
Using \eqref{eq.main_inverse_Lizorkin_Triebel6}, \eqref{eq.main_inverse_Lizorkin_Triebel7}, and Remark~\ref{Rm.fractional_gradient=_truncation}, we obtain
\begin{equation}
    \|\Ext_{\mathcal{T}}\phi\|_{F^s_{p, q}(\mathbb{R}^n)}\lesssim\|\phi\|_{\mathfrak{F}^{s-\theta/p}_{p, q, \sigma}(E)}.
\end{equation}
The proof is complete.
\end{proof}
\section{Lebesgue points and extension operator}
\label{section.extension_property}
Throughout this section we fix:
\begin{conditions}{\textbf{D}.4.}
    \item\label{cond:D41} a parameter $d\in [0,n]$ and a closed $d$-thick set
    $E\subset\mathbb{R}^n$;
    \item\label{cond:D42} a $d$-regular sequence of measures
    $\{\fm_k\}_{k=0}^{\infty}$ on $E$;
    \item\label{cond:D43} parameters $p\in [1,\infty]$, $q\in (0,\infty]$, and
    $s\in \left(\frac{n-d}{p},1\right)$.
\end{conditions}
\par
The goal of this section is to establish that functions for which the corresponding Besov or Lizorkin--Triebel functional is finite have $(\{\fm_k\}, \sigma)$-Lebesgue points $\fm_0$-almost everywhere. Furthermore, we establish the right-inverse property of the extension operator. We recall the notation $\theta = n-d$, \eqref{eq.besov_seminorm_on_thick_set}, and \eqref{eq.lizorkin_triebel_seminorm_on_thick_set}.
\begin{Lm}
    \label{Lm.Lebesgue_points_on_E}
    Let $\sigma\in (0, \infty)$ and let $\phi \in L_{\sigma}^{\loc}(\{\fm_k\})$. Assume that either $\sigma \le p$ and $\|\phi\|_{\mathfrak{b}^{s-\theta/p}_{p, q, \sigma}(E)}<\infty$, or $p<\infty$, $\sigma\le \min\{p, q\}$, and $\|\phi\|_{\mathfrak{f}^{s-\theta/p}_{p, q, \sigma}(E)}<\infty$. Then there exist a Borel function $\bar{\phi}:E\to\mathbb{R}$ and a Borel set $E_{\phi} \subset E$ with $\fm_0(E\setminus E_{\phi}) =0$ such that
    \begin{equation}
        \label{eq.Lebesgue_points_on_E}
        \lim_{k\to\infty}\fint\limits_{Q_k(x)}|{\phi}(y)-\bar{\phi}(x)|^{\sigma}d\fm_k(y) = 0 \qquad \text{for all } x\in E_{\phi}.
    \end{equation}
\end{Lm}
\begin{proof}
    For brevity we set
    \begin{equation}
        E_k(x):=\mathcal{E}_{\fm_k, \sigma}(\phi, Q_k(x)), \qquad \Phi_k(x) = \widetilde{\mathcal{E}}_{\fm_k, \sigma}(\phi, Q_k(x))
    \end{equation}
    for every $k\in \mathbb{N}_0$ and all $x\in\mathbb{R}^n$.
    \par
    \emph{Step 1.} Consider the function $R(x):= \sum_{k=0}^{\infty}E_k(x)$. We claim that $R(x)<\infty$ for $\fm_0$-a.e. $x\in E$. Fix $k\in \mathbb{N}_0$ and $x\in E$. For each $y\in Q_k(x)$, by \eqref{eq.local_approximation_outside_E_ordinary} in Lemma~\ref{Lm.local_approximation_outside_E}, applied with $c=2$ and $\underline{k}=0$,
    \begin{equation}
    \label{eq.Lebesgue_points_on_E1}
        \mathcal{E}_{\fm_k, \sigma}(\phi, Q_k(x)) \lesssim {\mathcal{E}}_{\fm_k, \sigma}(\phi, 2Q_k(y)) = \Phi_k(y).
    \end{equation}
    The latter equality follows from $x\in Q_k(y)\cap E$, and hence $Q_k(y)\cap E \neq \emptyset$. Therefore, for $p<\infty$, Jensen's inequality and property~\ref{cond:M1} of $\{\fm_k\}$ give
    \begin{equation}
        \label{eq.Lebesgue_points_on_E2}
        \begin{split}
        \|E_k\|_{L_p(\fm_0)}^p&\lesssim \int\limits_{\mathbb{R}^n}\Phi_k(y)^p\int\limits_{Q_k(y)}\frac{d\fm_0(x)}{\mathcal{L}^n(Q_k(x))}dy \\ &\lesssim 2^{k(n-d)}\|\Phi_k\|_{L_p(\mathbb{R}^n)}^p.
        \end{split}
    \end{equation}
    If $p=\infty$ the latter estimate is immediate. Then, by Minkowski's inequality and Lemma~\ref{Lm.besov_and_lizorkin_triebel_rapid_convergence}, applied with $\alpha = \frac{\theta}{p}$,
    \begin{equation}
        \label{eq.Lebesgue_points_on_E3}
        \|R\|_{L_p(\fm_0)} \lesssim \|\{2^{k\frac{\theta}{p}}\|\Phi_k\|_{L_p(\mathbb{R}^n)}\}_{k=0}^{\infty}\|_{\ell_1} \lesssim \|\phi\|_{\mathfrak{b}^{s-\theta/p}_{p, q, \sigma}(E)}.
    \end{equation}
    For the Lizorkin--Triebel case, we apply \eqref{eq.lizorkin_triebel_rapid_convergence} from Lemma~\ref{Lm.besov_and_lizorkin_triebel_rapid_convergence}. Thus,
    \begin{equation}
    \label{eq.Lebesgue_points_on_E5}
       \|R\|_{L_p(\fm_0)} \lesssim  \|\{2^{k\frac{\theta}{p}}\|\Phi_k\|_{L_p(\mathbb{R}^n)}\}_{k=0}^{\infty}\|_{\ell_1}\lesssim \|\phi\|_{\mathfrak{f}^{s-\theta/p}_{p, q, \sigma}(E)}.
    \end{equation}
    As a result, $R(x)<\infty$ for $\fm_0$-a.e. $x\in E$.
    \par
    \emph{Step 2.} Given $k\in\mathbb{N}_0$, set $M_k(x):=\operatorname{med}_{\fm_k}(\phi, Q_k(x))$.
    By Lemma~\ref{Lm.average_median_approximation_property}, there is $\lambda=\lambda(\sigma)$ such that $M_k(x) \in \mathfrak{C}^{\lambda}_{\fm_k, \sigma}(\phi, Q_k(x))$.
    \par
    Let $E_{\phi}:=\{x\in E: R(x)<\infty\}$. We claim that, for each $x\in E_{\phi}$, the sequence $\{M_k(x)\}_{k=0}^{\infty}$ is Cauchy. Indeed, for every $l>k$, by Lemma~\ref{Lm.local_approximation_sequence_of_measures}, we have
    \begin{equation}
        \label{eq.Lebesgue_points_on_E8}
        |M_l(x)-M_k(x)|\le \sum_{j=k}^{l-1}|M_{j+1}(x)-M_j(x)| \lesssim\sum_{j=k}^{\infty}E_j(x).
    \end{equation}
    The last quantity tends to $0$ as $k\to \infty$ at each $x\in E$ provided $R(x)<\infty$. Let $\bar{\phi}(x):=\lim_{k\to\infty}M_k(x)$ for all $x\in E_{\phi}$, and $\bar{\phi}(x) = 0$ for all $x\in E\setminus E_{\phi}$. Since $x\mapsto M_k(x)$ is Borel, it follows that $\bar{\phi}$ is Borel. Furthermore, for each $x\in E_{\phi}$,
    \begin{equation}
        \label{eq.Lebesgue_points_on_E9}
        \begin{split}
        \Bigl(\fint\limits_{Q_k(x)}|\phi(y)-\bar{\phi}(x)|^{\sigma}d\fm_k(y)\Bigr)^{1/\sigma}&\lesssim \Bigl(\fint\limits_{Q_k(x)}|\phi(y)-M_k(x)|^{\sigma}d\fm_k(y)\Bigr)^{1/\sigma}\\ + |M_k(x)-\bar{\phi}(x)| &\lesssim \sum_{j=k}^{\infty}E_j(x)\to 0, \quad \text{as } k\to \infty.
        \end{split}
    \end{equation}
    The proof is complete.
\end{proof}
Next, we prove that $\bar{\phi}$ constructed above coincides $\fm_0$-a.e. with $\phi$, and therefore $\fm_0(E\setminus \mathfrak{L}_{\{\fm_k, \sigma\}}(\phi))=0$. We follow the argument used in the proof of \cite[Theorem~6.10]{tyul}. We recall an additional property of a $d$-regular sequence of measures.
\begin{Def}
    \label{Def.strongly_regular_sequence_of_measure}
    Given $d\in [0, n]$, a $d$-thick closed set $E\subset \mathbb{R}^n$, and a $d$-regular sequence of measures $\{\fm_k\}_{k\in\mathbb{N}_0}$ on $E$, we say that $\{\fm_k\}_{k\in\mathbb{N}_0}$ is strongly $d$-regular if, for each Borel set $G\subset E$,
    \begin{equation}
        \label{eq.strongly_regular_sequence_of_measures}
        \limsup_{k\to\infty}\frac{\fm_k(Q_k(x)\cap G)}{\fm_k(Q_k(x))}>0 \qquad \text{for $\fm_0$-a.e. } x\in G.
    \end{equation}
\end{Def}
By \cite[Theorem~1.3]{tyul} every closed $d$-thick set admits a strongly $d$-regular sequence of measures. Throughout the rest of this section, we assume that $\{\fm_k\}$ fixed earlier is strongly $d$-regular. We are ready to prove the \emph{first main result} of this section.
\begin{Th}
    \label{Th.good_representative_on_thick_set}
     Let $\sigma\in (0, \infty)$ and let $\phi \in L_{\sigma}^{\loc}(\{\fm_k\})$. Assume that either $\sigma \le p$ and $\|\phi\|_{\mathfrak{b}^{s-\theta/p}_{p, q, \sigma}(E)}<\infty$, or $p<\infty$, $\sigma\le \min\{p, q\}$, and $\|\phi\|_{\mathfrak{f}^{s-\theta/p}_{p, q, \sigma}(E)}<\infty$. Then 
     \begin{equation}
         \label{eq.good_representative_on_thick_set}
         \fm_0(E\setminus \mathfrak{L}_{\{\fm_k\}, \sigma}(\phi)) = 0.
     \end{equation}
\end{Th}
\begin{proof}
    Let $\bar{\phi}$ and $E_{\phi}$ be given by Lemma~\ref{Lm.Lebesgue_points_on_E}. By \eqref{eq.Lebesgue_points_on_E} in Lemma~\ref{Lm.Lebesgue_points_on_E} and Remark~\ref{Rm.lebesgue_points_sequence_of_measures}, it suffices to prove that $\phi(x) = \bar{\phi}(x)$ for $\fm_0$-a.e. $x\in E_{\phi}$.
    \par
    \emph{Step 1.}
    Fix $\delta>0$ and a Borel set $F\subset E_{\phi}$ with $\fm_0(F)<\infty$. By Lusin's theorem, there exists a closed set $G\subset F$ such that $\fm_0(F\setminus G)<\delta$ and $\phi\big|_G$ is continuous.
    Fix $x\in G$ such that \eqref{eq.strongly_regular_sequence_of_measures} holds at $x$. Choose $c_x$ such that
    \begin{equation}
        \label{eq.good_representative_on_thick_set1}
        0<c_x<\limsup_{k\to\infty}\frac{\fm_k(Q_k(x)\cap G)}{\fm_k(Q_k(x))}.
    \end{equation}
    Then there exists an increasing sequence $\{k_j\}_{j=1}^{\infty}\subset\mathbb{N}$ such that
    \begin{equation}
    \label{eq.good_representative_on_thick_set2}
        \frac{\fm_{k_j}(Q_{k_j}(x)\cap G)}{\fm_{k_j}(Q_{k_j}(x))} >c_x
    \end{equation}
    for all $j\in\mathbb{N}$. Given $k\in \mathbb{N}_0$ and $\eta>0$, set
    \begin{equation}
    \label{eq.good_representative_on_thick_set3}
        G^{\eta}_k(x):=\{y\in Q_k(x):|\phi(y)-\bar{\phi}(x)|<\eta\}.
    \end{equation}
    Then Chebyshev's inequality gives
    \begin{equation}
    \label{eq.good_representative_on_thick_set4}
        \frac{\fm_k(Q_k(x)\setminus G_k^{\eta}(x))}{\fm_k(Q_k(x))} \le \frac{1}{\eta^{\sigma}}\fint\limits_{Q_k(x)}|\phi(y)-\bar{\phi}(x)|^{\sigma}d\fm_k(y).
    \end{equation}
    The last quantity tends to $0$ as $k\to\infty$ whenever $x\in E_{\phi}$. In particular, for all sufficiently large $j\in \mathbb{N}$,
    \begin{equation}
        \label{eq.good_representative_on_thick_set5}
        \frac{\fm_{k_j}(Q_{k_j}(x)\setminus G_{k_j}^{\eta}(x))}{\fm_{k_j}(Q_{k_j}(x))}<c_x.
    \end{equation}
    Consequently, using \eqref{eq.good_representative_on_thick_set2}, we obtain $Q_{k_j}(x)\cap G\cap G^{\eta}_{k_j}(x)\neq\emptyset$. Hence, we can choose $y_j \in Q_{k_j}(x)\cap G$ such that $|\phi(y_j)-\bar{\phi}(x)|<\eta$. Since $k_j\to\infty$, it follows that $y_j\to x$. By the continuity of $\phi\big|_G$, we obtain $\phi(y_j)\to \phi(x)$, and therefore
    \begin{equation}
    \label{eq.good_representative_on_thick_set6}
        |\phi(x)-\bar{\phi}(x)|\le \eta.
    \end{equation}
    Since $\eta>0$ is arbitrary, we get $\phi(x) = \bar{\phi}(x)$ for $\fm_0$-a.e. $x\in G$. Consequently,
    \begin{equation}
    \label{eq.good_representative_on_thick_set7}
        \fm_0(F \setminus \{x\in F: \phi(x) = \bar{\phi}(x)\}) <\delta.
    \end{equation}
    Letting $\delta \to 0$ gives $\phi(x) = \bar{\phi}(x)$ for $\fm_0$-a.e. $x\in F$.
    \par
    \emph{Step 2.} Given $m\in \mathbb{N}$, we put $F_m = E_{\phi}\cap Q_{-m}(0)$. Applying Step~1 to each $F_m$, we obtain $\phi(x) = \bar{\phi}(x)$ for $\fm_0$-a.e. $x\in F_m$. Since $E_{\phi}=\bigcup_{m=1}^{\infty}F_m$, it follows that $\phi(x)=\bar{\phi}(x)$ for $\fm_0$-a.e. $x\in E_{\phi}$. The proof is complete.
\end{proof}
The \emph{second main result} of this section reads as follows.
\begin{Th}
    \label{Th.extension_property}
    Let $\sigma\in (0, \infty)$ and let $\phi\in L_{\sigma}^{\loc}(\{\fm_k\})$. Assume that either $\sigma \le p$ and $\|\phi\|_{\mathfrak{b}^{s-\theta/p}_{p, q, \sigma}(E)}<\infty$, or $p<\infty$, $\sigma\le \min\{p, q\}$, and $\|\phi\|_{\mathfrak{f}^{s-\theta/p}_{p, q, \sigma}(E)}<\infty$. Let $\mathcal{T}=\{\mathcal{T}_k^x\}_{k\in\mathbb{N}_0, x\in E}$ be a family of $\lambda$-almost best approximating operators on $L_{\sigma}(Q_{k}(x), \fm_k)$, $\lambda>1$. Then
    \begin{equation}
        \label{eq.extension_property}
        \lim_{r\to0}\fint\limits_{Q_r(x)}|\Ext_{\mathcal{T}}\phi(y)-\phi(x)|dy = 0 \qquad \text{for $\fm_0$-a.e. } x\in E.
    \end{equation}
     In particular, $\Ext_{\mathcal{T}}\phi\big|_E^{\fm_0}=\phi$.
\end{Th}

\begin{proof}
    By Theorem~\ref{Th.good_representative_on_thick_set}, $ \fm_0(E\setminus\mathfrak{L}_{\{\fm_k\}, \sigma}(\phi))=0$.
    Moreover, in the proof of Theorem~\ref{Th.extension_L_p_estimate} we established that $\mathcal{L}^n\lfloor_E\lesssim\fm_0$. Hence, $\mathcal{L}^n(E\setminus\mathfrak{L}_{\{\fm_k\}, \sigma}(\phi))=0$.
    For brevity, we put
    \begin{equation}
        \label{eq.extension_property2}
        E_j(x):=\mathcal{E}_{\fm_j,\sigma}(\phi,Q_j(x)),
        \qquad
        \Phi_j(x):=\widetilde{\mathcal{E}}_{\fm_j,\sigma}(\phi,Q_j(x))
    \end{equation}
    for every $j\in \mathbb{N}_0$ and $x\in\mathbb{R}^n$.
    By the proof of Lemma~\ref{Lm.Lebesgue_points_on_E}, $\sum_{j=0}^{\infty}E_j(x)<\infty$ for $\fm_0$-a.e. $x\in E$.
     Let $k_0:=10$ as in Section~\ref{section.extension}, and, for each $k>k_0$, set
    \begin{equation}
        \label{eq.extension_property4}
        T_k(y):=\sum_{j=k-10}^{\infty}\Phi_j(y),
        \qquad
        A_k(x):=\fint\limits_{Q_k(x)}T_k(y)dy.
    \end{equation}
    \par
    \emph{Step 1.}
    We claim that $A_k(x)\to0$ as $k\to \infty$ for $\fm_0$-a.e. $x\in E$.
    Assume first that $p<\infty$. By Jensen's inequality, Fubini's theorem, and property~\ref{cond:M1}, we obtain
    \begin{equation}
        \label{eq.extension_property6}
        \begin{split}
        \|A_k\|_{L_p(\fm_0)}^p
        &\le
        \int\limits_E
        \fint\limits_{Q_k(x)}T_k(y)^pdy\,d\fm_0(x)
        \\
        &=
        \frac{1}{\mathcal{L}^n(Q_k(0))}
        \int\limits_{\mathbb{R}^n}
        T_k(y)^p\fm_0(Q_k(y))dy
        \lesssim
        2^{k\theta}
        \|T_k\|_{L_p(\mathbb{R}^n)}^p.
        \end{split}
    \end{equation}
    Consequently,
    \begin{equation}
        \label{eq.extension_property7}
        \|A_k\|_{L_p(\fm_0)}
        \lesssim
        2^{k\frac{\theta}{p}}\|T_k\|_{L_p(\mathbb{R}^n)}.
    \end{equation}
    Set $\alpha:=s-\frac{\theta}{p}$.
    Lemma~\ref{Lm.besov_and_lizorkin_triebel_rapid_convergence} then gives
    \begin{equation}
        \label{eq.extension_property9}
        \begin{split}
        \|A_k\|_{L_p(\fm_0)}
        \lesssim
        2^{k\theta/p}
        \sum_{j=k-10}^{\infty}
        \|\Phi_j\|_{L_p(\mathbb{R}^n)}\lesssim
        2^{-k\alpha}
        \|\phi\|_{\mathfrak{b}^{s-\theta/p}_{p,q,\sigma}(E)}.
        \end{split}
    \end{equation}
    For the Lizorkin--Triebel case, the same lemma yields
    \begin{equation}
        \label{eq.extension_property12}
        \begin{split}
        \|A_k\|_{L_p(\fm_0)}
        \lesssim
        2^{k\theta/p}
        \sum_{j=k-10}^{\infty}
        \|\Phi_j\|_{L_p(\mathbb{R}^n)}\lesssim
        2^{-k\alpha}
        \|\phi\|_{\mathfrak{f}^{s-\theta/p}_{p,q,\sigma}(E)}.
        \end{split}
    \end{equation}
    In either case, it follows from \eqref{eq.extension_property9} and \eqref{eq.extension_property12} that
    \begin{equation}
        \label{eq.extension_property13}
        \sum_{k=k_0+1}^{\infty}
        \|A_k\|_{L_p(\fm_0)}^p<\infty.
    \end{equation}
    Hence, by Chebyshev's inequality and the Borel--Cantelli lemma, $A_k(x)\to 0$ as $k\to \infty$ for $\fm_0$-a.e. $x\in E$.
    \par
    If $p=\infty$, only the Besov case is considered. In this case,
    \begin{equation}
        \label{eq.extension_property14}
        \begin{split}
        \|A_k\|_{L_{\infty}(\fm_0)}\le
        \|T_k\|_{L_{\infty}(\mathbb{R}^n)}
        \le
        \sum_{j=k-10}^{\infty}
        \|\Phi_j\|_{L_{\infty}(\mathbb{R}^n)}\lesssim
        2^{-ks}
        \|\phi\|_{\mathfrak{b}^{s}_{\infty,q,\sigma}(E)},
        \end{split}
    \end{equation}
    and hence $A_k(x)\to 0$ as $k\to \infty$ for every $x\in E$.
    \par
    \emph{Step 2.}
    Fix $x\in \mathfrak{L}_{\{\fm_k\}, \sigma}(\phi)$ such that $\sum_{j=0}^{\infty}E_j(x)<\infty$ and $A_j(x)\to 0$ as $j\to \infty$. Let $r>0$ be sufficiently small. Put $k:=k(r)$, so that $2^{-k-1}<r\le 2^{-k(r)}$. We may assume that $k>k_0+10$.
    Fix $y\in Q_r(x)$, $y\neq x$, and let $l=l(x,y)$ be the integer such that $2^{-l-1}\le |x-y|<2^{-l}$. Since $|x-y|<2^{-k}$, we have $l\ge k$.
    \par
    If $y\in E$, then the argument in the proof of Lemma~\ref{Lm.extension_fractional_gradient_points_in_E}, up to the application of Remark~\ref{Rm.two_local_approximations_relation}, gives
    \begin{equation}
        \label{eq.extension_property18}
        |\Ext_{\mathcal{T}}\phi(y)-\phi(x)|
        \lesssim
        \sum_{j=l-10}^{\infty}E_j(x)
        +
        \sum_{j=l-10}^{\infty}E_j(y).
    \end{equation}
    Since $E_j(y)\lesssim\Phi_j(y)$ by Remark~\ref{Rm.two_local_approximations_relation}, it follows that
    \begin{equation}
        \label{eq.extension_property19}
        |\Ext_{\mathcal{T}}\phi(y)-\phi(x)|
        \lesssim
        \sum_{j=k-10}^{\infty}E_j(x)
        +
        T_k(y).
    \end{equation}
    If $y\in\mathbb{R}^n\setminus E$, then the argument in the proof of Lemma~\ref{Lm.extension_fractional_gradient_points_mixed} gives the same estimate. In view of $\mathcal{L}^n(E\setminus \mathfrak{L}_{\{\fm_k\}, \sigma}(\phi))=0$, estimate \eqref{eq.extension_property19} therefore holds for $\mathcal{L}^n$-a.e. $y\in Q_r(x)$.
    \par
    Averaging \eqref{eq.extension_property19} over $Q_r(x)$, we obtain
    \begin{equation}
        \label{eq.extension_property20}
        \fint\limits_{Q_r(x)}
        |\Ext_{\mathcal{T}}\phi(y)-\phi(x)|dy
        \lesssim
        \sum_{j=k-10}^{\infty}E_j(x)
        +
        \fint\limits_{Q_r(x)}T_k(y)dy.
    \end{equation}
    The first term tends to $0$ as $r\to 0$, since $\sum_{j=0}^{\infty}E_j(x)<\infty$. Furthermore,
    \begin{equation}
        \label{eq.extension_property23}
        \fint\limits_{Q_r(x)}T_k(y)dy
        \lesssim\fint\limits_{Q_k(x)}T_k(y)dy
        = A_k(x)
        \to0 \qquad \text{as } r\to 0.
    \end{equation}
    Combining the preceding estimates, we obtain \eqref{eq.extension_property}. The proof is complete.
\end{proof}
\section{The direct trace theorem}
\label{Section.direct_trace_theorem}
Throughout this section we fix:
\begin{conditions}{\textbf{D}.5.}
    \item\label{cond:D51} a parameter $d\in [0,n]$ and a closed $d$-thick set
    $E\subset\mathbb{R}^n$;
    \item\label{cond:D52} a $d$-regular sequence of measures
    $\{\fm_k\}_{k=0}^{\infty}$ on $E$;
    \item\label{cond:D53} parameters $p\in [1,\infty]$, $q\in (0,\infty]$, and
    $s\in \left(\frac{n-d}{p},1\right)$.
\end{conditions}
Since $E$ and $\{\fm_k\}$ are fixed, we will write $\Tr := \Tr \big|_E^{\fm_0}$ throughout this section.
\par
The goal of this section is to establish the boundedness of Besov and Lizorkin--Triebel type functionals on the corresponding trace spaces. Our arguments rely on two complementary types of estimates: local pointwise estimates and a duality argument. The former covers the case of Besov spaces for the whole range of parameters and the case of Lizorkin--Triebel spaces for $q\ge p$. The case of Lizorkin--Triebel spaces for $q<p$ requires the duality argument.
\subsection{Local pointwise estimates}
\begin{Lm}
\label{Lm.pointw_est}
    Given $c\ge 1$, $\varepsilon\in(0,s)$, and $t\in (0, \infty)$, set $\underline{k}:=[\log_2{c}]+2$. Then there exists $C>0$ such that,
    for every Borel function $f:\mathbb{R}^n\to\mathbb{R}$, every
    $\varepsilon$-nonincreasing fractional gradient
    $\vec{g}\in\mathbb{D}^s(f)$, $k\in\mathbb{N}_0$, $x\in \mathbb{R}^n$, and every Lebesgue point $y\in cQ_k(x)$ of $f$,
    \begin{equation}
    \label{eq.pointw_est}
        |f(y)-\operatorname{med}(f, Q_k(x))| \le C \sum_{j=k-\underline{k}}^{\infty} 2^{-js} \Bigl(A_{j}g_j^t (y) \Bigr)^{1/t}.
    \end{equation}
\end{Lm}
\begin{proof}
For each $z\in \mathbb{R}^n$, define
\begin{equation}
\label{eq.pointw_est1}
    G_j(z) = \|\{2^{-is}g_i(z)\}_{i\ge j}\|_{\ell_1}.
\end{equation}
Since $\vec{g}$ is $\varepsilon$-nonincreasing, it follows that
\begin{equation}
\label{eq.pointw_est2}
    G_j(z) \le 2^{-j\varepsilon}\|\{2^{-i(s-\varepsilon)}g_j(z)\}_{i\ge j}\|_{\ell_1} \lesssim 2^{-js}g_j(z)
\end{equation}
for all $z\in\mathbb{R}^n$.
The definition of fractional gradient then gives, for $\mathcal{L}^{2n}$-a.e. $(z_1, z_2)\in \mathbb{R}^{2n}$ such that $|z_1-z_2|\le 2^{-j}$,
\begin{equation}
\label{eq.pointw_est3}
    |f(z_1)-f(z_2)| \le \bigl(G_j(z_1)+G_j(z_2) \bigr) \lesssim 2^{-js}(g_j(z_1)+g_j(z_2)).
\end{equation}
\par
We now assume that $y\in \mathbb{R}^n$ is a Lebesgue point of $f$ and let $m_j:= \operatorname{med}(f, Q_j(y))$ for brevity. Then, by \eqref{eq.median_at_lebesgue_points} in Lemma~\ref{Lm.meadian_average_relation}, applied with $\fm=\mathcal{L}^n$, we obtain
    \begin{equation}
    \label{eq.pointw_est4}
        f(y) = m_k + \sum_{j=k}^{\infty}\bigl(m_{j+1} - m_j \bigr).
    \end{equation}
    For each $j\ge k$, Lemma~\ref{Lm.meadian_average_relation} and estimate \eqref{eq.pointw_est3} give
    \begin{equation}
        \label{eq.pointw_est5}
        |m_{j+1}-m_j|^t \lesssim \fint\limits_{Q_j(y)}\fint\limits_{Q_{j+1}(y)}|f(z_1)-f(z_2)|^tdz_1dz_2 \lesssim 2^{-(j-1)s t}\fint\limits_{Q_{j-1}(y)}g_{j-1}(z)^tdz.
    \end{equation}
    Furthermore, if $y\in cQ_k(x)$, then $|z_1-z_2|\le (2+c)2^{-k} \le 2^{-k+\underline{k}}$ for all $z_1\in Q_k(x)$ and $z_2\in Q_k(y)$. Therefore, it follows from Lemma~\ref{Lm.meadian_average_relation}, estimate \eqref{eq.pointw_est3}, and the inclusions $Q_{k}(x) \subset Q_{k-\underline{k}}(y)$ and $Q_k(y) \subset Q_{k-\underline{k}}(y)$ that
    \begin{equation}
    \label{eq.pointw_est6}
    \begin{split}
        |m_k - \operatorname{med}(f, Q_k(x))|^t &\lesssim \fint\limits_{Q_k(y)}\fint\limits_{Q_{k}(x)}|f(z_1)-f(z_2)|^tdz_1dz_2 \\& \lesssim 2^{-(k-\underline{k})st} \fint\limits_{Q_{k-\underline{k}}(y)}g_{k-\underline{k}}(z)^tdz.
    \end{split}
    \end{equation}
    Combining \eqref{eq.pointw_est4}, \eqref{eq.pointw_est5}, and \eqref{eq.pointw_est6} gives \eqref{eq.pointw_est}, completing the proof.
\end{proof}

We are now ready to prove the \emph{key pointwise estimate} for
$\widetilde{\mathcal{E}}_{\fm_k, \sigma}(\Tr f,Q_k(x))$.

\begin{Prop}
\label{Prop.pointwise_estimation_direct}
    Let $\sigma \in (0, \infty)$ and let $\varepsilon\in(0,s)$. Assume that $r\in (0, \infty)$ such that $\sigma\le r$ and $s-\varepsilon>\frac{n-d}{r}$. Then there exists $C>0$ such that, for each Borel function $f$ that has a trace to $E$ in the sense of \eqref{eq.trace_definition}, each $\varepsilon$-nonincreasing $\vec{g} \in \mathbb{D}^s(f)$, $k\in \mathbb{N}_0$, and $x\in\mathbb{R}^n$,
    \begin{equation}
    \label{eq.pointwise_estimation_direct}
        \widetilde{\mathcal{E}}_{\fm_k, \sigma}(\Tr f,Q_k(x))
        \le
        C2^{-ks}
        \Bigl(
            \fint\limits_{10Q_k(x)}
            \bigl(g_{k-3}(y)\bigr)^r\,dy
        \Bigr)^{1/r}.
    \end{equation}
\end{Prop}
\begin{proof}
    Fix $k\in\mathbb{N}_0$ and $x\in\mathbb{R}^n$, and put
    $Q:=Q_k(x)$. If $Q\cap\operatorname{supp}\fm_k=\emptyset$, then
    \eqref{eq.pointwise_estimation_direct} follows directly from the
    definition of the modified local approximation. We may therefore
    assume that $Q\cap\operatorname{supp}\fm_k\neq\emptyset$. By the definition of the trace \eqref{eq.trace_definition}, for $\fm_0$-a.e. $y\in E$
\begin{equation}
    \label{eq.pointwise_estimation_direct0}
    \lim_{r\to 0}\fint\limits_{Q_r(y)}|f(z)-\Tr f(y)|dz = 0.
\end{equation}
Redefining $f$ on $E$, if necessary, by $f:=\Tr f$, we may therefore regard such points as Lebesgue points of $f$. Moreover, since $\vec{g}$ is $\varepsilon$-nonincreasing,
    $g_j\le 2^{(j-k+3)\varepsilon}g_{k-3}$ for every $j\ge k-3$. Consequently, Lemma~\ref{Lm.pointw_est}, applied with $c=2$ and $t=r$, gives
    \begin{equation}
    \label{eq.pointwise_estimation_direct1}
    \begin{split}
        \widetilde{\mathcal{E}}_{\fm_k, \sigma}(\Tr f, Q)^{\sigma}& \le \fint\limits_{2Q}|\Tr f(y)-\operatorname{med}(f, Q_k(x))|^{\sigma}d\fm_k(y)\\& \lesssim 2^{-k\varepsilon}\fint\limits_{2Q}\Bigl(\sum_{j=k-3}^{\infty}2^{-j(s-\varepsilon)}\bigl(A_{j}g_{k-3}^r(y) \bigr)^{1/r} \Bigr)^{\sigma}d\fm_k(y).
    \end{split}
    \end{equation}
    Since $r\ge \sigma$, it follows by Jensen's inequality and the preceding estimate that
    \begin{equation}
        \label{eq.pointwise_estimation_direct2}
        \widetilde{\mathcal{E}}_{\fm_k, \sigma}(\Tr f, Q) \lesssim 2^{-k\varepsilon}\Bigl(\fint\limits_{2Q}\Bigl(\sum_{j=k-3}^{\infty}2^{-j(s-\varepsilon)}\bigl(A_{j}g_{k-3}^r(y) \bigr)^{1/r} \Bigr)^{r}d\fm_k(y)\Bigr)^{1/r}.
    \end{equation}
    \par
    First, assume that $r\ge 1$. Then Minkowski's inequality gives
    \begin{equation}
        \label{eq.pointwise_estimation_direct3}
        \widetilde{\mathcal{E}}_{\fm_k, \sigma}(\Tr f, Q) \lesssim 2^{-k\varepsilon}\sum_{j=k-3}^{\infty}2^{-j(s-\varepsilon)}\Bigl(\fint\limits_{2Q}A_{j}g_{k-3}^r(y) d\fm_k(y)\Bigr)^{1/r}.
    \end{equation}
    By Lemma~\ref{Lm.relaxed_doubling_property} and properties~\ref{cond:M1}--\ref{cond:M2} of $\{\fm_k\}$, we obtain $\fm_k(2Q)\approx 2^{-kd}$. If $y\in2Q$ and $j\ge k-3$, then $Q_j(y)\subset10Q$. Therefore, property~\ref{cond:M1}, Lemma~\ref{Lm.relaxed_doubling_property}, and  Fubini's theorem yield, for each $j\ge k-3$
    \begin{equation}
        \label{eq.pointwise_estimation_direct4}
        \int\limits_{2Q}A_{j}g_{k-3}^r(y) d\fm_k(y) \lesssim 2^{j(n-d)} \int\limits_{10Q}g_{k-3}^r(z)dz.
    \end{equation}
    Since $s-\varepsilon > \frac{n-d}{r}$, it follows by \eqref{eq.pointwise_estimation_direct3} and \eqref{eq.pointwise_estimation_direct4} that
    \begin{equation}
        \label{eq.pointwise_estimation_direct5}
        \widetilde{\mathcal{E}}_{\fm_k, \sigma}(\Tr f, Q) \lesssim 2^{-k\varepsilon+k\frac{d}{r}-k(s-\varepsilon-\frac{n-d}{r})}\Bigl( \int\limits_{10Q}g_{k-3}^r(z)dz\Bigr)^{1/r} \lesssim 2^{-ks}\Bigl( \fint\limits_{10Q}g_{k-3}^r(z)dz\Bigr)^{1/r}.
    \end{equation}
    \par
    If $r< 1$, then, by \eqref{eq.pointwise_estimation_direct2} and the subadditivity, we obtain
    \begin{equation}
        \label{eq.pointwise_estimation_direct6}
        \widetilde{\mathcal{E}}_{\fm_k, \sigma}(\Tr f, Q) \lesssim 2^{-k\varepsilon}\Bigl(\sum_{j=k-3}^{\infty}2^{-j(s-\varepsilon)r}\fint\limits_{2Q}A_{j}g_{k-3}^r(y) d\fm_k(y)\Bigr)^{1/r}.
    \end{equation}
    By Fubini's theorem and properties of $\{\fm_k\}$, we still have \eqref{eq.pointwise_estimation_direct4}. As a result, we obtain \eqref{eq.pointwise_estimation_direct}.
    The proof is complete.
\end{proof}
We are now ready to prove the \emph{first main result} of this section. We recall the notation $\theta = n-d$ and \eqref{eq.besov_norm_on_thick_set}.
\begin{Th}
\label{Th.main_direct_Besov}
Let $\sigma \in (0, \infty)$ be such that $\sigma\le p$.
    There exists a constant $C>0$ such that for every
    $f\in B^s_{p,q}(\mathbb{R}^n)$,
    \begin{equation}
    \label{eq.direct_estimate_main}
        \|\Tr f\|_{\mathfrak{B}^{s-\theta/p}_{p,q, \sigma}(E)}
        \le
        C\|f\|_{B^s_{p,q}(\mathbb{R}^n)}.
    \end{equation}
\end{Th}

\begin{proof}
    Fix $f\in B^s_{p,q}(\mathbb{R}^n)$ and choose
    $\vec{g}\in\mathbb{D}^s(f)$ such that
    \begin{equation}
    \label{eq.choice_of_fractional_gradient_direct}
        \|\vec{g}\|_{\ell_q(L_p(\mathbb{R}^n))}
        \le
        2\|f\|_{B^s_{p,q}(\mathbb{R}^n)}.
    \end{equation}
    Choose $\varepsilon \in (0, s-\frac{\theta}{p})$.
    By Lemma~\ref{Lm.monotonic_fractional_gradient}, there exists an
    $\varepsilon$-nonincreasing fractional gradient
    $\vec{h}\in\mathbb{D}^s(f)$ such that
    \begin{equation}
    \label{eq.monotonic_gradient_norm_direct}
        \|\vec{h}\|_{\ell_q(L_p(\mathbb{R}^n))}
        \lesssim
        \|\vec{g}\|_{\ell_q(L_p(\mathbb{R}^n))}
        \lesssim
        \|f\|_{B^s_{p,q}(\mathbb{R}^n)}.
    \end{equation}
    \par
    Since $s-\varepsilon >\frac{\theta}{p}$, there is $r\in (0, \infty)$ such that $\sigma \le r\le p$ and $s-\varepsilon > \frac{\theta}{r}$. Therefore,
    by Proposition~\ref{Prop.pointwise_estimation_direct}, for every $k\in\mathbb{N}_0$ and
    $x\in\mathbb{R}^n$, we obtain
    \begin{equation}
    \label{eq.direct_seminorm_pointwise_bound}
        2^{ks}
        \widetilde{\mathcal{E}}_{\fm_k, \sigma}
        \bigl(\Tr f,Q_k(x)\bigr)
        \lesssim
        \Bigl(
            \fint\limits_{10Q_k(x)}
            \bigl(h_{k-3}(y)\bigr)^r\,dy
        \Bigr)^{1/r}.
    \end{equation}
    Assume first that $p<\infty$. Since $r\le p$, Jensen's inequality and Fubini's theorem then imply
    \begin{equation}
    \label{eq.direct_seminorm_Lp_bound}
        \begin{aligned}
        \left\|
            2^{ks}
            \widetilde{\mathcal{E}}_{\fm_k, \sigma}
            \bigl(\Tr f,Q_k(\cdot)\bigr)
        \right\|_{L_p(\mathbb{R}^n)}^p\lesssim
        \int\limits_{\mathbb{R}^n}
        \fint\limits_{10Q_k(x)}
        \bigl(h_{k-3}(y)\bigr)^p\,dy\,dx
        \lesssim
        \|h_{k-3}\|_{L_p(\mathbb{R}^n)}^p.
        \end{aligned}
    \end{equation}
    For $p=\infty$, the latter estimate is immediate. Taking the $\ell_q$ quasi-norm with respect to $k\in\mathbb{N}_0$ and
    using the definition of the Besov functional~\eqref{eq.besov_seminorm_on_thick_set} on $E$, we obtain
    \begin{equation}
    \label{eq.direct_estimate1}
        \|\Tr f\|_{\mathfrak{b}^{s-\theta/p}_{p,q, \sigma}(E)}
        \lesssim
        \|\vec{h}\|_{\ell_q(L_p(\mathbb{R}^n))}
        \lesssim
        \|f\|_{B^s_{p,q}(\mathbb{R}^n)}.
    \end{equation}
    \par
    It remains to estimate the $L_p(\fm_0)$-norm of the trace. For
    $\fm_0$-almost every $x\in E$, we have
    \begin{equation}
    \label{eq.trace_telescoping_series_Lp}
        \Tr f(x)
        =
        \operatorname{med}(f, Q_0(x))
        +
        \sum_{k=0}^{\infty}
        \left(
            \operatorname{med}(f, Q_{k+1}(x)) - \operatorname{med}(f, Q_k(x))
        \right).
    \end{equation}
    Arguing as in the proof of
    Lemma~\ref{Lm.pointw_est} with $t=1$, for every $k\in\mathbb{N}_0$ we
    obtain,
    \begin{equation}
    \label{eq.consecutive_averages_trace_bound}
        \left|
            \operatorname{med}(f, Q_{k+1}(x)) - \operatorname{med}(f, Q_k(x))
        \right|
        \lesssim
        2^{-ks}
        \fint\limits_{Q_k(x)}
        h_{k-1}(y)dy.
    \end{equation}
    Furthermore, Lemma~\ref{Lm.meadian_average_relation}, applied with $c=0$, gives
    \begin{equation}
        \label{eq.low_frequency_term_direct}
        |\operatorname{med}(f, Q_0(x))| \lesssim \fint\limits_{Q_0(x)}|f(y)|dy
    \end{equation}
    Combining \eqref{eq.trace_telescoping_series_Lp},
    \eqref{eq.consecutive_averages_trace_bound}, and \eqref{eq.low_frequency_term_direct}, we obtain
    \begin{equation}
    \label{eq.trace_pointwise_Lp_reduction}
        |\Tr f(x)|
        \lesssim
        A_{0}|f|(x)
        +
        \sum_{k=-1}^{\infty}
        2^{-ks}
        \fint\limits_{Q_{k+1}(x)}
        h_k(y)\,dy
    \end{equation}
    for $\fm_0$-almost every $x\in E$.
    \par
    We first estimate the initial average. First assume that $p<\infty$. Then, by Jensen's inequality,
    Fubini's theorem, and property~\ref{cond:M1} of $\{\fm_k\}$,
    \begin{equation}
    \label{eq.initial_average_trace_bound}
        \begin{aligned}
        \|A_{0}|f|\|_{L_p(\fm_0)}^p
        &\le
        \int\limits_{\mathbb{R}^n}
        \fint\limits_{Q_0(x)}
        |f(y)|^p\,dy\,d\fm_0(x)
        \\
        &=
        \frac{1}{|Q_0|}
        \int\limits_{\mathbb{R}^n}
        |f(y)|^p\fm_0(Q_0(y))\,dy
        \lesssim
        \|f\|_{L_p(\mathbb{R}^n)}^p.
        \end{aligned}
    \end{equation}
    Next, for each $k\ge -1$, Jensen's inequality, Fubini's
    theorem, and property~\ref{cond:M1} of $\{\fm_k\}$, give
    \begin{equation}
    \label{eq.averaging_operator_measure_bound}
        \begin{aligned}
        &
        \int\limits_{\mathbb{R}^n}
        \Bigl(
            \fint\limits_{Q_{k+1}(x)}
            h_k(y)\,dy
        \Bigr)^p
        d\fm_0(x)
        \\
        &\qquad\le
        \frac{1}{|Q_{k+1}|}
        \int\limits_{\mathbb{R}^n}
        \bigl(h_k(y)\bigr)^p
        \fm_0(Q_{k+1}(y))\,dy
        \lesssim
        2^{k\theta}
        \|h_k\|_{L_p(\mathbb{R}^n)}^p.
        \end{aligned}
    \end{equation}
    For $p=\infty$, these estimates are immediate. Therefore, Minkowski's inequality,
    \eqref{eq.trace_pointwise_Lp_reduction},
    \eqref{eq.initial_average_trace_bound}, and
    \eqref{eq.averaging_operator_measure_bound} yield
    \begin{equation}
    \label{eq.trace_Lp_weighted_sequence_bound}
        \|\Tr f\|_{L_p(\fm_0)}
        \lesssim
        \|f\|_{L_p(\mathbb{R}^n)}
        +
        \sum_{j=-1}^{\infty}
        2^{-j(s-\frac{\theta}{p})}
        \|h_j\|_{L_p(\mathbb{R}^n)}.
    \end{equation}
    Arguing as in Lemma~\ref{Lm.besov_and_lizorkin_triebel_rapid_convergence}, we obtain
    \begin{equation}
    \label{eq.weighted_sequence_bound_q_greater_one}
        \sum_{j=-1}^{\infty}
        2^{-j(s-\frac{\theta}{p})}
        \|h_j\|_{L_p(\mathbb{R}^n)}
        \lesssim
        \|\vec{h}\|_{\ell_q(L_p(\mathbb{R}^n))}.
    \end{equation}
    As a result,
    \begin{equation}
    \label{eq.direct_estimate2}
        \|\Tr f\|_{L_p(\fm_0)}
        \lesssim
        \|f\|_{L_p(\mathbb{R}^n)}
        +
        \|\vec{h}\|_{\ell_q(L_p(\mathbb{R}^n))}
        \lesssim
        \|f\|_{B^s_{p,q}(\mathbb{R}^n)}.
    \end{equation}

    Combining \eqref{eq.direct_estimate1} and
    \eqref{eq.direct_estimate2} proves
    \eqref{eq.direct_estimate_main}.
\end{proof}
We recall the notation \eqref{eq.lizorkin_triebel_norm_on_thick_set}.
The \emph{second main result} of this section reads as follows.
\begin{Th}
\label{Th.main_direct_Lizorkin_Triebel} Let $p<\infty$, $q\ge p$, and let $\sigma \in (0, p)$. There exists a constant $C>0$ such that for every
    $f\in F^s_{p,q}(\mathbb{R}^n)$,
    \begin{equation}
    \label{eq.direct_estimate_main_Lizorkin_Triebel}
        \|\Tr f\|_{\mathfrak{F}^{s-\theta/p}_{p,q, \sigma}(E)}
        \le
        C\|f\|_{F^s_{p,q}(\mathbb{R}^n)}.
    \end{equation}
\end{Th}
\begin{proof}
    Fix $f\in F^s_{p,q}(\mathbb{R}^n)$ and choose
    $\vec{g}\in\mathbb{D}^s(f)$ such that
    \begin{equation}
    \label{eq.direct_estimate_main_Lizorkin_Triebel1}
        \|\vec{g}\|_{L_p(\mathbb{R}^n, \ell_q)}
        \le
        2\|f\|_{F^s_{p,q}(\mathbb{R}^n)}.
    \end{equation}
    Choose $\varepsilon \in (0, s-\frac{\theta}{p})$. By Lemma~\ref{Lm.monotonic_fractional_gradient}, there exists an
    $\varepsilon$-nonincreasing fractional gradient
    $\vec{h}\in\mathbb{D}^s(f)$ such that
    \begin{equation}
    \label{eq.direct_estimate_main_Lizorkin_Triebel2}
        \|\vec{h}\|_{L_p(\mathbb{R}^n, \ell_q)}
        \lesssim
        \|\vec{g}\|_{L_p(\mathbb{R}^n, \ell_q)}
        \lesssim
        \|f\|_{F^s_{p,q}(\mathbb{R}^n)}.
    \end{equation}
    \par
    Fix $r\in [\sigma, p)$ such that $s-\varepsilon>\frac{\theta}{r}$.
    Proposition~\ref{Prop.pointwise_estimation_direct} gives, for every $k\in\mathbb{N}_0$ and every
    $x\in\mathbb{R}^n$,
    \begin{equation}
    \label{eq.direct_estimate_main_Lizorkin_Triebel3}
        2^{ks}
        \widetilde{\mathcal{E}}_{\fm_k, \sigma}
        \bigl(\Tr f,Q_k(x)\bigr)
        \lesssim
        \Bigl(
            \fint\limits_{10Q_k(x)}
            \bigl(h_{k-3}(y)\bigr)^r\,dy
        \Bigr)^{1/r} \le (M[h_{k-3}^r](x))^{1/r}.
    \end{equation}
    Since $r<p\le q$, it follows by the Fefferman--Stein inequality (Theorem~\ref{Th.Fefferman_Stein_inequality}) that
    \begin{equation}
    \label{eq.direct_estimate_main_Lizorkin_Triebel4}
        \begin{split}
        \left\|
            \{2^{ks}
            \widetilde{\mathcal{E}}_{\fm_k, \sigma}
            \bigl(\Tr f,Q_k(\cdot)\bigr)\}_{k\ge 0}
        \right\|_{L_p(\mathbb{R}^n, \ell_q)}\lesssim \left\|\{M[h_{k}^r]\}_{k\ge -3}\right\|_{L_{p/r}(\mathbb{R}^n, l_{q/r})}^{1/r} \\\lesssim \|\{h_k^r\}_{k\ge -3}\|_{L_{p/r}(\mathbb{R}^n, l_{q/r})}^{1/r} \le \|\vec{h}\|_{L_p(\mathbb{R}^n, \ell_q)} \lesssim \|f\|_{F^s_{p, q}(\mathbb{R}^n)}.
        \end{split}
    \end{equation}
    \par
    It remains to estimate the $L_p(\fm_0)$-norm of the trace. Arguing as in Theorem~\ref{Th.main_direct_Besov}, we obtain
    \begin{equation}
    \label{eq.direct_estimate_main_Lizorkin_Triebel5}
        \|\Tr f\|_{L_p(\fm_0)}
        \lesssim
        \|f\|_{L_p(\mathbb{R}^n)}
        +
        \sum_{j=-1}^{\infty}
        2^{-j(s-\frac{\theta}{p})}
        \|h_j\|_{L_p(\mathbb{R}^n)}
    \end{equation}
    Since $s-\frac{\theta}{p}>0$, it follows that
    \begin{equation}
    \label{eq.direct_estimate_main_Lizorkin_Triebel6}
        \begin{aligned}
        \|\Tr f\|_{L_p(\fm_0)}
        \lesssim
        \|f\|_{L_p(\mathbb{R}^n)}
        +
        \|\vec{h}\|_{L_p(\mathbb{R}^n,\ell_q)}
        \lesssim
        \|f\|_{F^s_{p,q}(\mathbb{R}^n)}.
        \end{aligned}
    \end{equation}
    Combining \eqref{eq.direct_estimate_main_Lizorkin_Triebel4} and
    \eqref{eq.direct_estimate_main_Lizorkin_Triebel6} completes the proof.
\end{proof}
\subsection{A duality argument}
In this subsection we develop the machinery needed for the direct trace estimate for Lizorkin--Triebel spaces $F^s_{p, q}(\mathbb{R}^n)$ when $q<p$. The difficulty when \(q<p\) is that the vector-valued maximal-function argument used above no longer gives the required estimate in the appropriate direction. We therefore dualize the \(L^{p/q}\)-norm of the \(q\)-th power of the trace functional. First, we define auxiliary measures and operators. More precisely, for each $k\in \mathbb{N}_0$, we let $\nu_k := 2^{-k\theta}\fm_k$. Given $k\ge 0$ and $m\in \mathbb{Z}$, we set
\begin{equation}
\label{eq.density_measure_definition}
    \mathcal{D}_{k, m}(x):=\frac{\nu_k(Q_{k+m}(x))}{\mathcal{L}^n(Q_{k+m}(x))}, \qquad x\in\mathbb{R}^n.
\end{equation}
Furthermore, we put
\begin{equation}
    \label{eq.modified_density_measure_definition}
    \mathcal{D}^*_{k, m}(x) := \sup_{l\ge k}\mathcal{D}_{l, m}(x), \qquad x\in\mathbb{R}^n.
\end{equation}
Given $m\in\mathbb{Z}$ and $\underline{k}\in \mathbb{N}_0$, we define, for each $H\in L_1^{\loc}(\mathbb{R}^n)$,
\begin{equation}
\label{eq.T_operator_definition}
    T_{m, \underline{k}}[H](x) := \sup_{k\ge 0}\mathcal{D}_{k, m}(x)A_{k-\underline{k}}|H|(x), \qquad x\in\mathbb{R}^n.
\end{equation}
Finally, for $m\in\mathbb{Z}$, $\underline{k}\in\mathbb{N}_0$, and $a\in (1, \infty)$, we set, for each $H\in L_1^{\loc}(\mathbb{R}^n)$,
\begin{equation}
    \label{eq.S_operator_definition}
    S^a_{m, \underline{k}}[H](x):=\sup_{k\ge 0}\mathcal{D}_{k, m}(x)\Bigl(A_{k-\underline{k}}|H|(x)\Bigr)^a, \qquad x\in\mathbb{R}^n.
\end{equation}
We prove that operators $S^a_{m, \underline{k}}: L_a(\mathbb{R}^n) \to L_1(\mathbb{R}^n)$ and $T_{m, \underline{k}}: L_a(\mathbb{R}^n)\to L_a(\mathbb{R}^n)$ are bounded. The argument is standard in spirit: it combines maximal-function estimates with Whitney decompositions of level sets of the Hardy--Littlewood maximal function. Since we have not found a convenient reference for the precise form needed here, we include the details.
We start with two simple properties of $\nu_k$ and $\mathcal{D}_{k, m}$.
\begin{Lm}
    \label{Lm.simple_properties}
    \begin{enumerate}
        \item There exists $C>0$ such that for each $k, j\in\mathbb{N}_0$, $\nu_{k+j}\le C\nu_k$.
        \item Given $m_0 \in\mathbb{Z}$, there exists $C>0$ such that for each $k\in \mathbb{N}_0$ and each $m\ge m_0$,
        \begin{equation}
            \label{eq.density_measure_global_estimate}
            \mathcal{D}_{k, m}(x)\le C 2^{m\theta}, \qquad \text{for all } x\in\mathbb{R}^n.
        \end{equation}
    \end{enumerate}
\end{Lm}
\begin{proof}
    The first assertion follows directly from property~\ref{cond:M3} of $\{\fm_k\}$. Indeed,
    \begin{equation}
        \nu_{k+j}=2^{-(k+j)\theta}\gamma_{k+j}\fm_0 \lesssim 2^{-k\theta}\gamma_k\fm_0 = \nu_k.
    \end{equation}
    To prove the second assertion, we use Corollary~\ref{Ca.M1_inflated}. Since $m\ge m_0$, it follows that
    \begin{equation}
        \mathcal{D}_{k, m}(x) \lesssim \frac{2^{-k\theta}2^{-(k+m)d}}{2^{-(k+m)n}} = 2^{m\theta}.
    \end{equation}
    for all $x\in\mathbb{R}^n$.
\end{proof}
The following property of $\mathcal{D}^*_{k, m}$ is a local analogue of the weak-$(1, 1)$ estimate for the Hardy--Littlewood maximal function.
\begin{Lm}
    \label{Lm.modified_density_estimation}
    Given $m_0\in \mathbb{Z}$, there exists $C>0$ such that, for each $k\in \mathbb{N}_0$, $m\ge m_0$, $\lambda>0$, and $x\in \mathbb{R}^n$,
    \begin{equation}
        \label{eq.modified_density_estimation}
        \mathcal{L}^n(\{y\in Q_k(x): \mathcal{D}^*_{k, m}(y)>\lambda\})\le C\frac{\mathcal{L}^n(Q_k(x))}{\lambda}.
    \end{equation}
\end{Lm}
\begin{proof}
    Assume that $\mathcal{D}^*_{k, m}(y)>\lambda$. Then there is $l(y)\ge k$ such that $\mathcal{D}_{l(y), m}(y)>\lambda$, that is,
    \begin{equation}
        \label{eq.modified_density_estimation1}
        \nu_{l(y)}(Q_{l(y)+m}(y))>\lambda \mathcal{L}^n(Q_{l(y)+m}(y)).
    \end{equation}
    Since the family $\{Q_{l(y)+m}(y)\}_{y\in Q_k(x)}$ forms a covering of $\{z\in Q_k(x): \mathcal{D}^*_{k, m}(z)>\lambda\}$, the \(5B\)-covering lemma yields a countable subfamily $\{Q_j:=Q_{l_j+m}(y_j)\}$ of pairwise disjoint cubes such that $\{z\in Q_k(x): \mathcal{D}^*_{k, m}(z)>\lambda\}\subset \bigcup_{j}5Q_j$. Furthermore, Lemma~\ref{Lm.simple_properties} guarantees that $\nu_l \lesssim \nu_k$ for all $l\ge k$. Since $m\ge m_0$, it follows that all the cubes $Q_j$ are contained in a fixed enlargement $cQ_k(x)$. Consequently, by Corollary~\ref{Ca.M1_inflated},
    \begin{equation}
        \lambda\sum_{j}\mathcal{L}^n(Q_j) < \sum_{j}\nu_{l_j}(Q_j)\lesssim \sum_{j}\nu_k(Q_j)\le \nu_k(cQ_k(x)) \lesssim2^{-k(n-d)}2^{-kd}=2^{-kn}.
    \end{equation}
    Thus, \eqref{eq.modified_density_estimation} follows immediately.
\end{proof}
\begin{Ca}
    \label{Ca.modified_density_estimation}
    Given $m_0\in \mathbb{Z}$, there exists $C>0$ such that, for each $k\in \mathbb{N}_0$, $m> m_0$, and $x\in \mathbb{R}^n$,
    \begin{equation}
        \label{eq.modefied_density_L_1_esstimation}
        \int\limits_{Q_k(x)}\mathcal{D}^*_{k, m}(y)dy \le C(m-m_0)\mathcal{L}^n(Q_k(x)).
    \end{equation}
\end{Ca}
\begin{proof}
    By the layer cake representation,
    \begin{equation}
        \label{eq.modefied_density_L_1_esstimation1}
        \int\limits_{Q_k(x)}\mathcal{D}^*_{k, m}(y)dy = \int\limits_0^{\infty}\mathcal{L}^n(\{y\in Q_k(x): \mathcal{D}^*_{k, m}(y)>\lambda\})d\lambda.
    \end{equation}
    For $\lambda \in (0, 1)$ we use the trivial estimate
    \begin{equation}
    \label{eq.modefied_density_L_1_esstimation2}
        \mathcal{L}^n(\{y\in Q_k(x): \mathcal{D}^*_{k, m}(y)>\lambda\}) \le \mathcal{L}^n(Q_k(x)).
    \end{equation}
    Moreover, by Lemma~\ref{Lm.simple_properties}, $\mathcal{D}_{k, m}^* \lesssim2^{m\theta}$ everywhere. Therefore, the upper limit of the integral in the right-hand side of \eqref{eq.modefied_density_L_1_esstimation1} can be restricted to $c2^{m\theta}$. Without loss of generality, we assume $c2^{m_0\theta}>1$. Consequently, by Lemma~\ref{Lm.modified_density_estimation},
    \begin{equation}
     \label{eq.modefied_density_L_1_esstimation4}
        \int\limits_1^{\infty}\mathcal{L}^n(\{y\in Q_k(x): \mathcal{D}^*_{k, m}(y)>\lambda\})d\lambda. \lesssim \int\limits_1^{c2^{m\theta}}\frac{\mathcal{L}^n(Q_k(x))}{\lambda}d\lambda \lesssim (m-m_0)\mathcal{L}^n(Q_k(x)).
    \end{equation}
    The proof is complete.
\end{proof}
\begin{Prop}
    \label{Prop.S_m_operator_bound} Given $m_0\in \mathbb{Z}$, $\underline{k}\in\mathbb{N}_0$, and $a\in (1, \infty)$, there exists $C>0$ such that for each $m> m_0$ and $H \in L_a(\mathbb{R}^n)$,
    \begin{equation}
        \label{eq.S_m_operator_bound}
        \|S^a_{m, \underline{k}}[H]\|_{L_1(\mathbb{R}^n)}\le C(m-m_0)\|H\|_{L_a(\mathbb{R}^n)}^a.
    \end{equation}
\end{Prop}
\begin{proof}
    We begin with a simple observation. For each $u\ge 0$,
    \begin{equation}
        \label{eq.S_m_operator_bound1}
        u^a \lesssim \sum_{j\in\mathbb{Z}}2^{ja}\chi_{\{u>2^{j}\}}.
    \end{equation}
    Therefore,
    \begin{equation}
        \label{eq.S_m_operator_bound2}
        S_{m, \underline{k}}^a[H](x)\lesssim \sum_{j\in\mathbb{Z}}2^{ja}\sup_{\substack{k\in\mathbb{N}_0:\\ A_{k-\underline{k}}|H|(x)> 2^{j}}}\mathcal{D}_{k, m}(x).
    \end{equation}
    For brevity, we set
    \begin{equation}
    \label{eq.S_m_operator_bound3}
        G_{j, m}(x):=\sup_{\substack{k\in\mathbb{N}_0:\\ A_{k-\underline{k}}|H|(x)> 2^{j}}}\mathcal{D}_{k, m}(x), \qquad (x, j)\in \mathbb{R}^n\times \mathbb{Z}.
    \end{equation}
    \par
    Fix $c_0<1$ sufficiently small. For each $j\in \mathbb{Z}$, set
    \begin{equation}
    \label{eq.S_m_operator_bound4}
        \Omega_j:=\{x\in \mathbb{R}^n: M[H](x)>c_02^j\}.
    \end{equation}
    Assume that $G_{j, m}(x)>0$. Then there is $k\in \mathbb{N}_0$ such that $A_{k-\underline{k}}|H|(x)>2^{j}$, and hence $x\in \Omega_j$. Consequently, $\operatorname{supp}G_{j, m}\subset \Omega_j$.
    \par
    Applying Theorem~\ref{Th.Whitney_decomposition} to the closed set $\mathbb R^n\setminus\Omega_j$, we obtain a Whitney decomposition $\{Q_\kappa\}_{\kappa \in I_j}$ of $\Omega_j$. Let $x\in Q_{\kappa}$ be such that $A_{k-\underline{k}}|H|(x)>2^{j}$. By property~\ref{cond:W2} of the Whitney decomposition (see also \eqref{eq.distance_to_projection} in Lemma~\ref{Lm.distance_estimates}), there exists $z_{\kappa} \in \mathbb{R}^n\setminus \Omega_j$ such that $|z_{\kappa} - x| \le 10\cdot 2^{-k(\kappa)}$. We claim that $2^{-k}\le 2^{-k(\kappa)}$. Assume to the contrary that $2^{-k}>2^{-k(\kappa)}$. Then $Q_{k-\underline{k}}(x) \subset c_1Q_k(z_{\kappa})$, where $c_1:=10+2^{\underline{k}}$ and in particular is independent of $x, \kappa$, and $k$. Therefore,
    \begin{equation}
        \label{eq.S_m_operator_bound5}
        M[H](z_{\kappa}) \ge \fint\limits_{c_1Q_k(z_{\kappa})}|H(y)|dy \ge c_0 \fint\limits_{Q_{k-\underline{k}}(x)}|H(y)|dy > c_02^{j},
    \end{equation}
    provided that $c_0$ is small enough. Thus, $z_\kappa \in \Omega_j$ contradicting $z_{\kappa} \in \mathbb{R}^n\setminus \Omega_j$. As a result, for each $x\in Q_{\kappa}$,
    \begin{equation}
        \label{eq.S_m_operator_bound6}
        G_{j, m}(x) \le \sup_{k\ge \max\{k(\kappa), 0\}}\mathcal{D}_{k, m}(x)=\mathcal{D}^*_{\max\{k(\kappa), 0\}, m}(x).
    \end{equation}
    If $k(\kappa)\ge 0$, then, by Corollary~\ref{Ca.modified_density_estimation}, we obtain
    \begin{equation}
        \label{eq.S_m_operator_bound7}
        \int\limits_{Q_{\kappa}}G_{j, m}(x)dx \lesssim (m-m_0)\mathcal{L}^n(Q_{\kappa}).
    \end{equation}
    If $k(\kappa)<0$, then the same estimate follows from the following argument. Divide $Q_{\kappa}$ into cubes with side length $2$ and with disjoint interiors. Applying Corollary~\ref{Ca.modified_density_estimation} on each small cube, we obtain \eqref{eq.S_m_operator_bound7}.
    Consequently,
    \begin{equation}
        \label{eq.S_m_operator_bound8}
        \int\limits_{\mathbb{R}^n}G_{j, m}(x)dx = \int\limits_{\Omega_j}G_{j, m}(x)dx = \sum_{\kappa \in I_j}\int\limits_{Q_{\kappa}}G_{j, m}(x)dx \lesssim (m-m_0)\mathcal{L}^n(\Omega_j).
    \end{equation}
    Thus, using the definition of $\Omega_j$, \eqref{eq.S_m_operator_bound1}, and Theorem~\ref{Th.maximal_function_boundedness}, we obtain
    \begin{equation}
    \begin{split}
        \label{eq.S_m_operator_bound9}
        \|S_{m, \underline{k}}^a[H]\|_{L_1(\mathbb{R}^n)} \lesssim(m-m_0) \sum_{j\in \mathbb{Z}}\quad\int\limits_{\mathbb{R}^n}2^{ja}\chi_{\{M[H]>c_02^{j}\}}(x)dx\\ \lesssim (m-m_0)\|M[H]\|_{L_a(\mathbb{R}^n)}^a\lesssim (m-m_0)\|H\|_{L_a(\mathbb{R}^n)}^a.
    \end{split}
    \end{equation}
    The proof is complete.
\end{proof}
\begin{Prop}
\label{Prop.T_m_operator_bound}
    Given $m_0\in \mathbb{Z}$, $\underline{k}\in \mathbb{N}_0$, and $a\in (1, \infty)$, there exists $C>0$ such that for each $m>m_0$ and $H\in L_a(\mathbb{R}^n)$,
    \begin{equation}
    \label{eq.T_m_operator_bound}
        \|T_{m, \underline{k}}[H]\|_{L_a(\mathbb{R}^n)} \le C (m-m_0)^{1/a}2^{m\theta\frac{a-1}{a}}\|H\|_{L_a(\mathbb{R}^n)}.
    \end{equation}
\end{Prop}
\begin{proof}
    The required assertion is easily reduced to Proposition~\ref{Prop.S_m_operator_bound}. Indeed, for all $x \in \mathbb{R}^n$,
    \begin{equation}
        \label{eq.T_m_operator_bound1}
        (T_{m, \underline{k}}[H](x))^a = \sup_{k\in\mathbb{N}_0}\mathcal{D}_{k, m}(x)^a\Bigl(A_{k-\underline{k}}|H|(x)\Bigr)^a.
    \end{equation}
    Furthermore, by the global upper estimate on $\mathcal{D}_{k, m}$, we obtain
    \begin{equation}
        \label{eq.T_m_operator_bound2}
        \mathcal{D}_{k, m}(x)^a \lesssim 2^{m\theta(a-1)}\mathcal{D}_{k, m}(x).
    \end{equation}
    Consequently,
    \begin{equation}
        (T_{m, \underline{k}}[H](x))^a  \lesssim 2^{m\theta(a-1)} S_{m, \underline{k}}^a[H](x),
    \end{equation}
    for all $x\in\mathbb{R}^n$. The application of Proposition~\ref{Prop.S_m_operator_bound} then yields \eqref{eq.T_m_operator_bound}, which complets the proof.
\end{proof}
We are now ready to prove the \emph{third main result} of this section.
\begin{Th}
    \label{Th.lizorkin_triebel_direct_q_le_p}
    Assume that $q<p<\infty$ and $\sigma \in (0, q]$. There exists $C>0$ such that for each $f\in F^s_{p, q}(\mathbb{R}^n)$,
    \begin{equation}
        \label{eq.lizorkin_triebel_direct_q_le_p}
        \|\Tr f\|_{\mathfrak{F}^{s-\theta/p}_{p, q, \sigma}(E)} \le C \|f\|_{F^s_{p, q}(\mathbb{R}^n)}.
    \end{equation}
\end{Th}
\begin{proof}
Fix $f\in F^s_{p, q}(\mathbb{R}^n)$ and let $\phi:=\Tr f$. 
\par
\emph{Step 1.} First, we estimate the $\mathfrak{f}^{s-\theta/p}_{p, q, \sigma}(E)$-functional of $\phi$. To this end, we define
\begin{equation}
\label{eq.lizorkin_triebel_direct_q_le_p1}
    \Phi(x) := \sum_{k=0}^{\infty}2^{ksq}\widetilde{\mathcal{E}}_{\fm_k, \sigma}(\phi, Q_k(x))^q, \qquad x\in\mathbb{R}^n.
\end{equation}
Let $R:=\frac{p}{q}$. Our goal is to establish that
\begin{equation}
\label{eq.lizorkin_triebel_direct_q_le_p2}
    \|\Phi\|_{L_R(\mathbb{R}^n)} \lesssim \|f\|_{F^s_{p, q}(\mathbb{R}^n)}^q.
\end{equation}
Since $R>1$, duality gives
\begin{equation}
\label{eq.lizorkin_triebel_direct_q_le_p3}
    \|\Phi\|_{L_R(\mathbb{R}^n)} = \sup_{\substack{H\ge0:\\\|H\|_{L_{R'}(\mathbb{R}^n)} = 1}}\quad\int\limits_{\mathbb{R}^n}\Phi(x)H(x)dx.
\end{equation}
Choose $\varepsilon \in (0, s)$ and $\varepsilon$-nonincreasing $\vec{h}\in \mathbb{D}^s(f)$ such that $\|\vec{h}\|_{L_p(\mathbb{R}^n, \ell_q)} \lesssim\|f\|_{F^s_{p, q}(\mathbb{R}^n)}$. By Lemma~\ref{Lm.pointw_est}, applied with $c=2$ and $t=q$, for $\fm_0$-a.e. $y\in E \cap 2Q_k(x)$, we have
\begin{equation}
\label{eq.lizorkin_triebel_direct_q_le_p4}
    |\phi(y)-\operatorname{med}(f, Q_k(x))| \lesssim \sum_{j=k-3}^{\infty}2^{-js}\Bigl(A_{j}h_j^q(y) \Bigr)^{1/q} =: T_k(y).
\end{equation}
Since $\sigma\le q$, Jensen's inequality gives, for each $k\in \mathbb{N}_0$ and $x\in \mathbb{R}^n$,
\begin{equation}
\label{eq.lizorkin_triebel_direct_q_le_p5}
    \widetilde{\mathcal{E}}_{\fm_k, \sigma}(\phi, Q_k(x))^q \lesssim \fint\limits_{2Q_k(x)}T_k(y)^qd\fm_k(y).
\end{equation}
Consequently,
\begin{equation}
\label{eq.lizorkin_triebel_direct_q_le_p6}
\begin{split}
    \int\limits_{\mathbb{R}^n}\Phi(x)H(x)dx \lesssim \sum_{k=0}^{\infty}2^{ksq+kd}\int\limits_{\mathbb{R}^n}H(x)\int\limits_{Q_{k-1}(x)}T_k(y)^qd\fm_k(y) \\ = \sum_{k=0}^{\infty}\quad\int\limits_{\mathbb{R}^n} (2^{ks}T_k(y))^q \fint\limits_{Q_{k-1}(y)} H(x)dx d\nu_k(y).
\end{split}
\end{equation}
Choose an arbitrary $\alpha \in (\frac{n-d}{p}, s)$. For each $k\in \mathbb{N}_0$, by H\"older's inequality when $q>1$ and by the subadditivity when $q\le 1$, we obtain
\begin{equation}
\label{eq.lizorkin_triebel_direct_q_le_p7}
    (2^{ks}T_k(y))^q \lesssim \sum_{j=k-3}^{\infty}2^{-(j-k)\alpha q} A_jh_j^q(y).
\end{equation}
Therefore,
\begin{equation}
\label{eq.lizorkin_triebel_direct_q_le_p8}
\begin{split}
    \int\limits_{\mathbb{R}^n}\Phi(x)H(x)dx \lesssim \sum_{k=0}^{\infty}\sum_{j=k-3}^{\infty}2^{-(j-k)\alpha q}\int\limits_{\mathbb{R}^n}A_jh_j^q(y)A_{k-1}H(y)d\nu_k(y) \\ = \sum_{k=0}^{\infty}\sum_{m=-3}^{\infty}2^{-m\alpha q}\int\limits_{\mathbb{R}^n}A_{k+m}h_{k+m}^q(y)A_{k-1}H(y)d\nu_k(y).
\end{split}
\end{equation}
\par
For each $m\ge -3$ and each $k\ge 0$, Fubini's theorem yields
\begin{equation}
\label{eq.lizorkin_triebel_direct_q_le_p9}
    \int\limits_{\mathbb{R}^n}A_{k+m}h_{k+m}^q(y)A_{k-1}H(y)d\nu_k(y) \le \int\limits_{\mathbb{R}^n} h_{k+m}(z)^q \int\limits_{Q_{k+m}(z)}\frac{A_{k-1}H(y)}{\mathcal{L}^n(Q_{k+m}(y))}d\nu_k(y)dz.
\end{equation}
Fix $z \in \mathbb{R}^n$. For each $y \in Q_{k+m}(z)$, we have $Q_{k-1}(y)\subset Q_{k-4}(z)$ (recall that $m\ge -3$). Therefore, $A_{k-1}H(y) \lesssim A_{k-4}H(z)$ for all $y \in Q_{k+m}(z)$, and hence
\begin{equation}
    \label{eq.lizorkin_triebel_direct_q_le_p10}
    \int\limits_{\mathbb{R}^n}A_{k+m}h_{k+m}^q(y)A_{k-1}H(y)d\nu_k(y)\lesssim \int\limits_{\mathbb{R}^n}h_{k+m}(z)^q \mathcal{D}_{k, m}(z)A_{k-4}H(z)dz.
\end{equation}
As a result, H\"older's inequality and Proposition~\ref{Prop.T_m_operator_bound} give
\begin{equation}
    \label{eq.lizorkin_triebel_direct_q_le_p11}
    \begin{split}
\sum_{k=0}^{\infty}\quad\int\limits_{\mathbb{R}^n}A_{k+m}h_{k+m}^q(y)A_{k-1}H(y)d\nu_k(y) \lesssim \int\limits_{\mathbb{R}^n}\sum_{k=0}^{\infty}h_{k+m}(z)^q 
T_{m, 4}[H](z)dz \\ \le \|\vec{h}\|_{L_p(\mathbb{R}^n, \ell_q)}^q \|T_{m, 4}[H]\|_{L_{R'}(\mathbb{R}^n)} \lesssim (m+4)^{1-\frac{q}{p}}2^{m\theta \frac{q}{p}}\|H\|_{L_{R'}(\mathbb{R}^n)}\|\vec{h}\|_{L_p(\mathbb{R}^n, \ell_q)}^q
\end{split}    
\end{equation}
Combining the latter estimate with \eqref{eq.lizorkin_triebel_direct_q_le_p8} yields
\begin{equation}
    \begin{split}
        \int\limits_{\mathbb{R}^n}\Phi(x)H(x)dx &\lesssim \|H\|_{L_{R'}(\mathbb{R}^n)}\|\vec{h}\|_{L_p(\mathbb{R}^n, \ell_q)}^q\sum_{m=-3}^{\infty}(m+4)^{1-\frac{q}{p}}2^{-m(\alpha - \frac{\theta}{p})q}\\& \lesssim \|H\|_{L_{R'}(\mathbb{R}^n)}\|f\|_{F^s_{p, q}(\mathbb{R}^n)}^q.
    \end{split}
\end{equation}
Thus, by \eqref{eq.lizorkin_triebel_direct_q_le_p3}, we obtain 
\begin{equation}
    \label{eq.lizorkin_triebel_direct_q_le_p12}
    \|\phi\|_{\mathfrak{f}^{s-\theta/p}_{p, q, \sigma}(E)} \lesssim \|f\|_{F^s_{p, q}(\mathbb{R}^n)}.
\end{equation}
\par
\emph{Step 2.} Now we estimate the $L_p(\fm_0)$-norm of the trace. Since $q<p$, it follows that
\begin{equation}
    \|\vec{h}\|_{\ell_p(L_p(\mathbb{R}^n))} =  \|\vec{h}\|_{L_p(\mathbb{R}^n, \ell_p)} \le \|\vec{h}\|_{L_p(\mathbb{R}^n, \ell_q)}.
\end{equation}
Therefore, by \eqref{eq.direct_estimate2} in Theorem~\ref{Th.main_direct_Besov}, applied with $q=p$, we obtain
\begin{equation}
\label{eq.lizorkin_triebel_direct_q_le_p13}
    \|\phi\|_{L_p(\fm_0)}\lesssim \|f\|_{L_p(\mathbb{R}^n)} + \|\vec{h}\|_{\ell_p(L_p(\mathbb{R}^n))} \lesssim \|f\|_{F^s_{p, q}(\mathbb{R}^n)}
\end{equation}
Combining \eqref{eq.lizorkin_triebel_direct_q_le_p12} and \eqref{eq.lizorkin_triebel_direct_q_le_p13} completes the proof.
\end{proof}
\section{Proofs of the main results}
\label{section.proof_of_the_main_results}
Throughout this section we fix:
\begin{conditions}{\textbf{D}.6.}
    \item\label{cond:D61} a parameter $d\in [0,n]$ and a closed $d$-thick set
    $E\subset\mathbb{R}^n$;
    \item\label{cond:D62} a strongly $d$-regular sequence of measures
    $\{\fm_k\}_{k=0}^{\infty}$ on $E$;
    \item\label{cond:D63} parameters $p\in [1,\infty]$, $q\in (0,\infty]$, and
    $s\in \left(\frac{n-d}{p},1\right)$.
\end{conditions}
\par
We prove Theorems~\ref{Th.main_stated_Besov} and~\ref{Th.main_stated_LT}.
\begin{proof}[Proof of Theorem~\ref{Th.main_stated_Besov}]
Fix $\sigma \in (0, \infty)$ such that $\sigma\le p$. If $\phi \in B^s_{p, q}(\mathbb{R}^n)\big|_E^{\fm_0}$, then Theorem~\ref{Th.main_direct_Besov} gives
\begin{equation}
    \|\phi\|_{\mathfrak{B}^{s-\theta/p}_{p, q, \sigma}(E)} \lesssim \inf\{\|f\|_{B^s_{p, q}(\mathbb{R}^n)}:f\in B^s_{p, q}(\mathbb{R}^n), \phi=\Tr f\} = \|\phi\|_{B^s_{p, q}(\mathbb{R}^n)\big|_E^{\fm_0}}.
\end{equation}
For each $k\in \mathbb{N}_0$ and each $x\in E$ define $\mathcal{M}^x_k:L_{\sigma}(Q_k(x), \fm_k)\to \mathbb{R}$ by
\begin{equation}
    \mathcal{M}^x_k\phi := \operatorname{med}_{\fm_k}(\phi, Q_k(x)).
\end{equation}
By Lemma~\ref{Lm.average_median_approximation_property}, there is $\lambda=\lambda(\sigma)$ such that the family $\mathcal{M} := \{\mathcal{M}^x_k\}$ is $\lambda $-almost best approximating. Hence, for $\phi \in L_{\sigma}^{\loc}(\{\fm_k\})$, Theorem~\ref{Th.main_inverse_Besov} and Theorem~\ref{Th.extension_property} yield
\begin{equation}
    \|\phi\|_{B^s_{p, q}(\mathbb{R}^n)\big|_E^{\fm_0}} \le \|\operatorname{Ext}_{\mathcal{M}}\phi\|_{B^s_{p, q}(\mathbb{R}^n)} \lesssim \|\phi\|_{\mathfrak{B}^{s-\theta/p}_{p, q, \sigma}(E)}.
\end{equation}
Thus, we proved the equivalence in Theorem~\ref{Th.main_stated_Besov}. In particular, for each $\sigma\le p$ functionals $\mathfrak{B}^{s-\theta/p}_{p, q, \sigma}(E)$ and $\mathfrak{B}^{s-\theta/p}_{p, q, 1}(E)$ are equivalent.
\par
To construct a bounded linear extension operator, we define $\mathcal{A} = \{\mathcal{A}^x_k\}_{x\in E, k\in \mathbb{N}_0}$ by
\begin{equation}
    \mathcal{A}^x_k\phi = A_{k, \fm_k}\phi(x).
\end{equation}
By Lemma~\ref{Lm.average_median_approximation_property}, each $\mathcal{A}^x_k$ is $\lambda$-almost best approximating on $L_1(Q_k(x), \fm_k)$ for some $\lambda>1$. Therefore,
\begin{equation}
    \|\Ext_{\mathcal{A}}\phi\|_{B^s_{p, q}(\mathbb{R}^n)} \lesssim \|\phi\|_{\mathfrak{B}^{s-\theta/p}_{p, q, 1}(E)} \lesssim \|\phi\|_{\mathfrak{B}^{s-\theta/p}_{p, q, \sigma}(E)}.
\end{equation}
Since each $\mathcal{A}^x_k$ is linear, $\operatorname{Ext}_{A}$ is linear.
The proof is complete.
\end{proof}
\begin{proof}[Proof of Theorem~\ref{Th.main_stated_LT}]
    The same argument applies, using Theorems~\ref{Th.main_direct_Lizorkin_Triebel}, \ref{Th.lizorkin_triebel_direct_q_le_p}, and~\ref{Th.main_inverse_Lizorkin_Triebel}. If $p>1$ and $q\ge 1$, then we can choose $\sigma = 1$. Hence, integral averages give a family of linear almost best approximating operators. Therefore, the extension operator associated with this family gives a bounded linear extension.
\end{proof}
\section{Examples}
\label{section.Examples}
The goal of this section is to recover the classical trace results for Ahlfors--David regular sets from
Theorems~\ref{Th.main_stated_Besov} and~\ref{Th.main_stated_LT} and to provide several other examples of interest.
\subsection{Traces to Ahlfors--David regular sets}
The following result can be found in \cite[Example~5.1]{tyul} and \cite[Lemma~2.2]{tyulenev_thick_sets_R_n}.
\begin{Lm}
\label{Lm.regular_hence_thick}
    Let $d\in [0, n]$ and let $E\subset\mathbb{R}^n$ be Ahlfors--David $d$-regular. Then $E$ is $d$-thick. Furthermore, the sequence of measures $\fm_k:=\mathcal{H}^d\lfloor_E$, $k\in\mathbb{N}_0$, is strongly $d$-regular.
\end{Lm}
First we recall the oscillation-based definition of Besov and Lizorkin--Triebel spaces (see, e.g., \cite{AKZ}).
\begin{Def}
     Let $E\subset\mathbb{R}^n$ be Ahlfors--David $d$-regular,
    $d\in(0,n]$, and let $s\in(0,1)$.
    Given $p,q\in(0,\infty]$, the Besov space $B^s_{p,q}(E)$
    consists of all $f\in L_p(\mathcal{H}^d\lfloor_E)$ such that
    \begin{equation}
    \label{eq.besov_seminorm_ahlfors_regular}
        \|f\|_{b^s_{p, q, \sigma}(E)} := \|\{2^{ks}\mathcal{E}_{\mathcal{H}^d\lfloor_E, \sigma}(f, Q_{k}(\cdot))\}_{k=0}^{\infty}\|_{\ell_q(L_p(\mathcal{H}^d\lfloor_E))}<\infty
    \end{equation}
    for some (equivalently, every) $\sigma\in(0,p]$.
    For every such $\sigma$, an equivalent quasi-norm on
    $B^s_{p,q}(E)$ is given by
    \begin{equation}
    \label{eq.besov_norm_ahlfors_regular}
        \|f\|_{B^s_{p, q, \sigma}(E)} := \|f\|_{L_p(\mathcal{H}^d\lfloor_E)} + \|f\|_{b^s_{p, q, \sigma}(E)}.
    \end{equation}
    \par
     If, in addition, $p<\infty$, the Lizorkin--Triebel space
    $F^s_{p,q}(E)$ consists of all
    $f\in L_p(\mathcal{H}^d\lfloor_E)$ such that
    \begin{equation}
    \label{eq.LT_seminorm_ahlfors_regular}
        \|f\|_{f^s_{p, q, \sigma}(E)} := \|\{2^{ks}\mathcal{E}_{\mathcal{H}^d\lfloor_E, \sigma}(f, Q_{k}(\cdot))\}_{k=0}^{\infty}\|_{L_p(\mathcal{H}^d\lfloor_E, \ell_q)}<\infty
    \end{equation}
    for some (equivalently, every)
    $\sigma\in(0,\min\{p,q\})$.
    For every such $\sigma$, an equivalent quasi-norm on
    $F^s_{p,q}(E)$ is given by
    \begin{equation}
    \label{eq.LT_norm_ahlfors_regular}
        \|f\|_{F^s_{p, q, \sigma}(E)} := \|f\|_{L_p(\mathcal{H}^d\lfloor_E)} + \|f\|_{f^s_{p, q, \sigma}(E)}.
    \end{equation}
    When the parameter $\sigma$ is immaterial, we will suppress it from the notation of the corresponding norms.
\end{Def}
\begin{Lm}
\label{Lm.local_approximation_on_regular}
    Assume that $d\in (0, n)$ and that $E\subset\mathbb{R}^n$ is Ahlfors--David $d$-regular. Given $p\in [1, \infty]$ and $\sigma \in (0, \infty)$, there exists $C>0$ such that for each $\phi \in L_{\sigma}^{\loc}(\mathcal{H}^d\lfloor_E)$ and each $k \in \mathbb{N}_0$,
    \begin{equation}
        \label{eq.local_approximation_on_regular}
        \| \widetilde{\mathcal{E}}_{\mathcal{H}^d\lfloor_E, \sigma}(\phi, Q_k(\cdot))\|_{L_p(\mathbb{R}^n)} \le C 2^{-k\frac{n-d}{p}}\| {\mathcal{E}}_{\mathcal{H}^d\lfloor_E, \sigma}(\phi, Q_{k-4}(\cdot))\|_{L_p(\mathcal{H}^d\lfloor_E)}.
    \end{equation}
    Moreover,
\begin{equation}
\label{eq.local_approximation_on_regular_lower}
2^{-k\frac{n-d}{p}}
\|\mathcal{E}_{\mathcal{H}^d\lfloor_E,\sigma}(\phi,Q_k(\cdot))\|_{L_p(\mathcal{H}^d\lfloor_E)}
\le
C
\|\widetilde{\mathcal{E}}_{\mathcal{H}^d\lfloor_E,\sigma}
(\phi,Q_k(\cdot))\|_{L_p(\mathbb{R}^n)}.
\end{equation}
\end{Lm}
\begin{proof}
    Fix $\phi \in L_{\sigma}^{\loc}(\mathcal{H}^d\lfloor_E)$ and $k\in \mathbb{N}_0$. For $j \in \mathbb{Z}$, we put $\Phi_j(x):= \widetilde{\mathcal{E}}_{\mathcal{H}^d\lfloor_E, \sigma}(\phi, Q_j(x))$ and $E_j(x):=\mathcal{E}_{\mathcal{H}^d\lfloor_E, \sigma}(\phi, Q_j(x))$ for brevity. From the definition, it is clear that $\operatorname{supp} \Phi_j \subset \overline{U_j(E)}$.
    Let $\{Q_{\kappa} = Q_{r_{\kappa}}(x_{\kappa})\}_{\kappa \in I}$ be a Whitney decomposition of $\mathbb{R}^n\setminus E$. We begin with \eqref{eq.local_approximation_on_regular}.
    \par
    \emph{Step 1.} Assume first that $p<\infty$. Since $\operatorname{supp} \Phi_j \subset \overline{U_j(E)}$ and $\mathcal{L}^n(E) = 0$ if $E$ is $d$-regular with $d<n$, it follows from property~\ref{cond:W1} of the Whitney decomposition that
    \begin{equation}
        \label{eq.local_approximation_on_regular1}
        \|\Phi_k\|_{L_p(\mathbb{R}^n)}^p = \sum_{\substack{\kappa \in I:\\r_{\kappa}\le 2^{-k}}}\int\limits_{Q_{\kappa}} \Phi_k(x)^pdx.
    \end{equation}
     Using \eqref{eq.distance_to_projection}, we get, for every $x\in Q_{\kappa}$ and every $y\in \widetilde{Q}_{\kappa}$,
    \begin{equation}
        \label{eq.local_approximation_on_regular2}
        |x-y| \le |x-x_{\kappa}| + |x_{\kappa}-\widetilde{x}_{\kappa}|+|\widetilde{x}_{\kappa}-y| \le 11\cdot 2^{-k(\kappa)} \le 11\cdot 2^{-k}.
    \end{equation}
    Consequently, arguing as in Lemma~\ref{Lm.local_approximation_outside_E}, we obtain
    \begin{equation}
        \label{eq.local_approximation_on_regular3}
        \Phi_k(x) \lesssim \inf_{y\in \widetilde{Q}_{\kappa}} {\mathcal{E}}_{\mathcal{H}^d\lfloor_E, \sigma}(\phi, 16Q_{k}(y)) = \inf_{y\in \widetilde{Q}_{\kappa}} E_{k-4}(y).
    \end{equation}
    We remark that for every fixed $k$, the family $\{\widetilde{Q}_{\kappa}\}_{k(\kappa) = k}$ has uniformly bounded multiplicity. Averaging the latter estimate over $\widetilde{Q}_{\kappa}$ with respect to $\mathcal{H}^d\lfloor_E$ then gives
    \begin{equation}
        \label{eq.local_approximation_on_regular4}
        \begin{split}
        &\|\Phi_k\|_{L_p(\mathbb{R}^n)}^p  \lesssim \sum_{\substack{\kappa \in I:\\r_{\kappa}\le 2^{-k}}}2^{-k(\kappa)n}\fint\limits_{\widetilde{Q}_{\kappa}} E_{k-4}(y)^pd\mathcal{H}^d\lfloor_E(y) \\= & \sum_{\substack{\kappa \in I:\\r_{\kappa}\le 2^{-k}}}2^{-k(\kappa)(n-d)}\int\limits_{\widetilde{Q}_{\kappa}} E_{k-4}(y)^pd\mathcal{H}^d\lfloor_E(y)\lesssim 2^{-k(n-d)}\|E_{k-4}\|_{L_p(\mathcal{H}^d\lfloor_E)}^p.
        \end{split}
    \end{equation}
    \par
    \emph{Step 2.} Now assume that $p=\infty$. Then \eqref{eq.local_approximation_on_regular3} yields 
    \begin{equation}
        \Phi_k(x) \lesssim \|E_{k-4}\|_{L_{\infty}(\mathcal{H}^d\lfloor_E)}
    \end{equation}
    for all $x\in Q_{\kappa}$, where $r_{\kappa}\le 2^{-k}$. Thus, we obtain \eqref{eq.local_approximation_on_regular}. 
    \par
    \emph{Step 3.} We now prove \eqref{eq.local_approximation_on_regular_lower}. The proof follows the same argument. For $x\in E$ and $y\in Q_k(x)$, Lemma~\ref{Lm.local_approximation_outside_E} gives $E_k(x)\lesssim \Phi_k(y)$. Averaging over $y \in Q_k(x)$ with respect to $\mathcal{L}^n$ then yields
    \begin{equation}
        E_k(x)\lesssim \fint\limits_{Q_k(x)}\Phi_k(y)dy.
    \end{equation}
    Consequently, by Jensen's inequality and Fubini's theorem when $p<\infty$, and trivially when $p=\infty$, we obtain
    \begin{equation}
        \int\limits_{E}E_k(x)^pd\mathcal{H}^d\lfloor_E(x) \lesssim \int\limits_{\mathbb{R}^n}\Phi_k(y)^p \int\limits_{Q_k(y)}\frac{d\mathcal{H}^d\lfloor_E(x)}{2^{-kn}}dy \lesssim 2^{k(n-d)}\|\Phi_k\|_{L_p(\mathbb{R}^n)}^p.
    \end{equation}
    Thus, \eqref{eq.local_approximation_on_regular_lower} follows immediately. The proof is complete.
\end{proof}
We are ready to recover the trace result for Besov spaces (compare, e.g., \cite{saks}).
\begin{Th}
    \label{Th.Besov_on_regular}
    Let $E\subset \mathbb{R}^n$ be Ahlfors--David $d$-regular, $d\in (0, n)$. Assume that $p\in [1, \infty]$, $q\in (0, \infty]$, $s\in (\frac{n-d}{p}, 1)$. Then $\phi \in L_p(\mathcal{H}^d\lfloor_E)$ belongs to the space $B^s_{p, q}(\mathbb{R}^n)\big|_E^{\mathcal{H}^d\lfloor_E}$ if and only if $\phi \in B^{s-(n-d)/p}_{p, q}(E)$.
    Furthermore, $ \|\phi\|_{B^s_{p, q}(\mathbb{R}^n)\big|_E^{\mathcal{H}^d\lfloor_E}}\approx \|\phi\|_{{B}^{s-(n-d)/p}_{p, q}(E)}$.  Finally, there exists a bounded linear extension operator $\Ext_{\mathcal{H}^d\lfloor_E}: B^s_{p, q}(\mathbb{R}^n)\big|_E^{\mathcal{H}^d\lfloor_E} \to B^s_{p, q}(\mathbb{R}^n)$.
\end{Th}
\begin{proof}
    For each $\phi \in L_{p}(\mathcal{H}^d\lfloor_E)$, Lemma~\ref{Lm.local_approximation_on_regular} gives
    \begin{equation}
    \begin{split}
        \|\phi\|_{\mathfrak{B}^{s-\theta/p}_{p, q, 1}(E)} = \|\phi\|_{L_p(\mathcal{H}^d\lfloor_E)}+ \|\{2^{ks}\widetilde{\mathcal{E}}_{\mathcal{H}^d\lfloor_E}(\phi, Q_k(\cdot))\}_{k=0}^{\infty}\|_{\ell_q(L_p(\mathbb{R}^n))} \approx\\ \|\phi\|_{L_p(\mathcal{H}^d\lfloor_E)}+\|\{{2^{k(s-\frac{\theta}{p})}\mathcal{E}}_{\mathcal{H}^d\lfloor_E}(\phi, Q_k(\cdot))\}_{k=0}^{\infty}\|_{\ell_q(L_p(\mathcal{H}^d\lfloor_E))} \approx \|\phi\|_{B^{s-\theta/p}_{p, q, 1}(E)},
    \end{split}
    \end{equation}
   where finitely many additional low-frequency terms are absorbed by $\|\phi\|_{L_p(\mathcal{H}^d\lfloor_E)}$.
    Therefore, Theorem~\ref{Th.main_stated_Besov} yields the desired result.
\end{proof}
Now we consider the Lizorkin--Triebel case.
The following porosity property of Ahlfors--David regular sets is well-known (see, e.g., \cite[Proposition~9.18]{Triebel_porousity}).
\begin{Lm}
\label{Lm.regular_set_porosity}
Let $E\subset\mathbb{R}^n$ be Ahlfors--David $d$-regular,
$d<n$. Then there exists $\eta\in(0,1)$ such that for every
$x\in E$ and every $r\in(0,1]$, there exists $y\in Q_r(x)$ such
that $Q_{\eta r}(y)\subset Q_{r}(x)\setminus E$.
\end{Lm}
The argument used below for the key estimates in the Lizorkin--Triebel case is inspired by \cite[Section~6.2]{Ihnat}.
\begin{Lm}
\label{Lm.LT_on_regular}
Let $p\in [1, \infty)$, $q\in (0, \infty]$, and $\sigma \in (0, \infty)$.
    Assume that $d\in (0, n)$, and $E\subset\mathbb{R}^n$ is Ahlfors--David $d$-regular. If $s\in (\frac{n-d}{p}, 1)$, then there exists $C>0$ such that for each $\phi \in L_{\sigma}^{\loc}(\mathcal{H}^d\lfloor_E)$,
    \begin{equation}
        \label{eq.LT_on_regular}
        \|\{2^{ks} \widetilde{\mathcal{E}}_{\mathcal{H}^d\lfloor_E, \sigma}(\phi, Q_k(\cdot))\}_{k=0}^{\infty}\|_{L_p(\mathbb{R}^n, \ell_q)} \le C \|\{2^{k(s-\frac{\theta}{p})}{\mathcal{E}}_{\mathcal{H}^d\lfloor_E, \sigma}(\phi, Q_k(\cdot))\}_{k=-3}^{\infty}\|_{L_p(\mathcal{H}^d\lfloor_E, \ell_p)}
    \end{equation}
    and 
    \begin{equation}
        \label{eq.LT_on_regular_lower}
       \|\{2^{k(s-\frac{\theta}{p})}{\mathcal{E}}_{\mathcal{H}^d\lfloor_E, \sigma}(\phi, Q_k(\cdot))\}_{k=3}^{\infty}\|_{L_p(\mathcal{H}^d\lfloor_E, \ell_p)}\le C \|\{2^{ks} \widetilde{\mathcal{E}}_{\mathcal{H}^d\lfloor_E, \sigma}(\phi, Q_k(\cdot))\}_{k=0}^{\infty}\|_{L_p(\mathbb{R}^n, \ell_q)}.
    \end{equation}
\end{Lm}
\begin{proof}
Fix $\phi \in L_{\sigma}^{\loc}(\mathcal{H}^d\lfloor_E)$. For $j \in \mathbb{Z}$, we put $\Phi_j(x):= \widetilde{\mathcal{E}}_{\mathcal{H}^d\lfloor_E, \sigma}(\phi, Q_j(x))$ and $E_j(x):=\mathcal{E}_{\mathcal{H}^d\lfloor_E, \sigma}(\phi, Q_j(x))$ for brevity. Furthermore, we set $\Phi:=\|\{2^{js}\Phi_j\}_{j=0}^{\infty}\|_{\ell_q}$. Let $\{Q_{\kappa} = Q_{r_{\kappa}}(x_{\kappa})\}_{\kappa \in I}$ be a Whitney decomposition of $\mathbb{R}^n\setminus E$. We begin with \eqref{eq.LT_on_regular}.
    \par
    \emph{Step 1.} Take an arbitrary $\kappa \in I$ such that $k(\kappa)\ge 0$. Since $\operatorname{supp} \Phi_j \subset \overline{U_j(E)}$, it follows by property~\ref{cond:W2} of the Whitney decomposition that, for each $x\in Q_{\kappa}$,
    \begin{equation}
        \label{eq.LT_on_regular1}
        \Phi(x) = \|\{2^{js}\Phi_j(x)\}_{j=0}^{k(\kappa)-1}\|_{\ell_q}.
    \end{equation}
    Using \eqref{eq.distance_to_projection}, we get, for every $j\in \{0, \ldots, k(\kappa)-1\}$, $x\in Q_{\kappa}$ and $y\in \widetilde{Q}_{\kappa}$,
    \begin{equation}
        \label{eq.LT_on_regular2}
        |x-y| \le |x-x_{\kappa}| + |x_{\kappa}-\widetilde{x}_{\kappa}|+|\widetilde{x}_{\kappa}-y| \le 11\cdot 2^{-k(\kappa)} \le 6\cdot 2^{-j}.
    \end{equation}
    Consequently, arguing as in Lemma~\ref{Lm.local_approximation_on_regular}, we obtain
    \begin{equation}
        \label{eq.LT_on_regular3}
        \Phi(x) \lesssim \inf_{y\in \widetilde{Q}_{\kappa}} \|\{2^{js}E_{j-3}(y)\}_{j=0}^{k(\kappa)-1}\|_{\ell_q}.
    \end{equation}
    Averaging the latter estimate over $\widetilde{Q}_{\kappa}$ with respect to $\mathcal{H}^d\lfloor_E$ then gives
    \begin{equation}
        \label{eq.LT_regular4}
        \begin{split}
        \|\Phi\|_{L_p(Q_{\kappa})}^p \lesssim 2^{-k(\kappa)(n-d)}\|\|\{2^{js}E_{j-3}(y)\}_{j=0}^{k(\kappa)-1}\|_{\ell_q}\|_{L_p(\widetilde{Q}_{\kappa}, \mathcal{H}^d\lfloor_E)}^p.
        \end{split}
    \end{equation}
    Since $\operatorname{supp}\Phi \subset \overline{U_0(E)}$, we consider only $\kappa\in I$ with $k(\kappa)\ge 1$.
    Summing over the Whitney cubes and taking into account that $\mathcal{L}^n(E) = 0$, we obtain
    \begin{equation}
        \label{eq.LT_regular5}
        \|\Phi\|_{L_p(\mathbb{R}^n)}^p \lesssim \sum_{k=0}^{\infty} 2^{-k(n-d)}\int\limits_{E}\|\{2^{js}E_{j-3}(y)\}_{j=0}^{k-1}\|_{\ell_q}^pd\mathcal{H}^d\lfloor_E(y).
    \end{equation} 
    Using the Hardy inequality when $p\ge q$ and subadditivity when $p<q$, we obtain
    \begin{equation}
    \label{eq.LT_regular6}
        \sum_{k=0}^{\infty} 2^{-k(n-d)}\|\{2^{js}E_{j-3}(y)\}_{j=0}^{k-1}\|_{\ell_q}^p \lesssim \|\{2^{j(s-\frac{n-d}{p})}E_{j-3}(y)\}_{j=0}^{\infty}\|_{\ell_p}^p.
    \end{equation}
    Combining \eqref{eq.LT_regular5} and \eqref{eq.LT_regular6} gives \eqref{eq.LT_on_regular}.
    \par
    \emph{Step 2.} We now prove \eqref{eq.LT_on_regular_lower}. Fix $k\ge3$. By the $5B$-covering lemma, there exists a
pairwise disjoint collection of cubes $Q_l^k:=Q_{2^{-k}}(x_{k,l})$, $x_{k,l}\in E$,
such that $E\subset\bigcup_l5Q_l^k$.
By Lemma~\ref{Lm.regular_set_porosity}, for every $l$ there exists
$y_{k,l}\in Q_l^k$ such that $Q_{\eta2^{-k}}(y_{k,l})\subset Q_l^k\setminus E$. We put $ R_l^k:=Q_{\frac{\eta}{2}2^{-k}}(y_{k,l})$.
Then $\mathcal{L}^n(R_l^k)\approx 2^{-kn}$,
and, for every $y\in R_l^k$,
\begin{equation}
    \frac{\eta}{2}2^{-k}
    \le \operatorname{dist}(y,E)
    \le 2^{-k}.
\end{equation}
     For every $y\in R^k_l$ and $x\in 5Q^k_l$, we clearly have $|x-y|\le 6 \cdot 2^{-k}$. Hence, Lemma~\ref{Lm.local_approximation_outside_E} gives $E_k(x) \lesssim \Phi_{k-3}(y)$.
    Averaging over $y \in R_l^k$ with respect to $\mathcal{L}^n$ then yields
    \begin{equation}
        E_{k}(x)\lesssim \fint\limits_{R_l^k}\Phi_{k-3}(y)dy.
    \end{equation}
    Since, for each $k\ge 3$, the family $\{5Q_{l}^k\}_{l}$ forms a covering of $E$ with uniformly bounded multiplicity, we obtain
    \begin{equation}
        \|E_k\|_{L_p(\mathcal{H}^d\lfloor_E)}^p \lesssim \sum_l \int\limits_{5Q^k_l}E_k(x)^pd\mathcal{H}^d\lfloor_E(x) \lesssim 2^{-k(d-n)}\sum_{l}  \int\limits_{R_l^k}\Phi_{k-3}(y)^pdy.
    \end{equation}
    If $y\in R^k_l$, then $\frac{\eta}{2}2^{-k}\le\operatorname{dist}(y, E) \le 2^{-k}$. Furthermore, for every fixed $k\ge 3$, the cubes $\{R^k_l\}_l$ are pairwise disjoint. Therefore, the family $\{R^k_l\}_{k\ge 3, l}$ has bounded multiplicity.
    Consequently,
    \begin{equation}
    \begin{split}
        \sum_{k=3}^{\infty}2^{k(ps-\theta )}\|E_k\|_{L_p(\mathcal{H}^d\lfloor_E)}^p \lesssim \sum_{k=3}^{\infty} 2^{kps}\sum_{l}\int\limits_{R^k_l} \Phi_{k-3}(y)^pdy \\ \lesssim \int\limits_{\mathbb{R}^n}\Bigl(\sup_{k\ge 0}2^{ks}\Phi_k(y) \Bigr)^pdy \le \int\limits_{\mathbb{R}^n}\Phi(y)^pdy.
    \end{split}
    \end{equation}
    The proof is complete.
\end{proof}
In particular, for $s\in (\frac{\theta}{p}, 1)$, $p\in [1, \infty)$, $q\in (0, \infty]$, and $\sigma \in (0, \min\{p, q\})$, the latter lemma yields
\begin{equation}
    \|\phi\|_{\mathfrak{F}^{s-\theta/p}_{p, q, \sigma}} \approx \|\phi\|_{B^{s-\theta/p}_{p, p, \sigma}(E)}.
\end{equation}
Combining this observation with Theorem~\ref{Th.main_stated_LT}, we obtain the following result (compare, e.g., \cite{saks}).
\begin{Th}
    \label{Th.LT_on_regular}
    Let $E\subset \mathbb{R}^n$ be Ahlfors--David $d$-regular, $d\in (0, n)$. Assume that $p\in [1, \infty)$, $q\in (0, \infty]$, $s\in (\frac{n-d}{p}, 1)$. Then $\phi \in L_p(\mathcal{H}^d\lfloor_E)$ belongs to the space $F^s_{p, q}(\mathbb{R}^n)\big|_E^{\mathcal{H}^d\lfloor_E}$ if and only if $\phi \in B^{s-(n-d)/p}_{p, p}(E)$.
    Furthermore, $ \|\phi\|_{F^s_{p, q}(\mathbb{R}^n)\big|_E^{\mathcal{H}^d\lfloor_E}}\approx \|\phi\|_{{B}^{s-(n-d)/p}_{p, p}(E)}$.  Finally, there exists a bounded extension operator $\Ext_{\mathcal{H}^d\lfloor_E}: F^s_{p, q}(\mathbb{R}^n)\big|_E^{\mathcal{H}^d\lfloor_E} \to F^s_{p, q}(\mathbb{R}^n)$. If $p>1$ and $q\ge1$, the extension operator can be chosen to be bounded and linear.
\end{Th}
We emphasize that our main results also cover traces to Ahlfors--David $n$-regular sets, and give a linear extension operator for $p\in [1, \infty]$, $q\in (0, \infty]$ in the Besov case and for $p\in (1, \infty)$, $q\in [1, \infty]$ in the Lizorkin--Triebel space. These facts are related to results in \cite{Heikkinen, Rychkov, shvartsman_regular}. In this case, the trace norm can also be made more transparent for some parameters.
\begin{Lm}
    \label{Lm.trace_to_n_regular}
    Let $s\in (0, 1)$, $p\in [1, \infty]$, and $q\in (0, \infty]$, and let $E \subset \mathbb{R}^n$ be Ahlfors--David $n$-regular. Let $\fm_k := \mathcal{H}^n\lfloor_E$, $k\in \mathbb{N}_0$, be the corresponding strongly $n$-regular sequence of measures. Then for each $\phi \in L_p(\mathcal{H}^n\lfloor_E)$ and each $\sigma \in (0, \infty)$ satisfying $\sigma \le p$,
    \begin{equation}
        \|\phi\|_{\mathfrak{B}^{s}_{p, q, \sigma}(E)} \approx \|\phi\|_{B^s_{p, q, \sigma}(E)}.
    \end{equation}
    If, in addition, $p\in (1, \infty)$, $q\in (1, \infty]$, and $\sigma <\min\{p, q\}$, then
    \begin{equation}
        \|\phi\|_{\mathfrak{F}^s_{p, q, \sigma}(E)} \approx \|\phi\|_{F^s_{p, q, \sigma}(E)}.
    \end{equation}
\end{Lm}
\begin{proof}
    The proof follows the arguments used in Lemma~\ref{Lm.local_approximation_on_regular} and Lemma~\ref{Lm.LT_on_regular} with $\{U_{\kappa}\}$ from Theorem~\ref{Th.substitute_of_projections} instead of $\{\widetilde{Q}_{\kappa}\}$ in the upper estimate. For the reverse Lizorkin--Triebel estimate, the porosity argument used in Step~2 of Lemma~\ref{Lm.LT_on_regular} is unnecessary. Instead, for every \(k\in\mathbb N_0\) and \(x\in E\), we have
    \begin{equation}
        \mathcal{E}_{\mathcal{H}^n\lfloor_{E}, \sigma}(\phi, Q_{k}(x))  \lesssim \fint\limits_{Q_k(x)} \widetilde{\mathcal{E}}_{\mathcal{H}^n\lfloor_E, \sigma}(\phi, Q_{k}(y))dy.
    \end{equation}
   The required estimate then follows from the Fefferman--Stein inequality.
\end{proof}
Combining the preceding lemma with Theorem~\ref{Th.main_stated_Besov} and~\ref{Th.main_stated_LT}, we obtain the following result (compare, e.g., \cite{shvartsman_regular}).
\par
\begin{Th}
     Let $E\subset \mathbb{R}^n$ be Ahlfors--David $n$-regular. Assume that $p\in [1, \infty]$, $q\in (0, \infty]$, $s\in (0, 1)$. Then $\phi \in L_p(\mathcal{H}^n\lfloor_E)$ belongs to the space $B^s_{p, q}(\mathbb{R}^n)\big|_E^{\mathcal{H}^n\lfloor_E}$ if and only if $\phi \in B^{s}_{p, q}(E)$.
    Furthermore, $ \|\phi\|_{B^s_{p, q}(\mathbb{R}^n)\big|_E^{\mathcal{H}^n\lfloor_E}}\approx \|\phi\|_{{B}^{s}_{p, q}(E)}$. Finally, there exists a bounded linear extension operator $\Ext_{\mathcal{H}^n\lfloor_E}: B^s_{p, q}(\mathbb{R}^n)\big|_E^{\mathcal{H}^n\lfloor_E} \to B^s_{p, q}(\mathbb{R}^n)$. 
    \par
    Assume additionally that $p, q>1$. Then $\phi \in L_p(\mathcal{H}^n\lfloor_E)$ belongs to the space $F^s_{p, q}(\mathbb{R}^n)\big|_E^{\mathcal{H}^n\lfloor_E}$ if and only if $\phi \in F^{s}_{p, q}(E)$.
    Furthermore, $ \|\phi\|_{F^s_{p, q}(\mathbb{R}^n)\big|_E^{\mathcal{H}^n\lfloor_E}}\approx \|\phi\|_{{F}^{s}_{p, q}(E)}$. Finally, there exists a bounded linear extension operator $\Ext_{\mathcal{H}^n\lfloor_E}: F^s_{p, q}(\mathbb{R}^n)\big|_E^{\mathcal{H}^n\lfloor_E} \to F^s_{p, q}(\mathbb{R}^n)$. 
\end{Th}
\subsection{Other examples}
Here we briefly discuss several other applications of the general results obtained above. Two main issues arise when applying Theorems~\ref{Th.main_stated_Besov} and~\ref{Th.main_stated_LT}. First, constructing a strongly \(d\)-regular sequence of measures may not be immediate. Second, one would like to replace the trace norm appearing in our main results by a more explicit equivalent intrinsic norm (see, e.g., \cite{tyul_picewise_regular}). The first two examples below were considered in \cite{tyul, tyulenev_thick_sets_R_n}, where explicit regular sequences of measures were constructed.
\par
\textit{Example 1.} In this example, we consider \emph{the cusp}. Let $\beta:[0, \infty) \to [0, \infty)$ be a strictly increasing continuous function such that $\beta(0) = 0$. Let $\beta^{-1}$ be its inverse function. Then the cusp
\begin{equation}
    G^{\beta} = \{x = (x', x_n) \in \mathbb{R}^n: \max_{1\le i \le n-1}|x_i| \le \beta(x_n)\}
\end{equation}
is $1$-thick (see \cite[Example~6.3]{tyulenev_thick_sets_R_n}). Furthermore, the sequence of weighted measures $\fm_k := w_k^{\beta}\mathcal{H}^n\lfloor_{G^{\beta}}$ is $1$-regular, where
\begin{equation}
    w_k^{\beta}(x) := w_k^{\beta}(x', x_n) := \begin{cases}
        \beta(x_n)^{1-n}, \quad x_n \in [0, \beta^{-1}(2^{-k})], \\
        2^{k(n-1)}, \quad x_n \ge \beta^{-1}(2^{-k}),\\
        0, \quad x \notin G^{\beta}.
    \end{cases}
\end{equation}
The Lebesgue density theorem. also shows that this sequence is strongly $1$-regular. Thus, our results give an intrinsic description of traces of Besov and Lizorkin--Triebel spaces to the cusp $G^{\beta}$.
\par
\textit{Example 2.} In this example we recall the construction of a strongly $d$-regular sequence of measures on the union of Ahlfors--David regular sets (see \cite[Example~5.2]{tyul}). Let $E_{i}$, $i \in \{1, \ldots, N\}$, be Ahlfors--David $d_i$-regular sets, $d_i \in [0, n]$, and let $E := \bigcup_{i=1}^N E_i$. Then $E$ is $d$-thick, where $d := \min\{d_1, \ldots, d_N\}$, and the sequence $\fm_k := \sum_{i=1}^N 2^{k(d_i-d)}\mathcal{H}^{d_i}\lfloor_{E_i}$, $k\in \mathbb{N}_0$, is strongly $d$-regular.
\par
The next example was considered in \cite{tyul} for Sobolev spaces and provides an example in which the trace norm can be simplified. 
\par
\textit{Example 3.} In this example we consider a particular union of Ahlfors--David regular sets and give a simplified form of the trace norm (see also \cite[Example~11.4]{tyul}). Let $n\ge 2$, let $\underline{Q} = Q_{R}(x)$, and let $\gamma:[0, 1] \to \mathbb{R}^n$ be a rectifiable curve such that $\Gamma = \gamma([0, 1])$ is Ahlfors--David $1$-regular. Assume that $\Gamma \cap \underline{Q} = \{\underline{x}\}$ for some $\underline{x} \in \partial \underline{Q}$. Suppose moreover that there exists $\lambda > 0$ such that $\operatorname{dist}(\gamma(t), \underline{Q}) \ge \lambda \ell(\gamma([0, t]))$ for all $t\in [0, 1]$. We set $E = \underline{Q}\cup \Gamma$. By the preceding example, $E$ is $1$-thick, and the sequence $\fm_k = 2^{k(n-1)}\mathcal{H}^n\lfloor_{\underline{Q}} + \mathcal{H}^1\lfloor_{\Gamma}$, $k \in \mathbb{N}_0$, is strongly $1$-regular.
\par
By the equivalence of the Besov-type functionals, it suffices to consider $\sigma=1$. Fix $p\in (n-1, \infty]$, $q\in (0, \infty]$, and $s \in (\frac{n-1}{p}, 1)$. For each $\phi \in L_{p}(\fm_0)$ and each $k \in \mathbb{N}_0$, we define the \emph{gluing function} by
\begin{equation}
    \operatorname{gl}_{k}(\phi, \underline{x}) = \fint\limits_{\underline{Q}\cap Q_k(\underline{x})} \fint\limits_{\Gamma \cap Q_k(\underline{x})} |\phi(y) - \phi(z)|d\mathcal{H}^1(z) d\mathcal{H}^n(y).
\end{equation}
\par
Let $\underline{k} := \min\{k\in \mathbb{N}: 2^k > 2+\frac{4}{\lambda}\}$. The choice of \(\underline k\) ensures that, for \(k>\underline k+1\), a cube \(Q_{k-1}(x)\) can meet both \(\underline Q\) and \(\Gamma\) only when \(x\) lies within \(O(2^{-k})\) of the junction point \(\underline x\). Thus, away from \(\underline x\), the two components may be treated separately by the usual trace estimates, whereas near \(\underline x\) their interaction is measured by the gluing function.
\par
Fix $k\le \underline{k}+1$. Assume that $x\in \mathbb{R}^n$ satisfies $\widetilde{\mathcal{E}}_{\fm_k}(\phi, Q_k(x)) \neq 0$. Then $Q_k(x) \cap E \neq \emptyset$, and hence, by Corollary~\ref{Ca.relaxed_doubling_property}, we have $\fm_k(2Q_k(x)) \approx 2^{-k}$. Consequently, using Jensen's inequality, Fubini's theorem, and the definition of $\mathcal{\widetilde{E}}_{\fm_k}$, we obtain
\begin{equation}
    \|\widetilde{\mathcal{E}}_{\fm_k}(\phi, Q_k(\cdot))\|_{L_p(\mathbb{R}^n)}^p \lesssim 2^{k}\int\limits_{E}|\phi(y)|^p   \int\limits_{2Q_{k}(y)}dx d\fm_k(y) \lesssim 2^{-k(n-1)}\|\phi\|_{L_p(\fm_k)}^p
\end{equation}
with the usual adjustments when $p=\infty$. As a result,
\begin{equation}
    \label{eq.ball_curve_example1}
    \|\{2^{ks} \|\widetilde{\mathcal{E}}_{\fm_k}(\phi, Q_k(\cdot))\|_{L_p(\mathbb{R}^n)}\}_{k=0}^{\underline{k}+1}\|_{l_q} \lesssim \|\phi\|_{L_p(\fm_0)} \approx \|\phi\|_{L_p(\underline{Q}, \mathcal{H}^n)} + \|\phi\|_{L_p(\Gamma, \mathcal{H}^1)}.
\end{equation}
\par
Now fix $k>\underline{k}+1$ and assume that $x\in \mathbb{R}^n$ satisfies $\widetilde{\mathcal{E}}_{\fm_k}(\phi, Q_k(x)) \neq 0$. If $x\in \mathbb{R}^n \setminus Q_{k-\underline{k}}(\underline{x})$, then the arguments used in Lemma~\ref{Lm.local_approximation_on_regular} and~\ref{Lm.trace_to_n_regular} apply almost verbatim. Indeed, if $Q_{k-1}(x) \cap \underline{Q} \neq \emptyset$, then $Q_{k-1}(x) \cap \Gamma = \emptyset$ and vice versa. Therefore,
\begin{equation}
    \label{eq.ball_curve_example2}
    \begin{split}
    \|\widetilde{\mathcal{E}}_{\fm_k}(\phi, Q_k(\cdot))\|_{L_p(\mathbb{R}^n\setminus Q_{k-\underline{k}}(\underline{x}))}^p \lesssim \|\mathcal{E}_{\mathcal{H}^n\lfloor_{\underline{Q}}}(\phi, Q_{k-4}(\cdot))\|_{L_p(\mathcal{H}^n\lfloor_{\underline{Q}})}^p \\+ 2^{-k(n-1)} \|\mathcal{E}_{\mathcal{H}^1\lfloor_{\Gamma}}(\phi, Q_{k-4}(\cdot))\|_{L_p(\mathcal{H}^1\lfloor_{\Gamma})}^p.
    \end{split}
\end{equation}
Furthermore, if $x \in Q_{k-\underline{k}}(\underline{x})$, then it easily follows from Corollary~\ref{Ca.relaxed_doubling_property} that $\widetilde{\mathcal{E}}_{\fm_k}(\phi, Q_{k}(x)) \lesssim \operatorname{gl}_{k-\underline{k}-1}(\phi, \underline{x})$. Therefore, by Corollary~\ref{Ca.relaxed_doubling_property},
\begin{equation}
    \label{eq.ball_curve_example3}
    \|\widetilde{\mathcal{E}}_{\fm_k}(\phi, Q_k(\cdot))\|_{L_p( Q_{k-\underline{k}}(\underline{x}))}^p \lesssim 2^{-kn}\bigl( \operatorname{gl}_{k-\underline{k}-1}(\phi, \underline{x})\bigr)^p.
\end{equation}
The converse estimate is obtained similarly. Indeed, if $x\in Q_{k-\underline{k}}(\underline{x})$, then Corollary~\ref{Ca.relaxed_doubling_property} and Lemma~\ref{Lm.local_approximation_outside_E} yield $\operatorname{gl}_{k-\underline{k}}(\phi, \underline{x})\lesssim \widetilde{\mathcal{E}}_{\fm_{k-\underline{k}}}(\phi, Q_{k-\underline{k}}(x))$. Consequently,
\begin{equation}
    \label{eq.ball_curve_example4}
    \bigl( \operatorname{gl}_{k-\underline{k}}(\phi, \underline{x})\bigr)^p \lesssim 2^{kn}\|\widetilde{\mathcal{E}}_{\fm_{k-\underline{k}}}(\phi, Q_{k-\underline{k}}(\cdot))\|_{L_p(Q_{k-\underline{k}}(\underline{x}))}^p.
\end{equation}
\par
We set $\mathcal{GL}_q(\phi, \underline{x}):= \|\{2^{k(s-\frac{n}{p})}\operatorname{gl}_k(\phi, \underline{x})\}_{k=1}^{\infty}\|_{\ell_q}$. The exponent \(s-\frac{n}{p}\) reflects the fact that the compatibility condition is imposed at the zero-dimensional junction point.
Combining \eqref{eq.ball_curve_example1}-\eqref{eq.ball_curve_example3}, we obtain
\begin{equation}
    \label{eq.ball_curve_example5}
    \begin{split}
   \|\phi\|_{\mathfrak{B}^{s-(n-1)/p}_{p, q, 1}(E)} \lesssim \|\phi\|_{B^s_{p, q, 1}(\underline{Q})} + \|\phi\|_{B^{s-(n-1)/p}_{p, q, 1}(\Gamma)} + \mathcal{GL}_q(\phi, \underline{x}).
    \end{split}
\end{equation}
On the other hand, if $f\in B^s_{p, q}(\mathbb{R}^n)$ and $\phi = f\big|^{\fm_0}_E$, then by Lemma~\ref{Lm.local_approximation_on_regular} and Lemma~\ref{Lm.trace_to_n_regular},
\begin{equation}
    \label{eq.ball_curve_example6}
    \|\phi\|_{B^s_{p, q, 1}(\underline{Q})} + \|\phi\|_{B^{s-(n-1)/p}_{p, q, 1}(\Gamma)} \lesssim \|f\|_{B^s_{p, q}(\mathbb{R}^n)}.
\end{equation}
Combining \eqref{eq.ball_curve_example4}-\eqref{eq.ball_curve_example6}, we obtain the following result.
\par
\textit{A function $\phi\in L_p(\underline{Q}, \mathcal{H}^n)\cap L_p(\Gamma, \mathcal{H}^1)$ belongs to the trace-space $B^s_{p, q}(\mathbb{R}^n)\big|^{\fm_0}_{E}$ if and only if $\phi\in B^s_{p, q}(\underline{Q}) \cap B^{s-(n-1)/p}_{p, q}(\Gamma)$ and $\mathcal{GL}_q(\phi, \underline{x})<\infty$. Furthermore,
\begin{equation}
    \|\phi\|_{B^s_{p, q}(\mathbb{R}^n)\big|^{\fm_0}_E} \approx \|\phi\|_{B^s_{p, q}(\underline{Q})} + \|\phi\|_{B^{s-(n-1)/p}_{p, q}(\Gamma)} + \mathcal{GL}_q(\phi, \underline{x}).
\end{equation}
Finally, there exists a bounded linear extension operator $\operatorname{Ext}: B^s_{p, q}(\mathbb{R}^n)\big|^{\fm_0}_E \to B^s_{p, q}(\mathbb{R}^n)$.}
\par
The same geometric decomposition applies in the Lizorkin--Triebel case. The essential difference is the order of the \(L_p\)- and \(\ell_q\)-norms. Since the junction cubes are nested, the contribution near \(\underline x\) leads to an \(\ell_p\)-gluing condition rather than an \(\ell_q\)-condition. More specifically, assume that $p \in (n-1, \infty)$, $q\in (1, \infty]$, $s\in (\frac{n-1}{p}, 1)$, and $\phi \in L_p(\fm_0)$. The low frequency terms with $k \le \underline{k}+1$ are absorbed by the $L_p(\fm_0)$-norm of $\phi$. For \(k>\underline k+1\), we split the norm into the contribution from \(Q_{k-\underline k}(\underline x)\) and that from its complement:
\begin{equation}
    \label{eq.ball_curve_example_LT1}
    \begin{split}
    \|\{2^{ks}\widetilde{\mathcal{E}}_{\fm_k}(\phi, Q_k(\cdot))\}_{k=\underline{k}+2}^{\infty}\|_{L_p(\mathbb{R}^n, \ell_q)} \le \|\{2^{ks}\chi_{ Q_{k-\underline{k}}(\underline{x})}(\cdot)\widetilde{\mathcal{E}}_{\fm_k}(\phi, Q_k(\cdot))\}_{k=\underline{k}+2}^{\infty}\|_{L_p(\mathbb{R}^n, \ell_q)}\\+ \|\{2^{ks}\chi_{\mathbb{R}^n\setminus Q_{k-\underline{k}}(\underline{x})}(\cdot)\widetilde{\mathcal{E}}_{\fm_k}(\phi, Q_k(\cdot))\}_{k=\underline{k}+2}^{\infty}\|_{L_p(\mathbb{R}^n, \ell_q)}
    \end{split}
\end{equation}
The second term in the latter estimate is bounded by the $B^{s-(n-1)/p}_{p, p}(\Gamma)$-norm and the $F^s_{p, q}(\underline{Q})$-norm of $\phi$. For the first term, Hardy's inequality together with the pointwise estimates of the local approximation and the gluing function yields
\begin{equation}
    \label{eq.ball_curve_example_LT2}
    \begin{split}
    \|\{2^{ks}\chi_{ Q_{k-\underline{k}}(\underline{x})}(\cdot)\widetilde{\mathcal{E}}_{\fm_k}(\phi, Q_k(\cdot))\}_{k=\underline{k}+2}^{\infty}\|_{L_p(\mathbb{R}^n, \ell_q)} \approx \|\{2^{k(s-\frac{n}{p})}\operatorname{gl}_{k}(\phi, \underline{x})\}_{k=1}^{\infty}\|_{\ell_p}.
    \end{split}
\end{equation}
Thus, we obtain the following intrinsic characterization of the trace-space.
\par
\textit{A function $\phi\in L_p(\underline{Q}, \mathcal{H}^n)\cap L_p(\Gamma, \mathcal{H}^1)$ belongs to the trace-space $F^s_{p, q}(\mathbb{R}^n)\big|^{\fm_0}_{E}$ if and only if $\phi\in F^s_{p, q}(\underline{Q}) \cap B^{s-(n-1)/p}_{p, p}(\Gamma)$ and $\mathcal{GL}_p(\phi, \underline{x})<\infty$. Furthermore,
\begin{equation}
    \|\phi\|_{F^s_{p, q}(\mathbb{R}^n)\big|^{\fm_0}_E} \approx \|\phi\|_{F^s_{p, q}(\underline{Q})} + \|\phi\|_{B^{s-(n-1)/p}_{p, p}(\Gamma)} + \mathcal{GL}_p(\phi, \underline{x}).
\end{equation}
Finally, there exists a bounded linear extension operator $\operatorname{Ext}: F^s_{p, q}(\mathbb{R}^n)\big|^{\fm_0}_E \to F^s_{p, q}(\mathbb{R}^n)$.}
\par

\renewcommand{\bibname}{References}
\bibliographystyle{plainurl}
\bibliography{refs}
\end{document}